\documentclass[11pt]{amsart}
\usepackage{amsthm}
\usepackage{amsmath}
\usepackage{amsxtra}
\usepackage{amscd}
\usepackage{amssymb}
\usepackage{color}
\usepackage{mathtools}
\usepackage{xspace}
\usepackage{amsfonts}
\usepackage{mathrsfs}
\usepackage{todonotes}
\usepackage{enumerate}
\usepackage{multirow}
\usepackage{tikz}
\usepackage{hyperref}
\usepackage{cancel}
\usepackage{tikz-cd}

\usepackage[all]{xy}
\usepackage{amscd}
\usepackage{caption}
\usepackage{epsfig}
\usepackage{graphicx,transparent}
\usepackage{here}
\usepackage{hyperref}
\usepackage{setspace}
\usepackage{wasysym}
\usepackage{xparse}
\usepackage{manfnt}
\usepackage{tabto}
\usepackage{imakeidx}

\RequirePackage{filecontents}

\usepackage{listings}
\usepackage{textcomp}
\usepackage{color}
\usepackage{microtype}
\makeatletter\g@addto@macro\@verbatim{\microtypesetup{activate=false}}\makeatother%

\definecolor{mygreen}{rgb}{0,0.6,0}
\definecolor{mygray}{rgb}{0.5,0.5,0.5}
\definecolor{mymauve}{rgb}{0.58,0,0.82}
\definecolor{backcolour}{rgb}{0.95,0.95,0.92}

\usepackage[
            backend=biber,
            style=alphabetic,
            maxalphanames=6,
            maxnames=7
            ]{biblatex}
\bibliography{reference}

\numberwithin{equation}{section}

\theoremstyle{plain}
\newtheorem{thm}{Theorem}[section]
\newtheorem{lm}[thm]{Lemma}
\newtheorem{prop}[thm]{Proposition}
\newtheorem{cor}[thm]{Corollary}

\newtheorem{conj}[thm]{Conjecture}

\theoremstyle{definition}
\newtheorem{dfx}[thm]{Definition}
\newtheorem{alg}[thm]{Algorithm}
\newtheorem{assu}[thm]{Assumption}
\newenvironment{df}
{\pushQED{\qed}\dfx}
{\popQED\enddfx}

\newtheorem{rem}[thm]{Remark}
\newtheorem{exa}[thm]{Example}

\def\={\;=\;}  \def\+{\,+\,}

\def\be{\begin{equation}}   \def\ee{\end{equation}}     \def\bes{\begin{equation*}}    \def\ees{\end{equation*}}
\def\ba{\be\begin{aligned}} \def\ea{\end{aligned}\ee}   \def\bas{\bes\begin{aligned}}  \def\eas{\end{aligned}\ees}

\DeclareTextFontCommand{\emph}{\bfseries\em}

\DeclareMathOperator{\ord}{ord}
\DeclareMathOperator{\Real}{Re}
\DeclareMathOperator{\res}{res}
\DeclareMathOperator{\ind}{ind}
\DeclareMathOperator{\divisor}{div}

\DeclareMathOperator{\CH}{CH}
\DeclareMathOperator{\RH}{RH}

\DeclareMathOperator{\DR}{DR}
\DeclareMathOperator{\LG}{LG}

\DeclareMathOperator{\Aut}{Aut}
\DeclareMathOperator{\spin}{spin}
\DeclareMathOperator{\Cont}{Cont}
\DeclareMathOperator{\Mod}{Mod}
\DeclareMathOperator{\St}{St}
\DeclareMathOperator{\SL}{SL}

\newcommand{\norm}[1]{\left\lVert#1\right\rVert}

\newcommand{\pmoxrmoduli}{\mathbb{P}\Xi\overline{\mathcal{M}} ^\mathfrak{R}_{\boldsymbol{g},\boldsymbol{n}}(\boldsymbol{\mu})}

\title[Computing spin stratum classes]{An algorithm to compute the fundamental classes of spin components of strata of differentials}
\author{Yiu Man Wong }
\address{Institut f\"ur Mathematik, Goethe-Universit\"at Frankfurt,
Robert-Mayer-Str. 6-8,
60325 Frankfurt am Main, Germany}
\email{wong@math.uni-frankfurt.de}
\thanks{Research of the author is supported
by the DFG-project MO 1884/2-1 and the Collaborative Research Centre
TRR 326 ``Geometry and Arithmetic of Uniformized Structures''.}

\begin{document}

\maketitle
\begin{abstract}
    We construct an algorithm for computing the cycle classes of the spin components of a stratum of differentials in the moduli space of stable curves $\overline{\mathcal{M}}_{g,n}$. In addition, we implement it within the \texttt{Sage} package \texttt{admcycles}. Our main strategy is to reconstruct these cycles by their restrictions to the boundary of $\overline{\mathcal{M}}_{g,n}$ via clutching maps. These restrictions can be computed recursively by smaller dimensional spin classes and determine the original class via a certain system of linear equations. To study the spin parities on the boundary of a stratum of differentials of even type, we make use of the moduli space of multi-scale differentials introduced in \cite{bainbridge2019moduli}.
    
    As an application of our algorithm, one can verify a conjecture on spin double ramification cycles stated in \cite{costantini2021integrals} in many examples, by using the results computed by our algorithm.
\end{abstract}

\tableofcontents
\section{\text{Introduction}}
\label{sec:Intro}

Let $\mathbb{P}\Omega\mathcal{M}_{g,n}(\mu)$ be a (projectivized) moduli space of differentials with labelled singularities of orders prescribed by an integer partition $\mu=(m_1,...,m_n)\in \mathbb{Z}^n$ of $2g-2$. For brevity, we will call the space $\mathbb{P}\Omega\mathcal{M}_{g,n}(\mu)$ a \emph{stratum of differentials} and the partition $\mu$ a \emph{signature}. The space $\mathbb{P}\Omega\mathcal{M}_{g,n}(\mu)$ parametrizes pairs $\big( (C,p_1,...,p_n),[\omega]\big)$, where $(C,p_1,...,p_n)$ is a genus $g$ Riemann surface with $n$ marked points while $[\omega]$ is the scaling equivalence class of a  differential $\omega$ whose divisor is $\sum_{i=1}^n m_ip_i$. A (projectivized) stratum $\mathbb{P}\Omega\mathcal{M}_{g,n}(\mu)$ can be embedded into $\mathcal{M}_{g,n}$ and we denote the image and its closure in $\overline{\mathcal{M}}_{g,n}$ by~$\mathcal{H}_g(\mu)$ resp. $\overline{\mathcal{H}}_g(\mu)$. The space $\overline{\mathcal{H}}_g(\mu)$ is called the Deligne-Mumford compactification of a (projectivized) stratum of differentials. Since $\overline{\mathcal{H}}_g(\mu)$ is a closed substack of $\overline{\mathcal{M}}_{g,n}$, it induces a cycle class $[\overline{\mathcal{H}}_g(\mu)]\in \mathrm{H}^*(\overline{\mathcal{M}}_{g,n},\mathbb{Q})$. The definitions of strata of differentials and the Deligne-Mumford compactification can be extended to $k$-differentials. For a given signature $\mu=(m_1,...,m_n)$, where $\sum_i m_i=k(2g-2)$, we denote the embedding of a (projectiviezed) stratum of $k$-differentials of type $\mu$ into $\overline{\mathcal{M}}_{g,n}$ and its closure in $\overline{\mathcal{M}}_{g,n}$ by $\mathcal{H}_g^k(\mu)$ resp. $\overline{\mathcal{H}}_g^k(\mu)$. The computation of the cycle classes $[\overline{\mathcal{H}}_g^k(\mu)]$ in $\CH^*(\overline{\mathcal{M}}_{g,n},\mathbb{Q})$ or $\mathrm{H}^*(\overline{\mathcal{M}},\mathbb{Q})$ is a problem that has been recently solved (cf. \cite{sauvaget2019cohomology}, \cite{schmitt2016dimension}, \cite{holmes2021infinitesimal}, \cite{bae2020pixton}).

The complete classification of connected components of a stratum of $1$-differentials can be found in \cite{kontsevich2003connected} and \cite{boissy2015connected}. On the other hand, the problem of classifying the connected components for strata of $k$-differentials is still not completely solved (cf. \cite{chen2021towards}). An important invariant of a stratum of $k$-differentials is the \emph{spin parity} which will be recalled in Section~\ref{sec:spin}. In general, if $k$ is odd and $\mu$ is of \emph{even type}, i.e. the integers $m_1,..,m_n$ are all even, then the connected components of a stratum of $k$-differentials of type $\mu$ can be partitioned into two non-empty sets according to spin parities. Throughout this paper, if there is an object $O$ that admits spin parities, then the spin components of it will be denoted by $O^+$ resp. $O^-$.  By taking into account the spin components of a stratum of differentials, one can split the computation of algebro-geometric invariants of a stratum into two components according to the spin parity. For instance, the formulae of Masur-Veech volumes of the spin components of a stratum have been derived in \cite{chen2020masur}; the formulae of the integral of $\psi$-classes on the Deligne-Mumford compactification of spin components of a stratum of $k$-differentials are proved in \cite{costantini2021integrals}. Despite these pieces of information about the spin components, the problem of computing the cycle classes $[\overline{\mathcal{H}}_g^k(\mu)^+]$ resp. $[\overline{\mathcal{H}}_g^k(\mu)^-]$ remains open. These classes are important because they allow us to do many computations in the intersection theory of the spin components by reducing to intersections on $\overline{\mathcal{M}}_{g,n}$.

\par For $k=1$, the stratum $\overline{\mathcal{H}}_{g}(\mu)$ has pure complex dimension $g$ if $\mu$ has any negative entry (so that $\overline{\mathcal{H}}_{g}(\mu)$ parameterizes meromorphic differentials) and of $g-1$ otherwise (for the space parameterizing holomorphic differentials). In this text, we construct an algorithm to split the fundamental class of a stratum of $1$-differentials of even type $[\overline{\mathcal{H}}_{g}(\mu)]\in H^{j}(\overline{\mathcal{M}}_{g,n};\mathbb{Q})$, where we let $j=2g$ resp. $j=2g-2$ if we are working on strata of meromorphic resp. holomorphic differentials, into the spin components $[ \overline{\mathcal{H}}_{g}(\mu)^{-}]$ and $[ \overline{\mathcal{H}}_{g}(\mu)^{+}]$. For short, we will call the fundamental class of a stratum of differentials $[\overline{\mathcal{H}}_g(\mu)]$ a \emph{stratum class}. To do the recursive computation, it will be more convenient if we instead compute the \emph{spin stratum class}, which is defined as: \begin{align*}
    [\overline{\mathcal{H}}_{g}(\mu)]^{\spin}:=[\overline{\mathcal{H}}_{g}(\mu)^{+}]-[\overline{\mathcal{H}}_{g}(\mu)^{-}].
\end{align*} 
Note that since there already exist recursive algorithms to compute the stratum classes $[\overline{\mathcal{H}}_{g}(\mu)]$, the spin stratum class $[\overline{\mathcal{H}}_{g}(\mu)]^{\spin}$ will determine $[\overline{\mathcal{H}}_{g}(\mu)^{-}]$ and $[\overline{\mathcal{H}}_{g}(\mu)^{+}]$. 

\begin{rem}
More generally, any (algebraic/cohomological) class on a stratum of $k$-differentials which admits a spin structure can be split into a sum of two summands corresponding to the even and odd spin parities respectively. Thus one can also define a linear map:
\begin{align*}
    \bullet^{\spin}: \mathrm{H}^*(&\mathbb{P}\Xi\overline{\mathcal{M}}_{g,n}(\mu))\longrightarrow \mathrm{H}^*(\mathbb{P}\Xi\overline{\mathcal{M}}_{g,n}(\mu))\\ & w=w^++w^-\mapsto w^+-w^-
\end{align*}
In our text, given a class $w$, we denote its spin variant by $w^{\spin}$. For instance, the spin stratum class of $k$-differentials will be denoted by $[\overline{\mathcal{H}}_g^k(\mu)]^{\spin}$; the tautological classes $\xi^{\spin}$ and $\psi^{\spin}$ on strata of differentials will appear in Section~\ref{sec:resolve_res}.
\end{rem}

\subsection{The outline of our strategy} 
Given a dual graph $\Gamma$ of some stable curve of genus $g$ and $n$ marked points, there is a natural clutching morphism 
\begin{align*}
    \xi_\Gamma:\overline{\mathcal{M}}_\Gamma:=\prod_{v\in V(\Gamma)}\overline{\mathcal{M}}_{g_v,n_v}\longrightarrow \overline{\mathcal{M}}_{g,n}.
\end{align*}
The strategy of our recursive algorithm is to compute $[ \overline{\mathcal{H}}_{g}(\mu)]^{\spin}$ by solving a system of linear equations from clutching pullbacks, which reduce the problem to computing the spin stratum classes $[\overline{\mathcal{H}}_{g'}(\mu')]^{\spin}$ of smaller genera or fewer singularities. By proving that $\mathrm{H}^{2g-2}(\overline{\mathcal{M}}_{g,n};\mathbb{Q})=0$ in Theorem~\ref{thm:vanish}, we can strengthen a result of \cite{arbarello1998calculating} (Proposition~\ref{prop:arbarello}), which will guarantee the injectivity of the direct sum of clutching pullbacks with respect to one-edge stable graphs $\Gamma$
\begin{align}\label{eq:clut}
       \oplus_\Gamma\xi_\Gamma^*:\mathrm{H}^{j}(\overline{\mathcal{M}}_{g,n};\mathbb{Q})\longrightarrow \bigoplus_\Gamma \mathrm{H}^{j}(\overline{\mathcal{M}}_\Gamma;\mathbb{Q}), 
\end{align}
for $j$ within some appropriate range. By such a result, one can actually already construct a theoretical algorithm to compute the spin stratum classes. However, in practice, the cohomology of moduli spaces of stable curves $\overline{\mathcal{M}}_{g,n}$, for all $g,n$, is not entirely known. Thus, it is not possible to directly implement the strategy we mentioned above. It will only be feasible if we restrict our computation to the tautological ring $\RH^*(\overline{\mathcal{M}}_{g,n})$, on which we have a better understanding of the generators and relations. The drawback is that if a spin stratum class is not tautological, then our algorithm implemented in \texttt{admcycles} will raise errors. Thus, the first assumption we need to make for our implementation of the algorithm is that the spin stratum classes are tautological.

A bottleneck for programming the aforementioned strategy is that the clutching pullback with respect to the self-loop graph will not reduce the computation complexity (see Section~\ref{sec:non_comp}). However, we observed that, in many cases where the number of marked points $n$ is larger equal to $3$, the injectivity of \eqref{eq:clut} holds when we restrict the clutching pullbacks to the tautological ring $\RH^*(\overline{\mathcal{M}}_{g,n})$ even without including the non-compact type graph. This means that in the recursive computation of a spin stratum class, we only need to perform the clutching pullback with respect to the self-loop graph at most once. This ensures that the runtime of our algorithm will not be unreasonably long.

However, we are not able to prove that the phenomenon above is always true in this paper. Hence, we rephrase this into an assumption of \emph{the sufficiency of compact type clutching} (Assumption~\ref{prop:assum}), by which we mean that the direct sum of clutching pullbacks, restricting to $\RH^*(\overline{\mathcal{M}}_{g,n})$, with respect to one-edge graphs of compact type is injective for $g\geq 1, n\geq 3$ and for codimension $j$ which is not too large.

 Our implemented algorithm for computing spin stratum classes in this paper will be based on two assumptions: First, the spin stratum classes are tautological; second, the assumption of the sufficiency of compact type clutching. For particular $g,n$ and $j$, it is known that the tautological classes generate the cohomology groups:
 
 \begin{itemize}
     \item $g=0$ and for any $n$ and $j$ (cf. \cite{keel1992intersection});
     \item $g=1$ and $n<11$, for any $j$ (cf. \cite{graber2003constructions});
     \item $g=2$ and $n<20$, and for any even number $j$ (cf. \cite{petersen2016tautological});
     \item $g=3$ and $n=0,1,...,8$ (cf. \cite{bergstrom2008cohomology}, \cite{getzler1999hodge}, \cite{getzler1998topological}, \cite{canning2022chow});
     \item $g=4$ and $n=0,1,...,6$ (cf. \cite{canning2022chow});
     \item and more cf. \cite{canning2022chow}.
 \end{itemize}
 Thus, in the cases listed above, the first assumption is automatically true. On the other hand, the second assumption has been checked by the \texttt{Sage} package \texttt{admcycles} on the tautological ring and it is true for $(g,n)$ equal to
\begin{align}\label{eq:range1}
    (1,3),(1,4),(1,5),(1,6),(2,3),(2,4),(2,5),(3,3),(3,4),(4,3).
\end{align}
To summarise, the assumptions we made are automatically true for the values of $g,n$ listed above. We have used the implemented algorithm to compute the spin stratum classes for plenty of signatures where $(g,n)$ is equal to
\begin{align}\label{eq:range2}
    (2,1),(2,2),(2,3),(2,4),(3,1),(3,2),(3,3),(4,1),(4,2),(4,3).
\end{align}
The results by the implemented algorithm for $(g,n)$ in the list above are guaranteed to be correct.

Lastly, we explain how we compute each individual clutching pullback. Indeed, we apply the Clutching Pullback Formula \eqref{eq:excess_p} of $[\overline{\mathcal{H}}_g(\mu)]^{\spin}$. The formula is derived by using the spin structure on the moduli space of multi-scale differentials $\mathbb{P}\Xi\overline{\mathcal{M}}_{g,n}(\mu)$ (the definitions will be recalled in Section~\ref{sec:multi-scale}), which defines a compactification of the stratum of differentials $\mathbb{P}\Omega\mathcal{M}_{g,n}(\mu)$. The compactification $\mathbb{P}\Omega\mathcal{M}_{g,n}(\mu)\subset\mathbb{P}\Xi\overline{\mathcal{M}}_{g,n}(\mu)$ has normal crossing boundary divisors and a proper morphism $p:\mathbb{P}\Xi\overline{\mathcal{M}}_{g,n}(\mu)\longrightarrow\overline{\mathcal{M}}_{g,n}$ by sending a multi-scale differential to its underlying stable curve. The advantage to working with the compactification $\mathbb{P}\Xi\overline{\mathcal{M}}_{g,n}(\mu)$ is that the set of connected components of $\mathbb{P}\Xi\overline{\mathcal{M}}_{g,n}(\mu)$ is in bijection with the set of connected components of $\mathbb{P}\Omega\mathcal{M}_{g,n}(\mu)$, whereas in the Deligne-Mumford compactification the spin components can have intersections on the boundary. This implies that the closures of different spin components are still disjoint, thus the spin parity on a multi-scale differential of even type can be defined as the spin parity of the differentials lying in the neighbourhood. As a result, $\mathbb{P}\Xi\overline{\mathcal{M}}_{g,n}(\mu)$ has spin components and 
\begin{align*}
    p_*\big([\mathbb{P}\Xi\overline{\mathcal{M}}_{g,n}(\mu)^+]\big)&= [ \overline{\mathcal{H}}_{g}(\mu)^+]\\
    p_*\big([\mathbb{P}\Xi\overline{\mathcal{M}}_{g,n}(\mu)^-]\big)&= [ \overline{\mathcal{H}}_{g}(\mu)^-]\\
    p_*\big([\mathbb{P}\Xi\overline{\mathcal{M}}_{g,n}(\mu)]^{\spin}\big)&=[ \overline{\mathcal{H}}_{g}(\mu)]^{\spin}.
\end{align*}

Note that deriving the Clutching Pullback Formula~\eqref{eq:excess_p} requires an understanding of how the spin parity behaves on the boundary strata of $\mathbb{P}\Xi\overline{\mathcal{M}}_{g,n}(\mu)$. The spin parity of a multi-scale differential lying on the boundary can be computed in the following way:
\begin{itemize}
    \item First, we undegenerate the multi-scale differential to a usual differential which can be realized by a flat surface
    \item Then, we compute the spin parity by the explicit Formula~\eqref{eq:parity} for the Arf invariant.
\end{itemize}

\subsection{Main results}

\par Our first observation on the spin parity of a multi-scale differential is the following proposition. We briefly define the objects involved in the following proposition (readers are referred to Section~\ref{sec:Background} for the details). 
\begin{itemize}
    \item An \emph{enhanced level graph} is a graph whose edges are weighted by positive integers (which we call enhancements) and there is a full order on the set of vertices (which can be realized by placing the vertices on different levels).
    \item A \emph{twisted differential} of type $\mu$ on a stable curve $(X;z_1,...,z_n)$ is a collection of non-zero meromorphic differentials on the irreducible components of $X$ such that their order of poles and zeros at $z_1,...,z_n$ are prescribed by $\mu$. A twisted differential is \emph{compatible} with an enhanced level graph $\Delta$ if $\Delta$ is the dual graph of $(X;z_1,...,z_n)$, and the orders of poles and zeros at the nodal points are prescribed by the enhancements (altered by $\pm 1$).
    \item A (global) \emph{prong-matching} is a collection of cyclic-order-reversing bijections which prescribe the matchings of horizontal geodesics (with respect to the meromorphic differentials) converging to the nodal points from two components.
\end{itemize}
 
\begin{prop}
\label{prop:1.1}
Let $\boldsymbol{\eta}$ be a twisted differential of type $\mu$ and compatible with an enhanced level graph $\Delta$. Then we have the following classification of the spin parity of the multi-scale differentials whose underlying twisted differential is $\boldsymbol{\eta}$:
\begin{enumerate}
    \item If $\Delta$ has any edge of even enhancement (non-zero), then half of the (global) prong-matchings on $\boldsymbol{\eta}$ will give us a multi-scale differential of even spin and half of the prong-matchings will give us a welded surface of odd spin.
   \item  If all the edges in $\Delta$ are of odd enhancements, then the spin parity of any multi-scale differential whose underlying twisted differential is $\boldsymbol{\eta}$ is just the sum of the spin parities of the components of $\boldsymbol{\eta}$.
\end{enumerate}
\end{prop}

The boundary divisors $D_\Delta$ in $\mathbb{P}\Xi\overline{\mathcal{M}}_{g,n}(\mu)$ are parametrized by enhanced level graphs $\Delta$ such that either the number of levels is $2$ and there is no \emph{horizontal edge} (an edge connecting two vertices on the same level), or there is exactly one level and one horizontal edge. We denote the set of two-level graphs without horizontal edges of a given stratum by $\LG_1(\mu)$. For short, we call such a level graph $\Delta$ a \emph{vertical two-level graph}. The degree of the map $p|_{D_\Delta}$ is equal to the number of prong matching equivalence classes $\frac{\prod_{e\in E(\Delta}\kappa_e}{\ell_\Delta}$, where $\ell_\Delta$ is the l.c.m. of the enhancements. In addition, the image $p(D_\Delta)$ will be equal to the image of
\begin{align*}
    \xi_{\Delta}:p^\top(\mathbb{P}\Xi\overline{\mathcal{M}}_{\Delta,\top})\times p^\bot(\mathbb{P}\Xi\overline{\mathcal{M}}_{\Delta,\bot})\subset \overline{\mathcal{M}}_\Delta\longrightarrow \overline{\mathcal{M}}_{g,n},
\end{align*}
where $\mathbb{P}\Xi\overline{\mathcal{M}}_{\Delta,\top}$ and $\mathbb{P}\Xi\overline{\mathcal{M}}_{\Delta,\bot}$ are the level strata corresponding to the top and bottom level of $\Delta$.
 We have the following corollary to Proposition~\ref{prop:1.1}. We define $[D_\Delta]^{\spin}=[D_\Delta^+]- [D_\Delta^-]\in \CH^1(\mathbb{P}\Xi\overline{\mathcal{M}}_{g,n}(\mu))$.
 
 \begin{cor}
\label{cor:1.3}
Let $p:\mathbb{P}\Xi\overline{\mathcal{M}}_{g,n}(\mu)\longrightarrow\overline{\mathcal{M}}_{g,n}$ be the proper projection morphism. For $\Delta\in \LG_1(\mu)$, if all edges of $\Delta$ are of odd enhancements, then
\begin{align*}
   &p_*\big([D_\Delta]^{\spin}\big)\in \CH^*(\overline{\mathcal{M}}_{g,n})
   \\=&\frac{\prod_{e\in E(\Delta)}\kappa_e}{|\Aut(\Delta)|\ell_\Delta}\xi_{\Delta*}\bigg(\pi_\top^*([\mathbb{P}\Xi\overline{\mathcal{M}}_{\Delta,\top}]^{\spin})\cdot\pi_\bot^*([\mathbb{P}\Xi\overline{\mathcal{M}}_{\Delta,\bot}]^{\spin}) \bigg)
\end{align*}
 Otherwise, if there is an edge of even enhancement, then we have \[p_*[D_\Delta]^{\spin}=0.\]
\end{cor}
Although Corollary~\ref{cor:1.3} will not be applied directly in the computation of the clutching pullback a spin stratum class, it will be used combining with the residue resolving (Proposition~\ref{prop:res_resl}). Furthermore, the principle of Corollary~\ref{cor:1.3} is used to derive the Clutching Pullback Formula of a spin stratum class. More precisely, if $\Delta$ has any edge of even enhancement, then the term corresponding to $\Delta$ in the Clutching Pullback Formula will be zero. Otherwise, the term corresponding to $\Delta$ will be just a product of the spin stratum classes on the top level and that on the bottom level.

\begin{thm}\label{prop:main_result1}
The spin stratum classes $[ \overline{\mathcal{H}}_{g}(\mu)]^{\spin}\in \mathrm{H}^{*}(\overline{\mathcal{M}}_{g,n};\mathbb{Q})$ are recursively computable.
\end{thm}

Theorem~\ref{prop:main_result1} only guarantees a solution in $\mathrm{H}^*(\overline{\mathcal{M}}_{g,n},\mathbb{Q})$. In practice, we work with the \texttt{Sage} package \texttt{admcycles}, by which we can only compute the solution in the tautological ring. Since \texttt{admcycles} uses the generalised Faber-Zagier relations (cf. \cite{pixton2012conjectural}, \cite{pandharipande2015relations},\cite{janda2017relations}) to compute for a basis of the tautological ring. As a result, the correctness of our implemented algorithm will depend on the assumption that the generalised Faber-Zagier relations span the complete set of relations of the decorated strata classes (which serve as the additive generators) in the tautological ring. We have the following conditional result:

\begin{thm}\label{prop:cond_result}
Assume that the generalised Faber-Zagier relations span the whole set of relations of the decorated strata classes in the tautological ring. If the spin stratum classes are tautological and the assumption of the sufficiency of compact type clutching (Assumption~\ref{prop:assum}) holds, for all $g\leq g_0$ and $n\leq n_0$, then the algorithm we implemented in \texttt{admcycles} will yield the correct spin stratum classes $[ \overline{\mathcal{H}}_{g}(\mu)]^{\spin}$ such that $g\leq g_0$ and $n\leq n_0$.
\end{thm}

By the implemented algorithm, we have computed the spin stratum classes of plenty of different signatures for $(g,n)$ within the range \eqref{eq:range2}, where the tautological ring is just the cohomology ring and the assumption of Faber-Zagier relations can be verified by Poincar\'e duality. Thus, in such range, our results are just the exact results. 

\subsection{Applications}

First, we want to give an overview of the strategies to compute cycle classes $[\overline{\mathcal{H}}^k_g(\mu)]$ in $\CH^*(\overline{\mathcal{M}}_{g,n})$ or $\mathrm{H}^*(\overline{\mathcal{M}}_{g,n})$. In \cite{sauvaget2019cohomology}, Sauvaget has constructed an algorithm to compute the stratum class of $1$-differentials $[\overline{\mathcal{H}}_g(\mu)]$ in $\mathrm{H}^*(\overline{\mathcal{M}}_{g,n})$ and he shows that the stratum classes are tautological. Later in \cite{farkas2018moduli}, Farkas and Pandhahripande conjectured that the sum of the stratum class of $1$-differentials and the boundary terms is equal to the twisted double ramification cycle (the definition will be recalled in Section~\ref{sec:multi-scale}). In \cite{schmitt2016dimension}, Schmitt extended this conjecture to $k$-differentials and gave a recursion to extract the stratum classes $[\overline{\mathcal{H}}^k_g(\mu)]$ from the conjectured equations. This conjecture has been finally proved by \cite{holmes2021infinitesimal} and \cite{bae2020pixton}. 

In \cite{costantini2021integrals}, Constantini, Sauvaget and Schmitt have constructed the spin variant of the twisted double ramification cycles. In addition, they conjectured an equation that is similar to the one conjectured by Farkas and Pandharipande in \cite{farkas2018moduli}, stating that the sum of the spin stratum class and some boundary terms is equal to the spin double ramification cycle. (The exact formulation will be recalled in Conjecture~\ref{conj:spin_pixton} in Section~\ref{sec:spin_DR_conj}.) Within the range in List~\eqref{eq:range1}, the spin stratum classes we computed agree with Conjecture~\ref{conj:spin_pixton}. Thus, the algorithm of our paper can help verify the conjecture in the range where computer power is sufficient to perform the computation.

Note that all the spin stratum classes of holomorphic strata of $1$-differentials are computed for $g\leq 4$, and they lie in the tautological ring (see Appendix~\ref{appendix:results}). Hence, Conjecture~\ref{conj:spin_pixton} will enable us to recursively compute the spin stratum classes $[\mathcal{H}^k_g(\mu)]^{\spin}$ by tautological classes. Hence, if we assume the conjecture about the spin stratum classes (Conjecture~\ref{conj:spin_pixton}) to be true, then we have the following proposition:
\begin{prop}
Let $k$ be an odd positive integer and $\mu$ be a signature of even type. If Conjecture~\ref{conj:spin_pixton} is true, then the spin stratum classes $[\mathcal{H}^k_g(\mu)]^{\spin}$ are tautologcial for $g\leq 4$.
\end{prop}

\paragraph{\textbf{Acknowledgement}}
The authors acknowledge support by the DFG under Grant MO1884/2-1
and the Collaborative Research Centre TRR 326 Geometry and
{\it Arithmetic of Uniformized Structures}, project number 444845124. In addition, the author thanks the advice of Martin Möller and the comments by Johannes Horn, Johannes Schmitt, Adrien Sauvaget and Matteo Constantini.

\section{Background}
\label{sec:Background}

In this section, we will first define the spin parity of a flat surface (or equivalently a differential on a compact Riemann surface). Then we will introduce multi-scale differentials and their associated welded surfaces. Furthermore, we will explain how to extend the definition of spin parity to a welded surface and the plumbed surface associated to a multi-scale differential. In the end, we will explain why the spin parity of a welded surface is equal to that of a plumbed surface.

\subsection{Spin parities}\label{sec:spin} In this subsection, we introduce the spin parity of a flat surface. First, we start with some terminology on a topological surface, then we will dive into the spin parity on a flat surface. By abusing symbols, we will denote a simple loop on a topological surface $X$ and its corresponding cycle in $\mathrm{H}_1(X,\mathbb{Z})$ both by the same letter. 

The \emph{algebraic intersection number} $i(\bullet,\bullet)$ induces a \emph{symplectic form} on $\mathrm{H}_1(X,\mathbb{Z})$, where $X$ is a compact orientable surface. Two cycles $\alpha,\beta$ form a \emph{symplectic pair} if their intersection number $i(\alpha,\beta)$ is $\pm 1$. Let the genus of $X$ be $g$. Then a generating set $\{\gamma_1,...,\gamma_{2g}\}$ of $\mathrm{H}_1(X,\mathbb{Z})$ is a \emph{symplectic basis} if it can be partitioned into symplectic pairs such that any two cycles from different symplectic pairs have zero intersection number.  

Let $(X,\omega)$ be a flat surface induced by a differential $\omega$ on the compact Riemann surface $X$ and $\gamma:[0,1]\longrightarrow (X,\omega)$ be a smooth simple curve not passing through the conical singularities. On a flat surface, parallel translations along a smooth simple loop (avoiding the conical singularities) will preserve the direction of a vector. This implies that we can define the north direction at each point (except for the conical singularities) simultaneously and define a continuous map 
\begin{align*}
 T_{\gamma}: [0,1]\longrightarrow S^1 \quad
 \theta \mapsto \dot{\gamma}(\theta)/\norm{\dot{\gamma}(\theta)}    
\end{align*}
There is a unique lift $\widetilde{T}_\gamma:[0,1]\longrightarrow \mathbb{R}$ if we fix the starting point on $\mathbb{R}$ to be the origin. We define the \emph{turning number} or \emph{index} $\ind(\gamma)$ to be $\widetilde{T}_\gamma(1)$. If $\omega$ is a differential of even type, then $\ind(\gamma) \pmod 2$ will be invariant under the deformation of $\gamma$. Let $\{a_1,...,a_g,b_1,...,b_g\}$ be a symplectic basis of $\mathrm{H}_1(X,\mathbb{Z})$. Then the \emph{parity of spin} of a flat surface $(X,\omega)$ is defined as:
 
 \begin{equation}\label{eq:parity}
      \phi((X,\omega))= \sum_{i=1}^{g}(\ind(a_i)+1)(\ind(b_i)+1) \pmod 2  
 \end{equation}
It has been shown in \cite{johnson1980spin} that the parity of spin does not depend on the choice of a symplectic basis and is a locally constant function on a stratum of differentials of even type. We have now defined the parity of spin on a compact flat surface. Actually, we can extend the definition to the case that the differential $\omega$ of the pair $(X,\omega)$ is meromorphic.
Let $P$ be the set of poles of the meromorphic differential $\omega$, then one can realize $X\setminus P$ by a noncompact orientable surface carrying a flat metric with conical singularities. We call such a flat realization of a meromorphic differential a \emph{noncompact flat surface}. A symplectic basis on $X$ can be realised by smooth simple closed loops on $X\setminus P$ which circumvent the conical singularities. By observation, one can see that if $\omega$ is of even type, the turning numbers $\pmod 2$ are independent of the choice of realizations. Hence, the definition of spin parities in~(\ref{eq:parity}) can be also applied to noncompact flat surfaces associated to meromorphic differentials of even type (cf. \cite{boissy2015connected}). 

\begin{rem}
Note that there is another equivalent definition of spin parities from the perspective of theta characteristics. A \emph{theta characteristic} is a line bundle $L$ on a Riemann surface $X$ such that $L^{\otimes 2}\simeq \omega_X$, where $\omega_X$ is the canonical sheaf on $X$. Given a flat surface $(X,\omega)$, where $\omega$ is of even type, the sheaf $\mathcal{O}_X(\frac{1}{2}\divisor(\omega))$ is a theta characteristic. The rank of the group of global sections $h^0(X,\mathcal{O}_X(\frac{1}{2}\divisor(\omega))) \pmod 2$ coincides with the parity of spin we defined above (cf. \cite{johnson1980spin}).
\end{rem}

\begin{rem}
There are two different ways to define the spin parity of a $k$-differential of even type $\eta$ on a compact Riemann surface $X$. 

The first definition can be applied to any $k\in \mathbb{N}$: one can construct a \emph{canonical cover} $\phi:\hat{X}\longrightarrow X$ (cf. \cite{chen2021towards}) such that $\phi^*\eta=\omega^{\otimes k}$ where $\omega$ is a $1$-differential on $\hat{X}$; the spin parity of $\eta$ is then defined to be that of $\omega$. 

The second definition can be applied to any $k\in \mathbb{N}\setminus 2\mathbb{N}$: one can define a theta characteristic
\[L=\omega_X^{\otimes (-k+1)/2}\otimes \mathcal{O}_X(\frac{1}{2}\divisor(\eta)); \]
the spin parity of $\eta$ is defined to be $h^0(X,L) \pmod 2$. Notice that for $k>1$, these two definitions do not coincide with each other. For example, if $X$ is a genus $0$ curve, then the spin parity of a $k$-differential $\eta$ with respect to the first definition may be odd while the spin parity of $\eta$ with respect to the second definition is always even. In this text, when we mention the spin components of a stratum of $k$-differentials $\mathcal{H}^k_g(\mu)\subset \overline{\mathcal{M}}_{g,n}$, the spin parity refers to the second definition.
\end{rem}

\subsection{Multi-scale differentials and prong-matchings}\label{sec:multi-scale} We will in this subsection recall the definition of twisted $k$-differentials, enhanced level graphs and prong-matchings. In order to define $k$-twisted canonical divisors, we introduce twisted differentials and enhanced level graphs for general values of $k$. However, in the definition of multi-scale differentials and the rest of the text, we will mainly focus on the case that $k=1$.

\begin{df}
     A \emph{twisted $k$-differential} of type $\mu=(m_1,...,m_n)$ on a stable curve~ $(X,\boldsymbol{z}=(z_1,...,z_n))$ a tuple of differentials $\boldsymbol{\eta}=(\eta_v)_{v}$, where $\eta_v$ is a differential on the irreducible component $X_v$ of $X$ such that the following conditions hold.
     \par
     \begin{enumerate}
	\item If $z_i$ is in $X_v$, then $\ord_{z_i}(\eta_v)=\mu_i$.
	\item If $q_1\in X_{v_1}$ and $q_2 \in X_{v_2}$ are identified as a node of $X$, then \[\ord_{q_1}(\eta_{v_1})+\ord_{q_2}(\eta_{v_2})=-2k. \]
	\item If $k=1$ and at a node $(q_1,q_2)$ the equation $\ord_{q_1}(\eta_{v_1})=\ord_{q_2}(\eta_{v_2})=-1$ holds, then \[\res_{q_1}(\eta_{v_1})+\res_{q_2}(\eta_{v_2})=0.\]
	\item The differentials $\eta_v$ have no poles or zeros away from the nodes and the given marked points.
\end{enumerate}
\end{df}

In this text, a twisted differential simply refers to a twisted $1$-differential. Some crucial information of a twisted differential can be described by an enhanced level graph. An \emph{enhanced level graph} $\Delta$ consists of:
\begin{itemize}
    \item[(i)] a dual graph $\bar{\Delta}$ of $(X,\boldsymbol{z})$;
    \item[(ii)] a level function $\ell:V(\bar{\Delta})\longrightarrow \{0,-1,-2,...\}$ such that $\ell^{-1}(0)$ is nonempty;
    \item[(iii)] for each vertical edge $e$ an assigned positive integer $\kappa_e$ which is called \emph{enhancement}.
    \end{itemize}
We say that a twisted $k$-differential $\boldsymbol{\eta}$ is \emph{compatible} with $\Delta$ if for the order on $V(\bar{\Delta})$ induced by $\ell$:

\begin{itemize}
    \item[(iv)] Given an edge $e$ connecting vertices $v_1,v_2$ in $\Delta$, $\ell(v_1)= \ell(v_2)$ if and only if $\ord_{q_1}(\eta_{v_1})=\ord_{q_2}(\eta_{v_2})= -k$. If $\ell(v_1)> \ell(v_2)$, then $\ord_{q_1}(\eta_{v_1})=\kappa_e-k$ and $\ord_{q_2}(\eta_{v_2})=-\kappa_e-k$.
\end{itemize}

If not specified, by saying that a twisted $1$-differential $\boldsymbol{\eta}$ is \emph{compatible} with $\Delta$, the residues of the poles of $\boldsymbol{\eta}$ should also satisfy the following condition called the \emph{global residue condition} (GRC). Since the GRC for general $k$ is technical and will not be used explicitly in our text, we will only state the GRC for $k=1$ here.
\begin{itemize}
    \item[(v)] For every level $L$, a connected component $Y$, which has no leg representing a marked pole, of the subgraph $\Delta_{> L}$ (which is a graph consisting of edges and vertices above level $L$) will give us a residue condition. Let $e_1,...,e_m$ be edges connecting $Y$ to vertices on level $L$. Then the residue condition induced by $Y$ is 
    \[\sum_{j=1}^m \res_{q_j^-}(\eta_{e_j^-})=0,\]
    where $e_j^-$ is the lower vertex of edge $e_j$ and $q_j^-$ is the point in $X_{e_j^-}$ representing the node.
\end{itemize}
We will abuse the notation for edges so that we write $q$ for both a node in $X$ and also the corresponding edge in the compatible graph.

\begin{rem}\label{rem:twist_can}
Let $D=\sum_i m_ip_i$ be a Weil Divisor on a stable curve $X$, where $p_i$ are nonsingular points on $X$ and $\sum_im_i=k(2g-2)$. We call $D$ a \emph{twisted $k$-canonical divisor} if there is a twisted $k$-differential whose poles and zeros away from the nodal points are exactly prescribed by $D$; in addition, the twisted $k$-differential has to be compatible with some enhanced level structure on the dual graph of $X$ in the sense of condition (iv).
\end{rem}

The rest of the text (except Section~\ref{sec:conj}) will mainly focus on the case $k=1$, thus if not specified, a differential will refer to a $1$-differential.

Let $\boldsymbol{\eta}$ be a twisted differential. Each component $\eta_v$ of the twisted differential $\boldsymbol{\eta}$ endows $X_v$ with a flat structure. Moreover, for two points $q^+,q^-$ in $\widetilde{X}=\bigsqcup_{v\in V(\Delta)} X_v$ corresponding to a \emph{vertical node} $q$ (a node corresponding to a vertical edge), there will be the same number of horizontal geodesics converging to these points. We call these geodesics \emph{prongs}. The horizontal geodesics (of direction to east) which run towards the vertical node are \emph{incoming prongs} and the others are \emph{outgoing prongs}. We denote the set of incoming prongs and the set of outgoing prongs at $x$ by $P^{in}_x$ and $P^{out}_x$ respectively. There is a natural cyclic order on the set of prongs. A \emph{local prong-matching} at node $q$ is a cyclic-order-reversing bijection $\sigma_q:P^{in}_{q^-}\longrightarrow P^{out}_{q^+}$ and a \emph{global prong-matching} is a collection $\boldsymbol{\sigma}$ of local prong-matchings at every vertical node.

Let $X$ be a stable curve; $\boldsymbol{\eta}$ be a twisted $1$-differential compatible with some enhanced level structure on the dual graph $\bar{\Delta}$ of $X$; $\boldsymbol{\sigma}$ be a global prong-matching. There is a natural action of the \emph{level rescaling group} $\mathbb{C}^{L}$ on the pair $(\boldsymbol{\eta},\boldsymbol{\sigma})$ (here $L$ is the number of levels), as now we describe:

\begin{itemize}
    \item[(1)] Let $\boldsymbol{d}=(d_j)_{j=0,-1,...,-L+1}\in \mathbb{C}^{L}$. We define 
    \begin{align*}
        \boldsymbol{d}\cdot\boldsymbol{\eta}= \big(\exp(i2\pi d_{\ell(v)})\cdot\eta_v\big)_{v\in V(\Delta)}.
    \end{align*}
    \item[(2)] For each vertical node $q$, we let 
    \begin{align*}
        \boldsymbol{d}\cdot \sigma_q: P^{in}_{q^-}\longrightarrow P^{out}_{q^+}
    \end{align*}
    be the map $\sigma_q$ precomposed and postcomposed by the rotations by the angles $-2\pi \Real(d_{\ell(q^-)})/\kappa_q$ and $2\pi \Real(d_{\ell(q^+)})/\kappa_q$ respectively, so that $\boldsymbol{d}\cdot \sigma_q$ remains to be a prong-matching.
\end{itemize}

Let 
\begin{align*}
\iota: &\mathbb{C}^{L-1}\hookrightarrow \mathbb{C}^{L}\\
& (d_{-1},...,d_{-L+1})\mapsto (0,d_{-1},...,d_{-L+1})
\end{align*}
be the natural inclusion. We call $\mathbb{C}^{L-1}$ the \emph{reduced level rescaling group}. The action of $\mathbb{C}^{L-1}$ on $(\boldsymbol{\eta},\boldsymbol{\sigma})$ is induced by that of the level rescaling group.

 \begin{df}
       A multi-scale differential is a tuple $(\boldsymbol{\eta},\Delta,\boldsymbol{\sigma})$ consisting of a \emph{twisted $1$-differential} $\boldsymbol{\eta}$ of type $\mu$ on a stable curve $(X,\boldsymbol{z})$ such that $\boldsymbol{\eta}$ is compatible with the enhanced level structure $\Delta$ on the underlying dual graph $\bar{\Delta}$ of $(X,\boldsymbol{z})$. Furthermore, $\boldsymbol{\sigma}$ is a global prong-matching of $\boldsymbol{\eta}$. Two multi-scale differentials are equivalent if they differ by the action of the reduced level rescaling group $\mathbb{C}^{L-1}$. The moduli space of multi-scale differentials of type $\mu$ is denoted by $\Xi\overline{\mathcal{M}}_{\boldsymbol{g,n}}(\mu)$.
       
       On the projectivized space $\mathbb{P}\Xi\overline{\mathcal{M}}_{\boldsymbol{g,n}}(\mu)$, two multi-scale differentials are equivalent if they differ by the action of the level rescaling group $\mathbb{C}^{L}$.
 \end{df}

\begin{rem}
We call the subgroup $\mathbb{Z}^L\subset\mathbb{C}^{L}$ the \emph{level rotation group}. It acts on the twisted differentials trivially while it rotates the prong-matchings. Two multi-scale differentials whose underlying twisted differentials are the same are equivalent if and only if their prong-matchings differ by the action of the level rotation group. Thus, in this text, we are only interested in the equivalence classes of prong-matchings. 
\end{rem}

\subsection{Generalised stratum}
\label{sec:general}
We now introduce the \emph{generalised stratum}. The generalised strata arise naturally if we consider each level of a (connected) enhanced level graph. It is a product of strata of differentials: \begin{align*}
    \Omega\mathcal{M}_{\boldsymbol{g},\boldsymbol{n}}(\boldsymbol{\mu})=\prod_{\nu=1}^k\Omega\mathcal{M}_{g_\nu,n_\nu}(\mu_\nu).
\end{align*}
It parametrizes differentials on disconnected compact Riemann surfaces. Similar to a moduli space of multi-scale differentials which compactifies a stratum of differentials, there is a compactification of a (projectivized) generalised stratum $\mathbb{P}\Omega\mathcal{M} _{\boldsymbol{g},\boldsymbol{n}}(\boldsymbol{\mu})$, which we denote by $\mathbb{P}\Xi\overline{\mathcal{M}} _{\boldsymbol{g},\boldsymbol{n}}(\boldsymbol{\mu}):=\mathbb{P}\big(\prod_{\nu=1}^k\Xi\mathcal{M}_{g_\nu,n_\nu}(\mu_\nu)\big)$. According to \cite{costantini2020chern}, $\mathbb{P}\Omega\mathcal{M} _{\boldsymbol{g},\boldsymbol{n}}(\boldsymbol{\mu})\subset \mathbb{P}\Xi\overline{\mathcal{M}} _{\boldsymbol{g},\boldsymbol{n}}(\boldsymbol{\mu})$ is a compactification with normal crossing boundary divisors. The boundary strata of $\mathbb{P}\Xi\overline{\mathcal{M}} _{\boldsymbol{g},\boldsymbol{n}}(\boldsymbol{\mu})$ are parametrized by (disconnected) enhanced level graphs.

 Moreover, if we consider a level below zero, there are also extra residue conditions induced by the global residue condition. Hence, we also introduce the notion of a \emph{generalised stratum with residue conditions} $\Omega\mathcal{M} ^\mathfrak{R}_{\boldsymbol{g},\boldsymbol{n}}(\boldsymbol{\mu})$. Here, $\mathfrak{R}$ refers to the residue space cut out by a set of residue conditions. Let $H_p$ be the set of pairs $(\nu, i)$ where $i$ is a marking of the component $\nu$ such that the marking $i$ is a pole. A collection $\mathfrak{R}$ of residue conditions consists of a partition $\boldsymbol{\lambda}_\mathfrak{R}$ of $H_p$ with parts $\lambda^{(j)}$. A part $\lambda^{(j)}$ represents a residue equation
\[\sum_{(\nu,i)\in \lambda^{(j)}}r_{\nu,i}=0 \]
Thus, 
\[\mathfrak{R}=\{(r_{\nu,i})_{(\nu,j)\in H_p}: \sum_{(\nu,i)\in \lambda^{(j)}}r_{\nu,i}=0, \lambda^{(j)}\in \boldsymbol{\lambda}_\mathfrak{R}\} \]

\begin{rem}
In analogy to the case of a usual stratum of differentials (i.e. with only one signature), we also call the projection pushforward of the fundamental class of a generalised stratum $p_*\big([\mathbb{P}\Omega\mathcal{M} _{\boldsymbol{g},\boldsymbol{n}}(\boldsymbol{\mu})]\big)\in H^*(\prod_i \mathcal{M}_{g_i,n_i},\mathbb{Q})$   a \emph{stratum class}. If all the signatures of a generalised stratum are of even type, we call such generalised stratum a \emph{generalised stratum of even type}. Given a generalised stratum of even type, we can define the class \[p_*\big([\mathbb{P}\Omega\mathcal{M} _{\boldsymbol{g},\boldsymbol{n}}(\boldsymbol{\mu})]^{\spin}\big)=p_*\big([\mathbb{P}\Omega\mathcal{M} _{\boldsymbol{g},\boldsymbol{n}}(\boldsymbol{\mu})^+]\big)-p_*\big([\mathbb{P}\Omega\mathcal{M} _{\boldsymbol{g},\boldsymbol{n}}(\boldsymbol{\mu})^-]\big)\] and call it a \emph{spin stratum class}.
\end{rem}

\begin{rem}
From now on, for brevity, we will call a (projectivized) moduli space of multi-scale differentials, which compactifies a (projectivized) stratum of differentials, a \emph{compactified stratum}. Similarly, a \emph{compactified generalised stratum} refers to a (projectivized) moduli space of multi-scale differentials of type $\boldsymbol{\mu}$ on disconnected stable curves which compactifies a (projectivized) generalised stratum. 
\end{rem}

\subsection{The spin parities of welded surfaces and plumbed surfaces}\label{sec:spin_welded}
\label{sec:welded}
 In this subsection, we will briefly introduce the welded surface and plumbed surface associated to a multi-scale differential. We refer the readers to \cite{bainbridge2019moduli} for more details. Then we will define the spin parities on them and show why these two spin parities coincide with each other. 
 
Given a multi-scale differential $(\boldsymbol{\eta},\Delta,\boldsymbol{\sigma})$, a \emph{welded surface} is a topological surface $\overline{X}_{\boldsymbol{\sigma}}$ with a map $f:\overline{X}_{\boldsymbol{\sigma}} \longrightarrow X$ such that 
\begin{itemize}
    \item the preimage of a node is a simple loop which we call a \emph{seam};
    \item away from the seams, $f$ is an isomorphism;
    \item the glueing of two boundaries circles of the preimages of two irreducible components of $X$ (to form a seam) is determined by the prong-matchings. 
\end{itemize}
 For instance, the compact closed surface in Figure~\ref{fig:welded_surf} is a welded surface associated to the nodal curve in Figure~\ref{fig:my_label}. Notice that the twisted differential $\boldsymbol{\eta}$ endows $\overline{X}_{\boldsymbol{\sigma}}$ with a piecewise flat structure. From now on, if not specified, the welded surface associated to a multi-scale differential will refer to the topological welded surface plus the piecewise flat structure on it.
 
 Now we will briefly introduce the plumbing surfaces. According to \cite{bainbridge2018compactification}, given a multi-scale differential $(\boldsymbol{\eta},\Delta,\boldsymbol{\sigma})$ such that $X$ is not smooth, one can construct a family of flat surfaces which is controlled by a collection of parameters $t_{-1},...,t_{-L+1}\in\mathbb{C}$ indexed by the levels below zero. A generic fiber is a connected flat surface, while the fiber over $t_{-1}=...=t_{-L+1}=0$ is the disconnected flat surface associated to the twisted differential $\boldsymbol{\eta}$. This deformation family of flat surfaces is called a \emph{plumbing family}. A connected fiber of the deformation family will be called a \emph{plumbed surface}. A plumbed surface is just a flat surface and its spin parity is defined as that of a flat surface. The construction of a plumbed surface is actually quite technical, nevertheless, if we only want to make arguments on the spin parity, it suffices to just consider the welded surface. 

A welded surface $\overline{X}_{\boldsymbol{\sigma}}$ is an orientable closed surface. By considering the piecewise flat structure on the welded surface, one can still calculate the turning number of a smooth simple loop if the curve is transverse to the seams. 

\begin{df}
Let $\overline{X}_{\boldsymbol{\sigma}}$ be a welded surface and $\gamma:[0,1]\longrightarrow \overline{X}_{\boldsymbol{\sigma}}$ be a smooth simple loop such that it is transverse to the seams. Let $0=y_0<y_1<...<y_N=1$ be a partition of the unit interval such that $\gamma\big( (y_j,y_{j+1}) \big)$ does not intersect any seam for $j=0,...,N-1$. The \emph{turning number} of $\gamma$ is defined as 
\begin{align*}
    \ind(\gamma)=\sum_{j=0}^{N-1}\ind(\gamma|_{[y_j,y_{j+1}]})
\end{align*}
Then the \emph{spin parity of a welded surface} is just computed by the Formula~\ref{eq:parity}, where the symplectic basis is chosen for $H_1(\overline{X}_{\boldsymbol{\sigma}},\mathbb{Z})$. In addition, on a disconnected welded surface, the spin parity is just defined to be the sum of the spin parities on the connected components. 
\end{df}

\begin{prop}
Given a multi-scale differential of even type, the spin parity of the welded surface associated to it equals the spin parity of a plumbed surface associated to it.
\end{prop}
\begin{proof}[Sketch of proof]
 Note that the flat picture of a plumbed surface differs from a welded surface by deleting some strips and disc neighbourhoods (cf. \cite{bainbridge2018compactification} for the flat picture of a plumbed surface). For every simple closed curve on the plumbed surface, one can complete it to a simple closed curve on the welded surface which is transverse to the seams. The turning number of the completed curve will be the same as that of the original curve. As a result, a symplectic basis on the plumbed surface induces a symplectic basis on the welded surface such that the turning numbers remain unchanged.
\end{proof}

By the above proposition, we can determine the parity of a multi-scale differential of even type by computing (\ref{eq:parity}) on the welded surface with respect to simple closed curves transverse to seams.

\section{Prong-matching and spin parity of a welded surface} 
 In this section, we construct a basis of homology on the welded surface associated to a multi-scale differential. This basis is convenient for computing the spin parity and we call such basis a $\Delta$-adapted symplectic basis. The idea comes from the paper \cite{benirschke2020boundary}, where Benirschke constructed a $\Delta$-adapted basis of the relative homology $\mathrm{H}_1(\overline{X}_{\boldsymbol{\sigma}}\setminus P,Z)$, namely a basis comprised of generators coming from each irreducible component of $X$ and generators capturing the homology of the enhanced level graph $\Delta$. Here, $\overline{X}_{\boldsymbol{\sigma}}$ is a welded surface compatible with the enhanced level graph $\Delta$; $P$ and $Z$ are the sets of poles and zeros corresponding to the legs of $\Delta$ respectively. In our text, we need a symplectic basis to compute the spin parity. That is why we cannot directly adopt the $\Delta$-adapted basis constructed in \cite{benirschke2020boundary}. More precisely, we have to ensure that the choice of generating cycles of the graph can be completed to a symplectic basis on the welded surface. On the other hand, the idea behind our definition is very similar to the definition in \cite{benirschke2020boundary}: we both want to construct a basis consisting of cycles from lifting of the bases of homology on various components, and cycles that cross the seams. 
\label{sec:prong}
\subsection{$\Delta$-adapted symplectic basis on a welded surface}\label{sec:sym_basis} 

 Let $\overline{X}_{\boldsymbol{\sigma}}$ be a welded surface compatible with a level graph $\Delta$. A cycle $\gamma$ on $\overline{X}_{\boldsymbol{\sigma}}$ is called a \emph{graph cycle} if it can be realized as a simple loop $\gamma$ such that the geometric intersection numbers $\langle\gamma,\delta\rangle \leq 1$ for each seam $\delta$ but not all the intersection numbers are zero. A cycle $\delta$ on $\overline{X}_{\boldsymbol{\sigma}}$ is called a \emph{vanishing cycle} if it is the homology class of some seam. A cycle which can be realised by a simple loop not intersecting any seam is called a \emph{non-crossing cycle}.

\begin{exa}\label{prop:graphical}

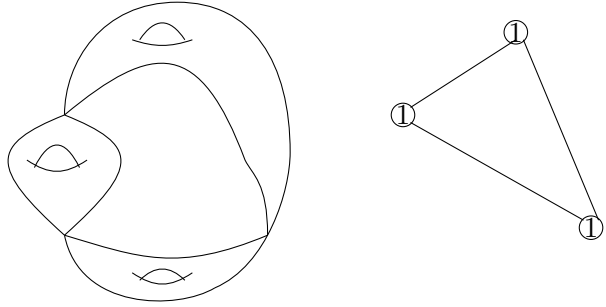
\begin{figure}
    \centering
     \begin{tikzpicture}
   \draw (0,0).. controls (0,0.7) and (0.5,1.5).. (1.5,1.5).. controls (2.5,1.5) and (3,0.7).. (3,-0.5).. controls (3,-1) and (2.7,-2.0).. (2,-2.3);
   \draw (0,0).. controls (1.2,1) and (1.8,1).. (2.4,-0.6).. controls (2.5,-0.8) and (2.7,-0.9)..(2.7,-1.6);
   \draw (0,0).. controls (-1,-0.4) and (-1,-0.7).. (0,-1.6);
   \draw (0,0).. controls (1,-0.4) and (1,-0.7).. (0,-1.6);
   \draw (0,-1.6).. controls (1,-1.9) and (1.5,-2.1).. (2.7,-1.6);
    \draw (0,-1.6).. controls (0.2,-2.6) and (1.5,-2.6).. (2,-2.3);
    \draw (1.0,1.0) .. controls (1.2,1.3) and (1.4,1.3) .. (1.6,1.0) ;
    \draw (0.9,1.0) .. controls (1.2,0.9) and (1.4,0.9) .. (1.7,1.0) ;
     \draw (-0.4,-0.7) .. controls (-0.2,-0.3) and (0,-0.3) .. (0.2,-0.7) ;
    \draw (-0.5,-0.6) .. controls (-0.2,-0.8) and (0,-0.8) .. (0.3,-0.6) ;
    \draw (1.0,-2.2) .. controls (1.2,-2) and (1.4,-2) .. (1.6,-2.2) ;
    \draw (0.9,-2.1) .. controls (1.2,-2.3) and (1.4,-2.3) .. (1.7,-2.1) ;
    
    \node at (6,1.1) [circle,draw,inner sep=0pt,minimum size=5pt]{1};
     \node at (4.5,0) [circle,draw,inner sep=0pt,minimum size=5pt]{1};
     \node at (7,-1.5) [circle,draw,inner sep=0pt,minimum size=5pt]{1};
     \draw (6,1.0)-- (4.6,0.1);
     \draw (4.6,-0.1)-- (6.9,-1.4);
     \draw (7.1,-1.4)-- (6.1,1.0);
   \end{tikzpicture}
    \caption{A nodal curve and its dual graph}
    \label{fig:my_label}
\end{figure}

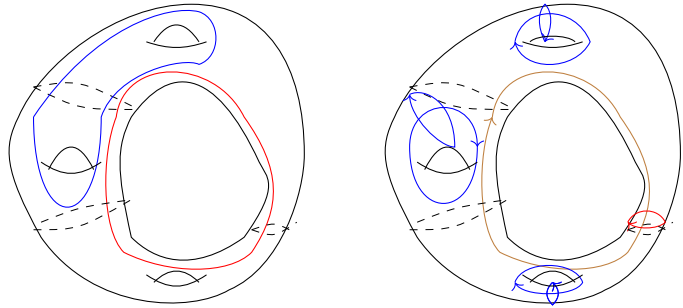
\begin{figure} 
    \centering
   \begin{tikzpicture}
   \draw (-0.8,0).. controls (-0.5,0.7) and (0.5,1.5).. (1.5,1.5).. controls (2.5,1.5) and (3,0.7).. (3,-0.5).. controls (3,-1) and (2.7,-2.0).. (2,-2.3).. controls (1.5,-2.6) and (0.2,-2.6)..  (-0.5,-1.6).. controls (-1,-0.7) and (-1,-0.4) .. (-0.8,0);
   \draw (0.7,0).. controls (1.2,0.7) and (1.8,0.7).. (2.4,-0.6).. controls (2.5,-0.8) and (2.7,-0.9)..(2.2,-1.6)..   controls (1.5,-2.1) and (1,-1.9) .. (0.7,-1.6).. controls (0.5,-0.7) and (0.5,-0.4).. (0.7,0) ;
  
   \draw [red] (0.5,0).. controls (0.6,0.8) and (1.8,0.8) .. (2.2,0).. controls (2.7,-0.7) and (2.7,-1.2).. (2.3,-1.8).. controls (2.0,-2.1) and (1.2,-2.1).. (0.6,-1.8).. controls (0.3,-1.5) and (0.3,-0.5).. (0.5,0) ;
   
   \draw [dashed,very thin] (-0.6,0.4).. controls (0,0.5).. (0.8,0.1);
   \draw [dashed,very thin] (-0.6,0.4).. controls (0,0.1).. (0.8,0.1);
   \draw [dashed,very thin] (-0.5,-1.4).. controls (0,-1.2).. (0.7,-1.1);
    \draw [dashed,very thin] (-0.6,-1.5).. controls (0,-1.5).. (0.7,-1.1);
    \draw [dashed,very thin] (2.3,-1.5).. controls (2.6,-1.6).. (2.9,-1.4);
    \draw [dashed,very thin] (2.3,-1.5).. controls (2.6,-1.4).. (2.8,-1.5);
    \draw [blue] (-0.6,0).. controls (0.2,1.2) and (1.2, 1.5).. (1.6,1.4).. controls (1.9,1.3) and (1.9,0.8).. (1.6,0.7).. controls  (1.4, 0.8) and (0.5,0.6)..(0.3,0).. controls (0.3,-1.6) and (-0.6,-1.6).. (-0.6,0);
    \draw (1.0,1.0) .. controls (1.2,1.3) and (1.4,1.3) .. (1.6,1.0) ;
    \draw (0.9,1.0) .. controls (1.2,0.9) and (1.4,0.9) .. (1.7,1.0) ;
     \draw (-0.4,-0.7) .. controls (-0.2,-0.3) and (0,-0.3) .. (0.2,-0.7) ;
    \draw (-0.5,-0.6) .. controls (-0.2,-0.8) and (0,-0.8) .. (0.3,-0.6) ;
    \draw (1.0,-2.2) .. controls (1.2,-2) and (1.4,-2) .. (1.6,-2.2) ;
    \draw (0.9,-2.1) .. controls (1.2,-2.3) and (1.4,-2.3) .. (1.7,-2.1) ;
    \begin{scope}[xshift=5cm]
        \draw (-0.8,0).. controls (-0.5,0.7) and (0.5,1.5).. (1.5,1.5).. controls (2.5,1.5) and (3,0.7).. (3,-0.5).. controls (3,-1) and (2.7,-2.0).. (2,-2.3).. controls (1.5,-2.6) and (0.2,-2.6)..  (-0.5,-1.6).. controls (-1,-0.7) and (-1,-0.4) .. (-0.8,0);
   \draw (0.7,0).. controls (1.2,0.7) and (1.8,0.7).. (2.4,-0.6).. controls (2.5,-0.8) and (2.7,-0.9)..(2.2,-1.6)..   controls (1.5,-2.1) and (1,-1.9) .. (0.7,-1.6).. controls (0.5,-0.7) and (0.5,-0.4).. (0.7,0) ;
  
   \draw [->,brown] (0.5,0).. controls (0.6,0.8) and (1.8,0.8) .. (2.2,0).. controls (2.7,-0.7) and (2.7,-1.2).. (2.3,-1.8).. controls (2.0,-2.1) and (1.2,-2.1).. (0.6,-1.8).. controls (0.3,-1.5) and (0.3,-0.5).. (0.5,0) ;
    \draw [->,red] (2.3,-1.4).. controls (2.4,-1.2) and (2.7,-1.2) .. (2.8,-1.4).. controls (2.7,-1.5) and (2.4,-1.5).. (2.3,-1.4);
   
   \draw [dashed,very thin] (-0.6,0.4).. controls (0,0.5).. (0.8,0.1);
   \draw [dashed,very thin] (-0.6,0.4).. controls (0,0.1).. (0.8,0.1);
   \draw [dashed,very thin] (-0.5,-1.4).. controls (0,-1.2).. (0.7,-1.1);
    \draw [dashed,very thin] (-0.6,-1.5).. controls (0,-1.5).. (0.7,-1.1);
    \draw [dashed,very thin] (2.3,-1.5).. controls (2.6,-1.6).. (2.9,-1.4);
    \draw [dashed,very thin] (2.3,-1.5).. controls (2.6,-1.4).. (2.8,-1.5);
    
    \draw [->,blue] (0.3,-0.4).. controls (0.3,-1.4) and (-0.6,-1.4).. (-0.6,-0.4).. controls (-0.6,0.3) and (0.3,0.3).. (0.3,-0.4);
    \draw [->,blue] (0.8,1.0).. controls (0.9,1.5) and (1.6,1.5).. (1.8,1.0).. controls (1.6,0.6) and (0.9,0.6).. (0.8,1.0);
    \draw [->,blue] (0.8,-2.2).. controls (0.9,-1.9) and (1.6,-1.9).. (1.7,-2.2).. controls (1.6,-2.4) and (0.9,-2.4).. (0.8,-2.2);
    \draw [->,blue] (1.2,1.0).. controls (1.1,1.2) and (1.1,1.4).. (1.2,1.5).. controls (1.3,1.4) and (1.3,1.2)..(1.2,1.0);
    \draw [->,blue] (-0.6,0.3).. controls (-0.4,0.4) and (0,0.2).. (0.0,-0.4).. controls (-0.2,-0.4) and (-0.6,0).. (-0.6,0.3);
    \draw [->,blue] (1.3,-2.2).. controls (1.2,-2.3) and (1.2,-2.4).. (1.3,-2.5).. controls (1.4,-2.4) and (1.4,-2.3).. (1.3,-2.2);
    \draw [->,blue] (1.3,-2.2).. controls (1.2,-2.3) and (1.2,-2.4).. (1.3,-2.5).. controls (1.4,-2.4) and (1.4,-2.3).. (1.3,-2.2);
   
    \draw (1.0,1.0) .. controls (1.1,1.1) and (1.5,1.1) .. (1.6,1.0) ;
    \draw (0.9,1.0) .. controls (1.2,0.9) and (1.4,0.9) .. (1.7,1.0) ;
     \draw (-0.4,-0.7) .. controls (-0.2,-0.3) and (0,-0.3) .. (0.2,-0.7) ;
    \draw (-0.5,-0.6) .. controls (-0.2,-0.8) and (0,-0.8) .. (0.3,-0.6) ;
    \draw (1.0,-2.2) .. controls (1.2,-2) and (1.4,-2) .. (1.6,-2.2) ;
    \draw (0.9,-2.1) .. controls (1.2,-2.3) and (1.4,-2.3) .. (1.7,-2.1) ;
    \end{scope}
\end{tikzpicture}
\caption{The associated welded surface and a $\Delta$-adapted symplectic basis} \label{fig:welded_surf}
\end{figure}
 
Consider the nodal curve in Figure~\ref{fig:my_label}, on the associated welded surface (Figure~\ref{fig:welded_surf} left) we have a blue loop $\gamma_b$ and a red loop $\gamma_r$. We can see that there is a seam such that the geometric intersection number with $\gamma_b$ is $2$. However, $\gamma_r$ intersects all the seams only once. Thus $\gamma_r$ is a graph cycle while $\gamma_b$ is not. 
\end{exa}

Note that a graph cycle $\gamma$ induces a loop on $\Delta$ which consists of edges whose corresponding seam intersects with $\gamma$.

\begin{df}
 Let $\overline{X}_{\boldsymbol{\sigma}}$ be a welded surface which is compatible with a level graph~$\Delta$. A \emph{ $\Delta$-adapted symplectic basis} on a welded surface $\overline{X}_{\boldsymbol{\sigma}}$ is a symplectic basis such that it contains a set of graph cycles whose associated loops on $\Delta$ form a basis of~$H_1(\Delta,\mathbb{Z})$ and moreover, remaining cycles are either non-crossing cycles or vanishing cycles dual to the graph cycles. 
\end{df}


\begin{exa}
Let us consider the welded surface in Example~\ref{prop:graphical}. The blue cycles in Figure~\ref{fig:welded_surf} right are non-crossing cycles and the brown one is a graph cycle and the red one is a vanishing cycle. They constitute a $\Delta$-adapted symplectic basis on the welded surface.
\end{exa}

\subsection{Algorithm to construct a $\Delta$-adapted symplectic basis}
 
This subsection is to prove the existence of a $\Delta$-adapted symplectic basis. Notice that we are arguing just by topological properties of a nodal curve and its dual graph. Thus, for every topological welded surface, our claim of the existence of a $\Delta$-adapted symplectic basis still holds. To construct a $\Delta$-adapted symplectic basis on the welded surface $\overline{X}_{\boldsymbol{\sigma}}$, it is equivalent to finding a basis of~$\mathrm{H}_1(\Delta)$. However, not every spanning tree can give us a set of fundamental cycles that can be completed to a $\Delta$-adapted symplectic basis. Namely, the simple closed curves on the welded surface, which realize the fundamental cycles of the graph, cannot be chosen to be disjoint. The following example will illustrate the problem.
 
 \begin{exa}
Consider the following graph, where the black edges constitute a spanning tree $T_0$ and every coloured edge gives rise to a fundamental cycle

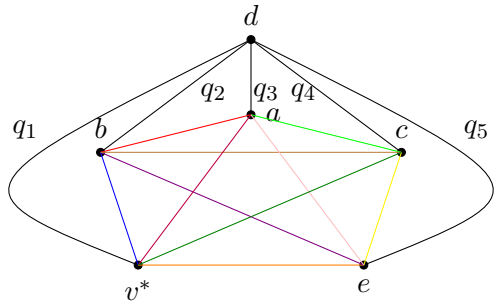
\begin{center}
    \begin{tikzpicture}
    \node at (0,0) [circle,draw,fill,inner sep=0pt,minimum size=3pt, label=above:$d$]{};
    \node at (0,-1) [circle,draw,fill,inner sep=0pt,minimum size=3pt,label=right:$a$]{};
    \node at (-2,-1.5) [circle,draw,fill,inner sep=0pt,minimum size=3pt,label=above:$b$]{};
    \node at (2,-1.5) [circle,draw,fill,inner sep=0pt,minimum size=3pt,label=above:$c$]{};
    \node at (-1.5,-3) [circle,draw,fill,inner sep=0pt,minimum size=3pt,label=below:$v^*$]{};
    \node at (1.5,-3) [circle,draw,fill,inner sep=0pt,minimum size=3pt,label=below:$e$]{};
    \node at (-3,-1.2) {$q_1$};
    \node at (-0.5,-0.7) {$q_2$};
    \node at (0.2,-0.7) {$q_3$};
    \node at (0.7,-0.7) {$q_4$};
    \node at (3,-1.2) {$q_5$};
    \node at (-1.8,-2.5) {};
    \node at (-1.2,-2.0) {};
    \node at (-1.0,-1.6) {};
    \node at (-1.0,-1.1) {};
    \node at (-1.1,-2.4) {};
    \node at (-0.5,-2.6) {};
    \node at (-0.2,-3.2) {};
    \node at (1.2,-2.5) {};
    \node at (1.9,-2.3) {};
    \node at (1.0,-1.1) {};
    \draw (0,0) -- (0,-1);
    \draw (0,0) -- (-2,-1.5);
     \draw (0,0) -- (2,-1.5);
     \draw (0,0).. controls(-4,-2).. (-1.5,-3);
     \draw (0,0).. controls(4,-2).. (1.5,-3);
     \draw [red](0,-1) -- (-2,-1.5);
      \draw [blue](-2,-1.5) -- (-1.5,-3);
      \draw [orange](-1.5,-3) -- (1.5,-3);
      \draw [yellow](1.5,-3) -- (2,-1.5);
     \draw [green](2,-1.5) -- (0,-1);
     \draw [violet](-2,-1.5) -- (1.5,-3);
     \draw [brown](-2,-1.5) -- (2,-1.5);
     \draw [purple](0,-1) -- (-1.5,-3);
     \draw [pink](0,-1) -- (1.5,-3);
     \draw [green!50!black](-1.5,-3) -- (2,-1.5);
    \end{tikzpicture}
\end{center}
Five edges are connected to $d$ and any two of them share a fundamental cycle induced by~$T_0$. Hence, if we want to realize the graph cycles as some closed curves, any two nodal marked points in $C_d$ will be connected by a segment. Then there will be some segments that intersect with each other because the 5 nodal marked points and the segments connecting them form a $K_5$-graph, which is not planar. 

\end{exa}
 
  We now show how this problem can be circumvented. Suppose we are given a spanning tree $T_0$ of a graph $\Delta$, we can apply the following iteration to obtain a tree $T$. This tree will induce a collection of graph cycles such that on each $C_v$ the nodal marked points and the segments connecting them form a planar graph. The main idea is to reduce the number of edges in a spanning tree adjacent to a vertex.
\begin{enumerate}
    \item Take an initial vertex $v^*$ such that it is only connected to one edge in $T_0$. The vertex $v^*$ is fixed throughout the iteration. On $T_k$, for every vertex $v$ except $v^*$, there is exactly one edge adjacent to $v$ going towards $v^*$, which we denote as $e^{(k)}_v$ and others are all going away from $v^*$.
    \item Let $\{q^{(k)}_1,...,q^{(k)}_{g'}\}$ be edges of $\Delta$ not included in $T_k$ and $c_{q^{(k)}_i}$ the fundamental cycles formed by adding $q^{(k)}_i$ in the tree. Let $w$ be a vertex of distance $k+1$ to~$v^*$ in $T_{k}$. For any two edges adjacent to $w$, which are not $e^{(k)}_w$ and share a fundamental cycle $c_{q^{(k)}_i}$, we replace one of these two edges by $q^{(k)}_i$. The edge $q^{(k)}_i$ should be connected only to vertices that are of distance greater equal $k+1$ to $v^*$. We carry out this operation until all the vertices of distance $k+1$ to $v^*$ in $T_{k}$ have no going-away edges sharing a fundamental cycle. The resulting tree will be set to be $T_{k+1}$.
    \item Notice that the set of vertices of distance $k'$ to $v^*$ in $T_k$, where $k'\leq k$, will be also the set of vertices of distance $k'$ to $v^*$ in $T_{k+1}$. Thus our operation in (2) will not change the distances we fixed in earlier stages. We terminate the iteration when we have exhausted all the vertices.
\end{enumerate}
 Let $\Sigma_v$ be the graph whose vertices are the nodal marked points on $C_v$, while its edges are segments connecting them. Now we show that the final spanning tree $T$ from the iteration above will give us on each $C_v$ a graph $\Sigma_v$ that is planar. First, notice that there are 3 types of vertices in $\Sigma_v$ which represent: 
\begin{itemize}
    \item[(i)] a node which is represented by an edge adjacent to vertex $v$ in the tree $T$ towards the initial vertex $v^*$,
    \item[(ii)] a node which is represented by an edge adjacent to vertex $v$ in the tree $T$ away $v^*$ and 
    \item[(iii)] a node which is represented by an edge, not in the tree $T$
\end{itemize}
If the $\Sigma_v$ is induced by the result spanning tree $T$ from the iteration, then the vertices of the same type are not connected. There is only one vertex of type (i) in $\Sigma_v$ and every vertex of type (iii) in $\Sigma_v$ has only one edge to some vertex of type (ii). Moreover, the unique vertex of type (i) is only connected to vertices of type (ii). Thus, any cycle in $\Sigma_v$ can only consist of multiple edges between the unique vertex of type (i) and a vertex of type (ii). This kind of $\Sigma_v$ is always planar. Indeed, if we regard the multiple edges between the unique vertex of type (i) and a vertex of type (ii) as one edge, the resulting graph $\Sigma_v'$ will be a tree. As adding an extra edge between two vertices will also add an extra face, the Euler characteristic will not change. Hence, the graph $\Sigma_v$ induced by the resulting tree $T$ is planar. Hence, we can now assert the existence of a $\Delta$-adapted symplectic basis on a welded surface. Finally, we can present our proof of Proposition~\ref{prop:adapt}. 

\begin{prop}\label{prop:adapt}
Let $\overline{X}_{\boldsymbol{\sigma}}$ be a welded surface compatible with an enhanced level graph~$\Delta$. Then there exists a $\Delta$-adapted symplectic basis on $\overline{X}_{\boldsymbol{\sigma}}$. Moreover, if $\Delta$ has an edge of even enhancements, then the basis can be so chosen that it contains a vanishing cycle corresponding to an edge of even enhancement. 
\end{prop}

\begin{proof}
 Notice that every vertex of $\Delta$ can only be adjacent to an even number of edges of even enhancements, and hence by contracting all the edges of odd enhancements, we will get a graph whose vertices are of even degree. It is well-known that every edge of such a graph is non-separating. Thus, every edge of even enhancement in a level graph is non-separating. Let $\Delta'$ be the graph obtained by removing the chosen edge. We can apply our algorithm to generate a spanning tree $T$ for a $\Delta'$-adapted symplectic basis. It is easy to see that we can construct a $\Delta$-adapted symplectic basis by using the spanning tree $T$ of $\Delta'$.
\end{proof}

\subsection{How a prong-matching rotation changes the spin parity}\label{sec:prong_rotation} In this subsection, we will describe how the spin parity changes when we change the prong-matching at a node.

  The finite group $P_\Gamma=\prod_{e\in E(\Delta)}\mathbb{Z}/\kappa_e\mathbb{Z}$ is called the \emph{prong rotation group}. Now let $r_q=(0,..,0,l,0,...0)$ (only the component for $e=q$ is $l$), where $\kappa_q>l>0$, be a representative of an element in the prong rotation group. Then it acts on a global prong-matching $\boldsymbol{\sigma}$ by altering the local prong-matching $\sigma_q$. It precomposes the cyclic-order-reversing bijection $\sigma_q:P^{in}_{q^-}\longrightarrow P^{out}_{q^+}$ with a rotation by $\frac{2\pi l}{\kappa_q}$ (in the ascending direction with respect to the cyclic order on $P^{in}_{q^-}$ ). 
  
  We denote the new welded surface obtained by applying $r_q$ on the prong-matchings be $\overline{X}_{r_q\boldsymbol{\sigma}}$. Let $\gamma$ on $\overline{X}_{\boldsymbol{\sigma}}$ be a loop representing a graph cycle passing through the seam corresponding to $q$. After we apply the prong rotation $r_q$, $\gamma$ will be a broken loop on~$\overline{X}_{r_q\boldsymbol{\sigma}}$. By connecting the two ends of the broken loop by an arc of the seam cycle of $q$ we will get back a loop $r_q\gamma$ on $\overline{X}_{r_q\boldsymbol{\sigma}}$. By smoothing the loop, one can easily see that~$\ind(\gamma)-\ind(r_q\gamma)=l\pmod 2$. 

Another easy observation is that if $q$ is of odd enhancement, the vanishing cycle of~$q$ will have odd turning number. Hence by expression~(\ref{eq:parity}), a symplectic pair which includes the vanishing cycle associated to $q$ will contribute $0$ to the spin parity, no matter which prong-matching is chosen. 

Now we have the preparatory material to present the proof of Proposition~\ref{prop:1.1}.

 \begin{proof}[Proof of Proposition~\ref{prop:1.1}]
 If there is an edge of even enhancement, then by Proposition~\ref{prop:adapt}, we can choose a symplectic basis such that it contains a vanishing cycle $\alpha$ corresponding to an edge of even enhancement. Let $\beta$ be the dual of $\alpha$ in the symplectic basis. By a generating local rotation at that edge, the turning number of $\beta$ will increase by $1$, while the turning numbers of other cycles in the symplectic basis remain unchanged. Hence, half of the prong-matchings give us an even spin and the other half of the pair of prong-matchings give us an odd spin. 
 
 Let $\Delta'$ be the graph obtained by splitting all the edges of odd enhancements in~$\Delta$ into legs. If $h^1(\Delta')=0$ , then there is no graph cycle in the (disconnected) welded surface corresponding to $\Delta'$. Hence, rotating any local prong-matching will not change the turning number
$\pmod 2$ of members in a $\Delta$-adapted symplectic basis, thus the spin parity will remain unchanged. As a result, the spin parity of the welded surface associated to $\Delta'$ is just the same as that of the welded surface associated to $\Delta$.
 
 \end{proof}
 
 \section{Vanishing of the rational cohomology of $\mathcal{M}_{g,2}$ at v.c.d.}
 In this section, we want to prove Theorem~\ref{thm:vanish} which will be one of the input to prove Proposition~\ref{prop:arbarello} on the injectivity of clutching pullbacks.
 Let $\Mod_g$ (resp. $\Mod_{g,n}$) be the mapping class group of an oriented closed genus $g$ surface (resp. with $n$-punctures). It is well-known that 
 \[H^i(\mathcal{M}_{g,n};\mathbb{Q})\simeq H^i(\Mod_{g,n};\mathbb{Q}) \]
 It has been shown by Harer (\cite{harer1986virtual}) that $\Mod_g$  resp. $\Mod_{g,1}$ are \emph{virtual duality group} of virtual cohomological dimension $4g-5$ resp. $4g-3$, i.e. there exists a dualizing module called Steinberg module $\St_g$ (that will be defined later) such that for any torsion-free finite index subgroup $G<\Mod_g$ (resp. $G'<\Mod_{g,1}$) and $\Mod_g$ (resp. $\Mod_{g,1}$) module~$A$, we have
 
 \begin{align*}
     H^i(G,A)&=H_{4g-5-i}(G,\St_g\otimes_\mathbb{Z} A)\\
     H^i(G',A)&=H_{4g-3-i}(G',\St_g\otimes_\mathbb{Z} A),
 \end{align*}
 where $G$ (resp. $G'$) acts on $\St_g\otimes_\mathbb{Z} A$ by diagonal action. If we take rational coefficients, then the virtual duality implies
 \begin{align*}
     H^i(\Mod_{g},A\otimes_\mathbb{Z}\mathbb{Q})&=H_{4g-5-i}(\Mod_g,\St_g\otimes_\mathbb{Z} A\otimes_\mathbb{Z}\mathbb{Q})\\
     H^i(\Mod_{g,1},A\otimes_\mathbb{Z}\mathbb{Q})&=H_{4g-3-i}(\Mod_{g,1},\St_g\otimes_\mathbb{Z} A\otimes_\mathbb{Z}\mathbb{Q}),
 \end{align*}

 In \cite{church2012rational}, it has been proved that the coinvariant $(\St_g)_{\Mod_g}=0$ and this leads to 
 $H^{4g-3}(\mathcal{M}_g;\mathbb{Q})=H^{4g-5}(\mathcal{M}_{g,1};\mathbb{Q})=0$. Our aim in this section is to extend the result for $\mathcal{M}_{g,2}$:
 
 \begin{thm}\label{thm:vanish}
 For any $g\geq 1$, we have $H^{4g-2}(\mathcal{M}_{g,2};\mathbb{Q})=0$.
 \end{thm}
 
 \subsection{The Steinberg module $\St_g$}
 In this subsection, we describe the Steinberg module $\St_g$. Let $G$ be a torsion-free finite index subgroup of $\Mod_{g}$ (resp. $\Mod_{g,1}$). Then it acts on the Teichmüller space $\mathcal{T}_g$ (resp. $\mathcal{T}_{g,1}$) freely and properly discontinuously, which implies that the quotient space will be a manifold. Ivanov \cite{ivanov1990attaching} and Harer \cite{harer1986virtual} have constructed some contractible bordification $\overline{\mathcal{T}}$ of the Teichmüller space such that the quotient space is compact. According to Theorem 6.2 of \cite{bieri1973groups}, the dualizing module for $G$ is then the reduced homology $\widetilde{H}_{2g-2}(\partial\overline{\mathcal{T}};\mathbb{Z})$. Ivanov \cite{ivanov1987complexes} (resp. Harer \cite{harer1986virtual}) has proved that 
 $\partial\overline{\mathcal{T}}_g$ (resp. $\partial\overline{\mathcal{T}}_{g,1}$) is homotopy equivalent to the curve complex $\mathcal{C}_g$ (resp. $\mathcal{C}_{g,1}$) constructed by Harvey in \cite{harvey2016boundary}. 
 
 In \cite{harer1986virtual}, Harer showed that $\mathcal{C}_{g,1}$ is $\Mod_{g,1}$-equivariantly homotopy equivalent to $\mathcal{C}_g$. This means that the virtual dualizing modules of $\Mod_g$ and $\Mod_{g,1}$ are the same and the $\Mod_{g,1}$-action on it is just induced by the surjective homomorphism $\Mod_{g,1}\longrightarrow\Mod_g$. In addition, in Harer's construction of the cellular decomposition of the Teichmüller space $\mathcal{T}_{g,1}$, he introduced the arc complex $\mathcal{A}(S_{g,n})$.
 \begin{df}
  Let $S_{g,n}$ be a genus $g$ closed surface with $n$ marked points. Then a $k$-arc system on it is a collection of isotopy classes (relative to the marked points) of $k$ arcs and closed loops $([\alpha_0],[\alpha_1],...,[\alpha_{k-1}])$ such that $\alpha_0,...,\alpha_{k-1}$ intersects only at the marked points. The arc complex $\mathcal{A}(S_{g,n})$ is the simplicial complex whose $k$-simplices are $k+1$-arc systems and the face relations are induced by inclusion.
 \end{df}
 
 An arc system $([\alpha_0],[\alpha_1],...,[\alpha_k])$ is a \emph{filling system} if $S_{g,n}\setminus \cup_i\alpha_i$ is a disjoint union of disks. An arc system is called \emph{oriented} if the arcs are ordered. We denote the closed subcomplex of $\mathcal{A}(S_{g,n})$ consisting of arc systems that do not fill $S_{g,n}$ by $\mathcal{A}_\infty(S_{g,n})$. Harer \cite{harer1986virtual} has shown that $\mathcal{A}_\infty(S_{g,1})$ is homotopy equivalent to $\mathcal{C}_g$. Since $\mathcal{A}(S_{g,1})$ is contractible, we have
 \[\St_g=\widetilde{H}_{2g-2}(\mathcal{A}_\infty(S_{g,1});\mathbb{Z})\simeq H_{2g-1}(\mathcal{A}(S_{g,1})/\mathcal{A}_\infty(S_{g,1});\mathbb{Z}) \]
 
Let $\mathcal{F}_j$ be the free abelian group generated by oriented filling systems which consist of $2g+j$ arcs. Note that a filling system has at least $2g$ arcs and $\mathcal{F}_\bullet$ is just the cellular chain complex $C_\bullet(\mathcal{A}(S_{g,1})/\mathcal{A}_\infty(S_{g,1});\mathbb{Z})$ shifted by $2g-1$. 

\begin{prop}[\cite{broaddus2012homology}]
We have a finite resolution of the Steinberg module $\St_g$:
\begin{align*}
    0\longrightarrow \mathcal{F}_{4g-3}\longrightarrow...\longrightarrow \mathcal{F}_1\longrightarrow\mathcal{F}_0\longrightarrow \St_g\longrightarrow 0,
\end{align*}
where the maps are cellular chain maps. Moreover, if we tensor the resolution by $\mathbb{Q}$, then we get a projective resolution of $\St_g\otimes_\mathbb{Z}\mathbb{Q}$.
\end{prop}
\begin{cor}[\cite{broaddus2012homology}]\label{cor:proj_resolution}
 Let $A$ be a $\Mod_g$- (or $\Mod_{g,1}$-) module. Then
 \begin{align*}
     H^i(\Mod_{g,1};\mathbb{Q}\otimes_\mathbb{Z}A)&\simeq H_{4g-3-i}((\mathcal{F}_\bullet\otimes_\mathbb{Z}\mathbb{Q})\otimes_{\Mod_{g,1}}A)\\
     H^i(\Mod_g;\mathbb{Q}\otimes_\mathbb{Z}A)&\simeq H_{4g-3-i}((\mathcal{F}_\bullet\otimes_\mathbb{Z}\mathbb{Q})\otimes_{\Mod_{g}}A).
 \end{align*}
\end{cor}

\subsection{Chord diagrams}
In order to keep track of the filling systems, it is convenient to use a chord diagram to represent them. Given an arc system on $S_{g,1}$, if we cut out a disk neighbourhood of the marked point and deform the cut arcs away from the marked point, then we will get a collection of chords of the boundary circle (for instance, see Figure~\ref{fig:filling_chord}).
\begin{figure}
\centering
    \begin{tikzpicture}
      \draw (-3.0,0).. controls (-2.9,1.7) and (1.9,1.7).. (2,0);
      \draw (-3.0,0).. controls (-2.9,-1.7) and (1.9,-1.7).. (2,0);
      \draw (-2.1,0).. controls (-1.7,0.4) and (-1.5,0.4).. (-1.0,0);
      \draw (-2.2,0.1).. controls (-1.7,-0.2) and (-1.4,-0.2).. (-0.9,0.1);
      \draw (1.1,0).. controls (0.7,0.4) and (0.5,0.4).. (0,0);
      \draw (1.2,0.1).. controls (0.7,-0.2) and (0.4,-0.2).. (-0.1,0.1);
      \draw[red] (-1.5,0) ellipse (1cm and 0.5cm);
      \draw[blue] (-0.5,0).. controls (-0.7,-0.1) and (-0.9,-0.1).. (-1.2,-0.1);
      \draw[blue,dotted] (-1.2,-0.1).. controls (-1.6,-0.4) and (-1.6,-1).. (-1.2,-1.3);
      \draw[blue] (-1.2,-1.3).. controls (-0.7,-0.7) and (-0.6,0).. (-0.5,0);
       \draw[brown] (-0.5,0).. controls (-0.3,-0.1) and (-0.1,-0.1).. (0.2,-0.1);
      \draw[brown,dotted] (0.2,-0.1).. controls (0.6,-0.4) and (0.6,-1).. (0.2,-1.3);
      \draw[brown] (0.2,-1.3).. controls (-0.3,-0.7) and (-0.4,0).. (-0.5,0);
      \draw[green] (0.5,0) ellipse (1cm and 0.5cm);
      \node at (-0.5,0) [circle,draw,fill,inner sep=0pt,minimum size=3pt,label=below:$p$]{};
      \begin{scope}[xshift=5cm]
         \draw (0,0) circle (2cm);
         \draw[red] (-1.9,0.6).. controls (-0.5,0.5) and (-0.2,-0.4).. (-0.7,-1.8);
         \draw[blue] (-1.8,-0.6).. controls (-0.2,0.0) and (0.3,-0.4).. (0.8,-1.8);
         \draw[brown] (1.9,-0.6).. controls (0.5,-0.5) and (0.2,0.4).. (0.7,1.8);
         \draw[green] (1.8,0.6).. controls (0.2,0.0) and (-0.3,0.4).. (-0.8,1.8);
      \end{scope}
    \end{tikzpicture}
    \caption{A filling system (left) and its chord diagram}\label{fig:filling_chord}
\end{figure}
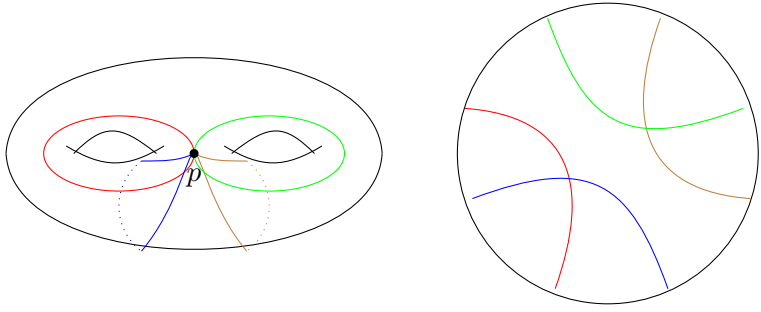
\begin{df}
 A \emph{labelled $k$-chord diagram} is a collection of $(2g+k)$ chords of an oriented circle such that we labelled them by isotopy classes of arcs on $S_{g,1}$. The orientation of the chords has to be compatible with that of the circle.
\end{df} 

Given an oriented filling system, one can construct a labelled chord diagram associated with it. We call a chord diagram \emph{proper} if it really comes from an oriented filling system. A labelled chord diagram with $2g+k$ chords is proper if and only if (cf. \cite{broaddus2012homology}):
\begin{itemize}
    \item there are no parallel edges (include chords and edges of the circle), i.e. no two edges can deform to another without changing the intersections of chords;
    \item there are $k$ cycles if we trace along the chords and edges (each chord and edge can be gone through twice) of the circle.
\end{itemize}
The advantage to work on chord diagrams is that $\Mod_{g,1}$ (or $\Mod_g$) acts transitively on the set of labelled chord diagrams of the same topological type (i.e. the underlying topological spaces of $1$-complexes are orientation preserving isomorphic). It will be convenient to make some argument about the generators of $\St_g$ by just considering the chord diagrams of some certain topological type.

Let $x_1,...,x_{2g}$ be a standard set of generators of the fundamental group of a genus $g$ closed surface. Let $\phi_{g}$ be the filling system with $2g$ arcs so that the associated chord diagram is as depicted in Figure~\ref{figure:chord1}, where
\begin{align*}
    a_1&=x_1\\
    a_{i}&=\begin{cases}
    x_{i} & \mbox{ if } i=2j\\
    x_{i} x_{i-2}^{-1}& \mbox{ if } i=2j+1
    \end{cases}
\end{align*}

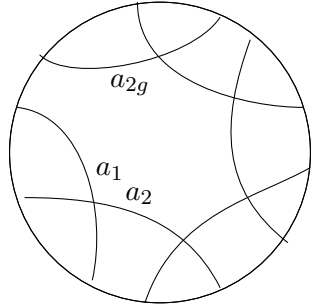
\begin{figure}
    \centering
  \begin{tikzpicture}
    \draw (0,0) circle (2cm);
         \draw (0,0) circle (2cm);
         \draw(-1.9,0.6).. controls (-1,0.5) and (-0.7,-0.8)..node[right]{$a_1$} (-0.9,-1.7);
         \draw(-1.8,-0.6).. controls (-0.7,-0.6) and (0.3,-0.8)..node[above]{$a_2$} (0.8,-1.8);
         \draw(-0.2,-2).. controls (0.2,-0.8) and (1.3,-0.6).. (2,-0.2);
         \draw (1.7,-1.2).. controls (0.8,-0.5) and (0.8,0.4).. (1.2,1.5);
         \draw (1.9,0.6).. controls (1.2,0.6) and (-0.3,0.9).. (-0.3,2);
          \draw (0.8,1.8).. controls (0.4,1.2) and (-1.2,0.9).. node[below]{$a_{2g}$}(-1.6,1.3);
  \end{tikzpicture}
    \caption{The chord diagram corresponding to filling system $\phi_g$}\label{figure:chord1}
\end{figure}

\begin{prop}[\cite{broaddus2012homology}]\label{prop:gen}
The Steinberg module $\St_g$ is generated by $[\phi_g]$ as a $\Mod_{g,1}$- (or $\Mod_g$-) module. Moreover, the class of any filling system whose chord diagram has more than one connected component of chords will be trivial in $\St_g$.
\end{prop}

\subsection{Proof of Theorem~\ref{thm:vanish}}
To prove Theorem~\ref{thm:vanish}, it is equivalent to proving that $H^{4g-2}(\Mod_{g,2};\mathbb{Q})=0$. According to \cite{birman1969mapping}, we have the following exact sequence relating the mapping class group $\Mod_{g,1}$ and $\Mod_{g,2}$:
\begin{align}
    1\longrightarrow \pi_1(S_{g,1}\setminus \{p_1\},p_2)\longrightarrow \Mod_{g,2}\longrightarrow \Mod_{g,1}\longrightarrow 0,
\end{align}
 where the first arrow is induced by point pushing and the second arrow is just forgetting a marked point. We can then apply the Lyndon-Hochschild-Serre spectral sequence of group cohomology to such exact sequence, i.e. there is a spectral sequence of cohomological type
 \begin{align*}
     E^{p,q}_2=H^p\big(\Mod_{g,1};H^q(\pi_1(S_{g,n}\setminus \{p_1\},p_2);\mathbb{Q})\big)\implies H^{p+q}(\Mod_{g,2};\mathbb{Q}).
 \end{align*}
 Note that $\pi_1(S_{g,n}\setminus \{p_1\},p_2)$ is a free group generated by $2g$ elements, and we have
 \begin{align*}
     H^{q}(\pi_1(S_{g,n}\setminus \{p_1\},p_2);\mathbb{Q})=
     \begin{cases}
     \mathbb{Q}& \mbox{ if } q=0\\
     \mathbb{Q}^{2g}& \mbox{ if } q=1\\
     0 & \mbox{ otherwise }
     \end{cases}
 \end{align*}
 For $g=1$, due to Eichler-Shimura isomorphism and the fact that there is no nontrivial odd weight modular form, we have $H^1(\SL(2,\mathbb{Z});\mathbb{Q}[x,y]_1)=0$, where $\mathbb{Q}[x,y]_1$ is the space of $\mathbb{Q}$-linear function on $\mathbb{Q}^2$. Thus, we only need to consider the case where $g>1$. 
 Since the virtual cohomological dimension of $\Mod_{g,1}$ is $4g-3$, the terms $E^{p,q}_2$ where $p+q=4g-2$ can only be non-trivial if $p=4g-3$ and $q=1$. Hence, we need to show that $H^{4g-3}(\Mod_{g,1};\mathbb{Q}^{2g})=0$. The action of $\Mod_{g,1}$ on $\mathbb{Q}^{2g}$ comes from the action of $\Mod_{g,2}$ on $\pi_1(S_{g,n}\setminus \{p_1\},p_2)$. It is obvious that any point pushing with respect to $p_1$ will act trivially on $\mathbb{Q}^{2g}$ which is isomorphic to the abelianization of $\pi_1(S_{g},p_1)$ tensored with $\mathbb{Q}$. Since we also have the exact sequence
 \begin{align*}
    1\longrightarrow \pi_1(S_{g},p_1)\longrightarrow \Mod_{g,1}\longrightarrow \Mod_g\longrightarrow 1, 
 \end{align*}
 the $\Mod_{g,1}$-module structure on $\mathbb{Q}^{2g}$ is induced by the $\Mod_g$-module structure on it via the map $\Mod_{g,1}\longrightarrow \Mod_g$. Moreover, by Corollary~\ref{cor:proj_resolution}, we have
 \begin{align*}
     H^{4g-3}(\Mod_{g,1};\mathbb{Q}^{2g})&\simeq H_0(\Mod_{g,1};\St_g\otimes_\mathbb{Z}\mathbb{Q}^{2g})\\
     &\simeq H_0(\Mod_g;\St_g\otimes_\mathbb{Z}\mathbb{Q}^{2g})\\
     &\simeq H_0(\mathcal{F}_\bullet\otimes_{\Mod_g}\mathbb{Q}^{2g})
 \end{align*}
 Thus, once we can show that $H_0(\mathcal{F}_\bullet\otimes_{\Mod_g}\mathbb{Q}^{2g})=0$, the theorem is proved. We end our proof with the following proposition.
 \begin{prop}
 Let $\mathbb{Q}^{2g}$ be the $\Mod_g$- (or $\Mod_{g,1}$-) module we described above. Then $H_0(\mathcal{F}_\bullet\otimes_{\Mod_g}\mathbb{Q}^{2g})=0$.
 \end{prop}
 \begin{proof}
For $g=1$, by the virtual duality and our reasoning above, we are done. The argument we give now is for $g>1$. Note that $H_0(\Mod_g;\St_g\otimes_\mathbb{Z}\mathbb{Q}^{2g})$ is just the coinvariant $(\St_g\otimes_\mathbb{Z}\mathbb{Q}^{2g})_{\Mod_g}$. Then by Proposition~\ref{prop:gen}, it is generated by $\{[\phi_g]\otimes x_i^\vee\}$ because $\{x_1,...,x_{2g}\}$ form a basis for $H_1(S_{g,1};\mathbb{Q})$. Here, $x_i^\vee$ refers to the intersection product $(-,x_i)$ with respect to $x_i$. Thus, it suffices to show that $\phi_g\otimes x_i^\vee$ lies in the image of the cellular chain map 
\begin{align*}
    \partial\otimes id:\mathcal{F}_1\otimes_{\Mod_g}\mathbb{Q}^{2g}\longrightarrow\mathcal{F}_0\otimes_{\Mod_g}\mathbb{Q}^{2g}
\end{align*} for every $i$. 

For $g>1$, we consider the filling system corresponding to Figure~\ref{figure:chord2}, where $a_1,...,a_{2g}$ as defined before and $a_{2g+1}=x_{2g-1}^{-1}$. The cellular chain map image will be \begin{align*}
    (\partial\otimes 1)(\phi_g'\otimes a_i^\vee)=\big(\sum_{j=1}^{2g+1}(-1)^{i-1}(a_1,...,\hat{a_j},...,a_{2g+1})\big)\otimes a_i^\vee
\end{align*}
\begin{figure}
    \centering
  \begin{tikzpicture}
    \draw (0,0) circle (2cm);
         \draw (0,0) circle (2cm);
         \draw(-1.9,0.6).. controls (-1,0.5) and (-0.7,-0.8)..node[right]{$a_1$} (-0.9,-1.7);
         \draw(-1.8,-0.6).. controls (-0.7,-0.6) and (0.3,-0.8)..node[above]{$a_2$} (0.8,-1.8);
         \draw(-0.2,-2).. controls (0.2,-0.8) and (1.3,-0.6).. (2,-0.2);
         \draw (1.7,-1.2).. controls (0.8,-0.5) and (0.8,0.4).. (1.2,1.5);
         \draw (1.9,0.6).. controls (1.2,0.6) and (-0.3,0.9).. node[below]{$a_{2g}$}(-0.3,2);
          \draw (0.8,1.8).. controls (0.4,1.2) and (-1.2,0.9).. node[below]{$a_{2g+1}$}(-1.6,1.3);
  \end{tikzpicture}
    \caption{The chord diagram corresponding to filling system $\phi_g'$}\label{figure:chord2}
\end{figure}
By Proposition~\ref{prop:gen}, the terms for $j=2,3,...,2g$ will be zero because the corresponding chord diagram will be disconnected. In addition, there is some $t\in \Mod_g$ mapping $a_{i}$ to $a_{i+1}$ for $i=1,...,2g$. Hence, for $i=1,2,...,2g$, we have
\begin{align*}
    (\partial\otimes 1)(\phi_g'\otimes a_i^\vee)=&(t\cdot\phi_g)\otimes a_i^\vee+\phi_g\otimes a_i^\vee\\
    = &\phi_g\otimes (t\cdot a_i)^\vee+\phi_g\otimes a_i^\vee\\
    =&\phi_g\otimes a_{i+1}^\vee+\phi_g\otimes a_i^\vee
\end{align*}
These yield for $i=1,...,2g+1$
\begin{align*}
    [\phi_g\otimes a_i^\vee]=(-1)^{i-1}[\phi_g\otimes a_1^\vee] \in (\St_g\otimes_\mathbb{Z}\mathbb{Q}^{2g})_{\Mod_g}.
\end{align*}
Then we have
\begin{align*}
    (g+1)[\phi_g\otimes x_1^\vee]=&[\phi_g\otimes x_1^\vee]+\bigg(\sum_{j=1}^{g-1}[\phi_g\otimes(x_{2j+1}^\vee-x_{2j-1}^\vee)]\bigg)+[-\phi_g\otimes x_{2g-1}^\vee]
    =0
\end{align*}
Thus, $[\phi_g\otimes x_1^\vee]=0$ and one can easily deduce that $[\phi_g\otimes x_i^\vee]=0$ for $i=1,...,2g$.
 \end{proof}

 \section{Preparatory material for the computation of spin stratum classes}\label{sec:prepare}
 
 The goal of this section is to provide technical details about how to compute the clutching pullback of a (spin) stratum class. In addition, we will explain why the clutching pullbacks will give us a solution under certain assumptions. In Section~\ref{sec:inj_pull} to Section~\ref{sec:excess_int}, we will explain how this pullback mechanism works, while in Section~\ref{sec:resolve_res} to Section~\ref{sec:horG}, we will demonstrate how to do the actual computation when we reduce to the case of lower genus or fewer marked points.
 
 \subsection{Injectivity of clutching pullbacks}\label{sec:inj_pull}
 
Recall that given a stable graph $\Gamma$ of genus~$g$ and $n$ legs, there is a natural morphism 
 \begin{align*}
     \xi_\Gamma:\overline{\mathcal{M}}_\Gamma= \prod_{v\in V(\Gamma)}\overline{\mathcal{M}}_{g(v),n(v)}\longrightarrow \overline{\mathcal{M}}_{g,n},
 \end{align*}
 which is called the clutching map. We have mentioned in the introduction that in order to compute the spin stratum classes $[\overline{\mathcal{H}}_g(\mu)]^{\spin}$, we will compute the clutching pullback of them and solve a system of linear equations. 
 
 The existence of a solution is clear. The uniqueness of the solution is due to the strengthened version of a result in \cite{arbarello1998calculating}. 
 
 \begin{prop}\label{prop:arbarello}
Consider the clutching maps $\xi_\Gamma:\overline{\mathcal{M}}_{\Gamma}\longrightarrow \overline{\mathcal{M}}_{g,n}$, where $\Gamma$ are stable graphs of genus $g$ and $n$-legs and with only one edge, the direct sum of pullback maps
\[\oplus_\Gamma\xi_\Gamma^*:\mathrm{H}^{j}(\overline{\mathcal{M}}_{g,n};\mathbb{Q})\longrightarrow \bigoplus_\Gamma \mathrm{H}^{j}(\overline{\mathcal{M}}_\Gamma;\mathbb{Q})\] is injective if $j\leq d(g,n)$, where
\begin{align}\label{eq:range_inj}
    d(g,n)=
    \begin{cases}
    2n-7 &\mbox{ if } g=0\\
    2g-1 & \mbox{ if } n=0,1\\
    2g & \mbox{ if } n=2\\
    2g-3+n & \mbox{ otherwise } 
    \end{cases}
\end{align}
\end{prop}
 
\begin{proof}
The first thing different from the original version is the first row of $d(g,n)$. Indeed, according to \cite{keel1992intersection}, the cohomology ring $\mathrm{H}^*(\overline{\mathcal{M}}_{0,n};\mathbb{Q})$ is generated by cycle classes of boundary strata. In particular, the cohomology groups of odd degrees are trivial. Moreover, a cycle class of a boundary stratum of $\overline{\mathcal{M}}_{0,n}$ can be written as a product of classes of boundary divisors. If a class $x$ of even codimension $j\leq 2n-8$ lies in the kernel of the sum of clutching pullbacks, then any intersection product $x\cdot \prod_{i=1,...,j}[D_{\Gamma_i}]$ will also be zero because $x\cdot [D_\Gamma]= \xi_{\Gamma *}\xi_\Gamma^*(x)$. Due to Poincaré duality, this will force that $x=0$.

The other thing different from the original version is the second and third rows of the values of $d(g,n)$. According to Theorem 1 in \cite{church2012rational} and our proof in Section~\ref{thm:vanish}, $\mathrm{H}^{4g-5}(\mathcal{M}_{g},\mathbb{Q})=\mathrm{H}^{4g-3}(\mathcal{M}_{g,1},\mathbb{Q})=\mathrm{H}^{4g-2}(\mathcal{M}_{g,2},\mathbb{Q})=0$. This implies that $\mathrm{H}^{2g-1}_c(\mathcal{M}_{g};\mathbb{Q})=\mathrm{H}^{2g-1}_c(\mathcal{M}_{g,1};\mathbb{Q})=\mathrm{H}^{2g}_c(\mathcal{M}_{g,2};\mathbb{Q})=0$. Then by the long exact sequence:
\begin{align*}
    ...\longrightarrow \mathrm{H}^j_c(\mathcal{M}_{g,n};\mathbb{Q})\longrightarrow \mathrm{H}^j(\overline{\mathcal{M}}_{g,n};\mathbb{Q})\longrightarrow \mathrm{H}^j(\partial\mathcal{M}_{g,n};\mathbb{Q})\longrightarrow ...,
\end{align*}
we can conclude that $\mathrm{H}^j(\overline{\mathcal{M}}_{g,n};\mathbb{Q})\longrightarrow \mathrm{H}^{j}(\partial\mathcal{M}_{g,n};\mathbb{Q})$ is injective if $j\leq d(g,n)$. 

Then, by the same argument of Lemma 2.6 in \cite{arbarello1998calculating}, it yields that 
\begin{align*}
   \oplus_\Gamma\xi_\Gamma^*:\mathrm{H}^{j}(\overline{\mathcal{M}}_{g,n};\mathbb{Q})\longrightarrow \bigoplus_\Gamma \mathrm{H}^{j}(\overline{\mathcal{M}}_\Gamma;\mathbb{Q})
\end{align*}
is injective for $j\leq d(g,n)$.
\end{proof}

\begin{rem}
If $\overline{\mathcal{H}}_g(\mu)$ is a meromorphic stratum, then the stratum classes $[\overline{\mathcal{H}}_g(\mu)]$ lie in the degree $2g$ component of the tautological ring $\RH^{2g}(\overline{\mathcal{M}}_{g,n})\subseteq \mathrm{H}^{2g}(\overline{\mathcal{M}}_{g,n})$; if $\overline{\mathcal{H}}_g(\mu)$ is a holomorphic stratum, then the stratum class lies in $\RH^{2g-2}(\overline{\mathcal{M}}_{g,n})\subseteq \mathrm{H}^{2g-2}(\overline{\mathcal{M}}_{g,n})$. Thus, for meromorphic strata and $n=2$, the codimension of the (spin) stratum class will be out of the range of injectivity $d(g,n)=2g-3+n$ in the original proposition in \cite{arbarello1998calculating}. That is why we need to improve $d(g,2)$ to $2g$. 
\end{rem}

In practice, computing the clutching pullback of a stratum class with respect to the self-node graph $\Gamma_0$ will be very difficult, especially if our signature $\mu$ already contains a pair of simple poles. We will illustrate the problem in Section~\ref{sec:simpole}. Hence, we want to assume that the injectivity is still guaranteed if we neglect the clutching pullback with respect to $\Gamma_0$ under the condition $n\geq 3$. 
\begin{assu}
\label{prop:assum}
For $g>0,n\geq 3$, the pullback map 
\begin{align*}
   \oplus_{\Gamma\neq \Gamma_0}\xi_\Gamma^*:\mathrm{H}^{k}(\overline{\mathcal{M}}_{g,n})\longrightarrow \bigoplus_\Gamma \mathrm{H}^{k}(\overline{\mathcal{M}}_\Gamma) 
\end{align*} is injective if $k\leq 2g-3+n$.
\end{assu}
This assumption is tested to be true for the cases of $(g,n)$ equal
\begin{align*}
    (1,3),(1,4),(1,5),(1,6),(2,3),(2,4),(2,5),(3,3),(3,4), (4,3)
\end{align*}
for which our computer is capable to do the calculations by \texttt{admcycles} in a reasonable time.

\subsection{Clutching Pullback Formula}\label{sec:excess_int}
In this subsection, we aim to establish a formula for doing intersection theory with the (spin) stratum class. Recall that in the introduction we mentioned that we want to compute the clutching pullback
\[\xi_\Gamma^*\big([\overline{\mathcal{H}}_g(\mu)]^{\spin}\big)= \xi_\Gamma^*p_*\big([\mathbb{P}\overline{\mathcal{M}}_{g,n}(\mu)]^{\spin}\big),\]
where $\Gamma$ is a one-edge stable graph of genus $g$ and $n$ leg. We will first proceed to derive the formula for the stratum class, and then give the spin variant of it. 


We will apply the excess intersection formula (see for example \cite{fulton2016intersection} Prop. 17.4.1) to derive our pullback formula for $\xi_\Gamma^*$. First, we have to identify the fiber product
 
 \begin{center}
   \begin{tikzcd}
 \mathcal{F}_{\Gamma,\mu}\arrow[r,"t_1"]\arrow[d,"t_2"]& \mathbb{P}\Xi\overline{\mathcal{M}}_{g,n}(\mu) \arrow[d,"p"]\\
 \overline{\mathcal{M}}_{\Gamma} \arrow[r,"\xi_\Gamma"] & \overline{\mathcal{M}}_{g,n}
 \end{tikzcd}  
 \end{center}
 and the top Chern class of the excess normal bundle $E=t_2^*\mathcal{N}_{\Gamma}/\mathcal{N}_{t_1}$, where $\mathcal{N}_{\Gamma}$ and $\mathcal{N}_{t_1}$ are the normal bundles with respect to $\iota_\Gamma$ resp. $t_1$. Then \[\xi_\Gamma^*p_*\big([\mathbb{P}\Xi\overline{\mathcal{M}}_{g,n}(\mu)]\big)=t_{2*}c_{\textrm{top}}(E).\]
 
 Let $\Delta$ and $\Gamma$ be stable graphs. A \emph{graph contraction} of $\Delta$ to $\Gamma$ is defined as a pair of maps between the set of half edges (i.e. $H(\Delta)$ and $H(\Gamma)$) resp. the set of vertices (i.e.$V(\Delta)$ and $V(\Gamma)$):
\[\bigg(f_v:V(\Delta)\twoheadrightarrow V(\Gamma),\quad f_h:H(\Gamma)\hookrightarrow H(\Delta)\bigg), \]
satisfying the following conditions:
\begin{itemize}
    \item[(i)] For any half edge $h_0\in H(\Gamma)$ at a vertex $v_0\in V(\Gamma)$, the image $f_h(h_0)\in H(\Delta)$ should be at a vertex in $f_v^{-1}(v_0)$,
    \item[(ii)] the map $f_h$ respects markings and edges are mapped to edges
    \item[(iii)] for any vertex $v_0\in V(\Gamma)$, the subgraph in $\Delta$ consisting of all the vertices in $f_v^{-1}(v_0)$ and the half edges adjacent to them is a connected stable graph of genus $g(v_0)$ and $|H(\Gamma)_{v_0}|$-markings.
\end{itemize}

A \emph{$\Gamma$-structure} on $\Delta$ is just a graph contraction of $\Delta$ to $\Gamma$. 

 
\begin{prop}\label{prop:excess_int}
Given a signature $\mu=(m_1,...,m_n)$ such that $\sum m_i=2g-2$, then for a one-edge stable graph $\Gamma$, we have the equality
\begin{equation}
\label{eq:excess}
    \begin{split}
      \xi_\Gamma^*\big([\overline{\mathcal{H}}_g(\mu)]\big)
=&\sum_{\Gamma_H\overset{f}{\rightsquigarrow}\Gamma}\frac{1}{|\Aut(\Gamma)|}\xi_{f*}p_{\Gamma}\bigg([\mathbb{P}\Xi\overline{\mathcal{M}}_{\Gamma_H}]\bigg)+\\
&\sum_{\substack{\Delta\overset{f}{\rightsquigarrow}\Gamma\\ \Delta\in \LG_1(\mu)}}\frac{1}{|\Aut(\Delta)|}\cdot \frac{\prod_{e\in E(\Delta)}\kappa_e}{\kappa_{f}}\cdot  \xi_{f*}\bigg(\pi_\top^*p^{\top}_*[\mathbb{P}\Xi\overline{\mathcal{M}}_{\Delta,\top}]\cdot \pi_\bot^*p^{\bot}_*[\mathbb{P}\Xi\overline{\mathcal{M}}_{\Delta,\bot}]\bigg)
\end{split}
\end{equation}
where $\Gamma_H$ is the horizontal level graph whose underlying graph is $\Gamma$ and $\kappa_f$ is the enhancement of the image edge in $\Delta$ corresponding to the single edge in $\Gamma$ via $f$.
\end{prop}

 Before we prove the above formula, we have to understand the fiber product $\mathcal{F}_{\Gamma,\mu}$. 
 \begin{df}
  Let $\pi:\mathcal{X}\longrightarrow B$ be a family of stable curves such that the dual graph of every fiber is some degeneration of $\Gamma$. Then a $\Gamma$-marking on $\mathcal{X}$ is the following data:
  
  \begin{itemize}
      \item a collection of $|E(\Gamma)|$ sections $\sigma_1,...,\sigma_{|E(\Gamma)|}$ with image in the singular locus of $\pi$;
      \item a collection of $|H(\Gamma)|$ sections $\sigma_{1,1},\sigma_{1,2},\sigma_{2,1},...,\sigma_{|E(\Gamma)|,2}$ of the normalization $\widetilde{\mathcal{X}}$ of $\mathcal{X}$;
      \item a collection of $|V(\Gamma)|$ connected components of $\mathcal{X}\setminus \ \cup_{i=0}^{|E(\Gamma)|}\sigma_i$
  \end{itemize} 
  such that each fiber nodal curve is marked according to $\Gamma$.
 \end{df}
 
 We can regard $\overline{\mathcal{M}}_\Gamma$ as the stack of families of stable curves with a $\Gamma$-marking. As a result, one can easily see that the fiber product $\mathcal{F}_{\Gamma,\mu}$ is just the moduli stack of families of multi-scale differentials $(\mathcal{X},\boldsymbol{\eta})$ with a $\Gamma$-marking on $\mathcal{X}$. Let $\LG^1(\mu)$ be the set of level graphs compatible with signature $\mu$ such that the codimension of $D_\Delta$ is $1$.
 

\begin{lm}
The natural morphism 
\begin{align*}
    j:=\sqcup j_f:\coprod_{\substack{\Delta\in \LG^1(\mu) \\ \Delta\overset{f}{\rightsquigarrow}\Gamma}}D_{\Delta,f}\longrightarrow \mathcal{F}_{\Gamma,\mu},
\end{align*} where $D_{\Delta,f}$ is the stack of families of multi-scale differentials with a $\Delta$-marking on the underlying family of curves and a $\Gamma$ structure $f$ on $\Delta$ being specified, is finite, flat and surjective.
\end{lm}
\begin{proof}
Notice that a family of multi-scale differentials $(\mathcal{X},\boldsymbol{\eta})$ compatible with a level graph $\Delta$ whose underlying family of stable curves has a partial normalization according to $\Gamma$ if and only if there is a $\Gamma$-structure on $\Delta$. The morphism
\begin{align*}
    \coprod_{\substack{\Delta\in\LG^1(\mu)\\ \Delta\rightsquigarrow\Gamma}} D_{\Delta}\times_{\overline{\mathcal{M}}_{g,n}} \overline{\mathcal{M}}_\Gamma\longrightarrow \mathcal{F}_{\Gamma,\mu}
\end{align*}

is \'etale. Moreover,

\begin{align*}
    D_{\Delta}\times_{\overline{\mathcal{M}}_{g,n}}\overline{\mathcal{M}}_\Delta \times_{\overline{\mathcal{M}}_{g,n}}\overline{\mathcal{M}}_\Gamma\longrightarrow D_{\Delta}\times_{\overline{\mathcal{M}}_{g,n}} \overline{\mathcal{M}}_\Gamma
\end{align*}
is finite, flat and surjective. Note that 
\begin{align*}
    \overline{\mathcal{M}}_\Delta \times_{\overline{\mathcal{M}}_{g,n}}\overline{\mathcal{M}}_\Gamma\simeq \coprod_{\Delta\overset{f}{\rightsquigarrow}\Gamma}\overline{\mathcal{M}}_\Delta
\end{align*}
Hence, $j$ is finite, flat and surjective. 

\end{proof}

\begin{lm}\label{lm:formula}
The zeroth Chern class of the excess bundle is
\begin{align*}
c_0(t_2^*\mathcal{N}_\Gamma/\mathcal{N}_{t_1})=\sum_{\substack{\Delta\in\LG^1(\mu)\\ \Delta\rightsquigarrow\Gamma}}\frac{1}{|\Aut(\Delta)|}j_{f,*}\bigg(c_0(j_f^*t_2^*\mathcal{N}_\Gamma/\mathcal{N}_{t_1\circ j_f}) \bigg) 
\end{align*}

\end{lm}
\begin{proof}
First note that 
\begin{align*}
    j_f^*(t_2^*\mathcal{N}_\Gamma/\mathcal{N}_{t_1})=j_f^*t_2^*\mathcal{N}_\Gamma/j_f^*\mathcal{N}_{t_1}
\end{align*}
and $j_f^*\mathcal{N}_{t_1}=\mathcal{N}_{t_1\circ j_f}$. In addition, if $Q$ is an irreducible component of $\mathcal{F}_{\Gamma,\mu}$, then for any $[C]\in \CH^0(Q)$,
\begin{align*}
    j_*j^*([C])=\deg_Q(j)\cdot[C].
\end{align*}
Since the degree of the morphism $j_f$ is just $|\Aut(\Delta)|$, it yields the formula above.

\end{proof}
Now we can give the proof of Proposition~\ref{prop:excess_int}.
\begin{proof}[Proof of Proposition~\ref{prop:excess_int}] 
We first consider those $\Delta$ which are two level graphs, i.e. $\Delta\in\LG_1$, then we consider the case that $\Delta=\Gamma_H$. Let $(h,h')$ be the edge ($h$ and $h'$ are the half-edges) of the one-edge graph $\Gamma$ and $\mathcal{L}_h$ be the section pullback of the dualizing sheaf of $\overline{\mathcal{M}}_\Gamma$. Then the normal bundle $\mathcal{N}_{\xi_\Gamma}$ is just $\mathcal{L}^\vee_h\otimes\mathcal{L}^{\vee}_{h'} $ (cf. \cite{Arbarello}). According to the proof of Proposition 7.2 in \cite{costantini2020chern}, there is a short exact sequence of coherent sheaves on $D_{\Delta,f}$ (at least outside a subvariety of
codimension two) 
\begin{align*}
   0\longrightarrow \mathcal{N}_{t_1\circ j_f}^{\otimes\ell(\Delta)/\kappa_f}\longrightarrow j_f^*t_2^*\mathcal{N}_{\xi_\Gamma}\longrightarrow \mathcal{Q}\longrightarrow 0 ,
\end{align*}
where $ \mathcal{N}_{t_1\circ j_f}$ is the normal bundle of $D_{\Delta,f}\longrightarrow \mathbb{P}\Xi\overline{\mathcal{M}}_{g,n}(\mu)$. Here, $\ell(\Delta)$ is the l.c.m of the enhancements of the vertical edges of $\Delta$. As a result, $\mathcal{Q}$ has rank $0$ and by default the zeroth Chern class of $\mathcal{Q}$ is just $1$. Hence, the zeroth Chern class of $j_f^*t_2^*\mathcal{N}_{\xi_\Gamma}/\mathcal{N}_{t_1\circ j_f}$ is just $\frac{\ell(\Delta)}{\kappa_f}$.

If $\Delta=\Gamma_H$, then we have
\begin{align*}
    \mathcal{N}_{t_1\circ j_f} \simeq j_f^*t_2^*\mathcal{N}_{\xi_\Gamma} 
\end{align*} 
Hence, the zeroth Chern class of the excess bundle on this component will be just $1$. 

By Lemma~\ref{lm:formula} and due to the fact that $t_2\circ j_f= \xi_{f}\circ p_\Delta$, we have 

\begin{align*}
    t_{2,*}\bigg(c_0(t_2^*\mathcal{N}_\Gamma/\mathcal{N}_{t_1})\bigg)=&\sum_{\Gamma_H\overset{f}{\rightsquigarrow}\Gamma}\frac{1}{|\Aut(\Gamma)|}\xi_{f*}p_{\Gamma*}\bigg([D_{\Gamma_H,f}]\bigg)+\\
&\sum_{\substack{\Delta\overset{f}{\rightsquigarrow}\Gamma\\ \Delta\in \LG_1(\mu)}}\frac{1}{|\Aut(\Delta)|}\cdot \frac{\ell(\Delta)}{\kappa_{f}}\cdot  \xi_{f*}p_{\Delta*}\bigg([D_{\Delta,f}]\bigg) 
\end{align*}

Note that the degree of the morphism $p_\Delta$ is just the number of equivalence class of prong matchings which is just $\frac{\prod_{e\in E(\Delta)}\kappa_e}{\ell(\Delta)}$. Since the image of $p_\Delta$ is just $p^{\top}(\mathbb{P}\Xi\overline{\mathcal{M}}_{\Delta,\top})\times p^{\bot}(\mathbb{P}\Xi\overline{\mathcal{M}}_{\Delta,\bot})$, one gets the Clutching Pullback Formula~\ref{eq:excess}.
\end{proof}

By the same method as in Corollary~\ref{cor:1.3} to determine the spin variant of the projection pushforward of divisor associated to a two-level graph, we now state the spin version of (\ref{eq:excess}):
\begin{prop}
Given a one-edge stable graph $\Gamma$, the Clutching Pullback Formula for the spin stratum class is:
\begin{equation}
\label{eq:excess_p}
    \begin{split}
      \xi_\Gamma^*p_*\big([\mathbb{P}\Xi\overline{\mathcal{M}}_{g,n}(\mu)]^{\spin}\big)
=&\sum_{\Gamma_H\overset{f}{\rightsquigarrow}\Gamma}\frac{1}{|\Aut(\Gamma)|}\xi_{f*}p_{\Gamma*}\bigg([\mathbb{P}\Xi\overline{\mathcal{M}}_{\Gamma_H}]^{\spin}\bigg)+\\
&\sum_{\substack{\Delta\overset{f}{\rightsquigarrow}\Gamma\\ \Delta\in \LG_1^{odd}(\mu)}}\frac{1}{|\Aut(\Delta)|}\cdot \frac{\prod_{e\in E(\Delta)}\kappa_e}{\kappa_f}\cdot  \\ &\xi_{f*}\bigg(
\pi_\top^*p_{\top*}[\mathbb{P}\Xi\overline{\mathcal{M}}_{\Delta,\top}]^{\spin}\cdot \pi_\bot^*p_{\bot*}[\mathbb{P}\Xi\overline{\mathcal{M}}_{\Delta,\bot}]^{\spin}\bigg)
\end{split}
\end{equation}
\end{prop}

In Section~\ref{sec:exa_clutch}, we will give an example to illustrate how to use the clutching pullback formula to compute the clutching pullbacks of a stratum class. Then later in Section~\ref{sec:exa_spin_pull}, we will demonstrate the use of the clutching pullback formula to compute the clutching pullbacks of a spin stratum class and reconstruct that spin stratum class.

\subsection{Resolving the residue conditions $\mathfrak{R}$ of a generalised stratum}\label{sec:resolve_res}

Assume that we are given a signature $\mu$ and the spin stratum class of the level strata in the spaces~$\overline{\mathcal{M}}_{\Delta,\top}$ and $\overline{\mathcal{M}}_{\Delta,\bot}$. Then by applying the Clutching Pullback Formula~\eqref{eq:excess_p} on each one-edge graph, we get a system of linear equations. According to Proposition~\ref{prop:arbarello}, this system of linear equations has a unique solution which is the spin stratum class. Hence, it is crucial to compute the spin stratum class of a level stratum. 

 Due to the global residue conditions, the generalised stratum associated to the lower level can have residue conditions. For such a (spin) stratum class, it is not possible to use the method of solving a system of linear equations from clutching pullbacks. Namely, the codimension of the (spin) stratum class is too high such that it is outside the range in Proposition~\ref{prop:arbarello}. Nevertheless, the (spin) stratum class of a stratum with residue conditions is closely related to the (spin) stratum class of its ambient stratum which has one residue condition less. We can actually recursively compute the spin stratum class of a stratum with residue conditions from the spin stratum class of a stratum with no residue conditions.




Our main tools to resolve the residue conditions of a compactified generalised stratum are the following two propositions which are originally from \cite{sauvaget2019cohomology}, and proved in \cite{costantini2020chern} for the versions on the compactified generalised stratum. 

Let $\boldsymbol{g},\boldsymbol{n},\boldsymbol{\mu}$ be tuples of integers. On a (projectivized) generalised stratum $\pmoxrmoduli$ that admits spin structure, every (algebraic/cohomological) class can be written as the sum of summands carrying the even and odd spins. Hence, we can in general define the spin variant of a class. For example, let $\xi$ be the class $c_1(\mathcal{O}(1))$ on $\pmoxrmoduli$, and $\psi_{(\nu,i)}$ be the $\psi$-class at the $i$-th marked point on the component $\nu$. We write $\xi^{\spin}$ and $\psi^{\spin}_{(\nu,i)}$ for the spin variants of $\xi$- and $\psi$-classes.

Assume that after we remove one residue condition from the set of residue conditions inducing $\mathfrak{R}$, the new residue space $\mathfrak{R}_0$ is strictly larger than conditions $\mathfrak{R}_0$, then $\pmoxrmoduli$ will be a subvariety of codimension one in $\mathbb{P}\Xi\overline{\mathcal{M}} ^{\mathfrak{R}_0}_{\boldsymbol{g},\boldsymbol{n}}(\boldsymbol{\mu}) $. 

\begin{prop}
\label{prop:res_resl}
	The class of the stratum $\pmoxrmoduli$ with residue condition $\mathfrak{R}$ compares inside the Chow ring of the generalized stratum $\overline{B}= \mathbb{P}\Xi\overline{\mathcal{M}} ^{\mathfrak{R}_0}_{\boldsymbol{g},\boldsymbol{n}}(\boldsymbol{\mu}) $ to the class~$\xi$ by the formula
	\[[\pmoxrmoduli]=-\xi -\sum_{\Delta\in\LG^\mathfrak{R}_1(\overline{B})}\ell_\Delta [D_\Delta] \]
	where $\LG^\mathfrak{R}_1(\overline{B}) $ is the set of two-level graphs such that the removed residue condition from $\mathfrak{R}$ will induce no extra condition on the top level. In particular, if both $\pmoxrmoduli$ and $\overline{B}$ admit spin structure, then we have
    	\[[\pmoxrmoduli]^{\spin}=-\xi^{\spin} -\sum_{\Delta\in\LG^\mathfrak{R}_1(\overline{B})}\ell_\Delta [D_\Delta]^{\spin}\]
\end{prop}

The following proposition tells us that $\xi$-class can actually be expressed as $\psi$-class plus boundary terms. This enable us easily compute the projection pushforward of the (spin) fundamental class of a stratum to $\overline{\mathcal{M}}_{g,n}$.

\begin{prop}\label{prop:xi}
The class $\xi$ on $\overline{B}=\pmoxrmoduli$ can be expressed as 
\[\xi=(m_{\nu,i}+1)\psi_{(\nu,i)}-\sum_{\Delta\in {}_{(\nu,i)}\LG_1(\overline{B})} \ell_\Delta [D_\Delta]\]
where 	${}_{(\nu,i)}\LG_1(\overline{B}) $ is the set of two-level graphs with the leg $(\nu,i)$ on the lower level; $m_{\nu,i}$ is the order of singularity at the leg $(\nu,i)$. In particular, if $\overline{B}$ admits spin structure, then 

\[\xi^{\spin}=(m_{\nu,i}+1)\psi_{(\nu,i)}^{\spin}-\sum_{\Delta\in {}_{(\nu,i)}\LG_1(\overline{B})} \ell_\Delta [D_\Delta]^{\spin}\]
\end{prop}

The spin variants of the above propositions require that all the involved strata admit spin structure. However, if we want to resolve a residue relation on a pair of simple poles (i.e. of the form $r_1+r_2=0$), then the ambient stratum will no more admit a spin structure and the spin variant of the above propositions cannot apply. In the recursion of computing the spin stratum classes, there is a case where we cannot directly apply the clutching pullback method so that we have to resolve the residue condition on a pair of simple poles. This will be handled in Proposition~\ref{prop:sim_res_resolve} of Section~\ref{sec:simpole}.

By Proposition~\ref{prop:xi}, Proposition~\ref{prop:res_resl} and Proposition~\ref{prop:sim_res_resolve}, computing the projection pushforward of (the spin variant of) the fundamental class of a generalised stratum with residue conditions can be boiled down to computing stratum classes of smaller genus or fewer markings. Hence, now we basically have all the ingredients one would expect to carry out the recursive computation of the spin stratum classes. However, in fact, we still have some problems in some special cases if we really start the recursion. These will be handled in the next two subsections.

\subsection{Clutching pullback of $\Gamma_0$}\label{sec:non_comp}

In this subsection, we will explain why we in general want to avoid computing the clutching pullback with respect to the self-loop graph $\Gamma_0$. To compute the clutching pullback $\xi_{\Gamma_0}$ of the (spin) stratum class of signature $\mu$, we inevitably need to know the stratum class $p_*[\mathbb{P}\Xi\overline{\mathcal{M}}_{g-1,n+2}^\mathfrak{R}(\mu')]\in \mathrm{H}^{2g}(\overline{\mathcal{M}}_{g-1,n+2})$, where $\mu'=(m_1,...,m_n,-1,-1)$ and $\mathfrak{R}$ is induced by the residue condition $r_{n+1}+r_{n+2}=0$. 

If we recursively compute the clutching pullback of a stratum class with respect to $\Gamma_0$ , we will get into a bottleneck. Namely, we will end up with some level stratum that actually represents a complicated horizontal level graph. The following example may give some insight into it.

\begin{exa}\label{exa:recur_clutch_0}
Let $\mu=(6)$. Then by repeating the clutching pullback of non-separable $\Gamma$, we will end up computing the spin stratum class $p_*[\mathbb{P}\Xi\overline{\mathcal{M}}_{0,9}^\mathfrak{R}(6,-1^8)^{\spin}]$, where $\mathfrak{R}$ consists of the residue conditions
\[r_2+r_3=0,\quad r_4+r_5=0, \quad r_6+r_7=0 \]
Notice that the residue conditions above imply $r_8+r_9=0$. One of the vertical two-level graphs of that stratum will be as the following:
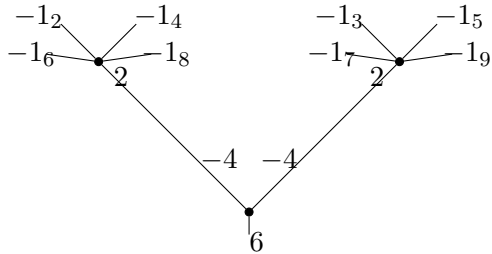
\begin{center}
    \begin{tikzpicture}
    
        \node (a) at (-1,0) [circle,draw,fill,inner sep=0pt,minimum size=3pt]{};
    \node (b) at (1,-2) [circle,draw,fill,inner sep=0pt,minimum size=3pt]{};
    \node (c) at (3,0) [circle,draw,fill,inner sep=0pt,minimum size=3pt]{};
    \path [-] (a) edge (b);
    \path [-] (c) edge (b);
    \path [-] (a) edge (-1.7,0.1);
    \path [-] (a) edge (-1.5,0.5);
    \path [-] (a) edge (-0.3,0.1);
    \path [-] (a) edge (-0.5,0.5);
    \path [-] (c) edge (2.3,0.1);
    \path [-] (c) edge (2.5,0.5);
    \path [-] (c) edge (3.7,0.1);
    \path [-] (c) edge (3.5,0.5);
    \node at (-1.8,0.6){$-1_2$} ;
    \node at (-0.2,0.6){$-1_4$} ;
    \node at (-1.9,0.1){$-1_6$} ;
    \node at (-0.1,0.1){$-1_8$} ;
    \node at (2.2,0.6){$-1_3$} ;
    \node at (3.8,0.6){$-1_5$} ;
    \node at (2.1,0.1){$-1_7$} ;
    \node at (3.9,0.1){$-1_9$} ;
    \node at (-0.7,-0.2){$2$} ;
    \node at (2.7,-0.2){$2$} ;
    \node at (0.6,-1.3){$-4$} ;
    \node at (1.4,-1.3){$-4$} ;
    \path [-] (b) edge (1,-2.3);
    \node at (1.1,-2.4){$6$};
    \end{tikzpicture}
\end{center}
 The top level stratum illustrates the following horizontal level graph

 
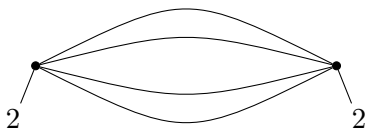
\begin{center}
    \begin{tikzpicture}
    
        \node (a) at (-1,0) [circle,draw,fill,inner sep=0pt,minimum size=3pt]{};
   
    \node (c) at (3,0) [circle,draw,fill,inner sep=0pt,minimum size=3pt]{};

    \draw (a) .. controls (1,-0.5).. (c) ;
     \draw (a) .. controls (1,0.5).. (c) ;
     \draw (a) .. controls (1,1).. (c) ;
      \draw (a) .. controls (1,-1).. (c) ;
    \node at (-1.3,-0.7){$2$} ;
    \node at (3.3,-0.7){$2$} ;
    \path [-] (a) edge (-1.2,-0.5);
    \path [-] (c) edge (3.2,-0.5);
    \end{tikzpicture}
\end{center}
for which we do not have a method to reduce the computation of the spin stratum class to the computation of the spin stratum class of each component.

\end{exa}
 
However, we expect that we will not always need to consider the clutching pullback of $\Gamma_0$. In the practice of application of Theorem~\ref{prop:arbarello} we observe that in actual computation for $n\geq 3$ the common kernel of all clutching pullbacks except that with respect to $\Gamma_0$ is already trivial. Although we cannot give a proof of Assumption~\ref{prop:assum}, we can propose a recursion scheme based on this assumption.

Nevertheless, we still need to handle the spin stratum classes of strata of differentials that have simple poles. Now if we want to construct the recursion, then the base cases of our recursion will be for $g=0$:
\begin{itemize}
    \item[(i)] $\mu$ does not contain any $-1$,
    \item[(ii)] or $\mu=(2k,-2k,-1,-1)$, 
    \item[(iii)] or $\mu=(0,-1,-1)$.
\end{itemize}
Case (i) is obvious. Since there is no symplectic basis on a sphere, the spin will just be taken as $0$. Case (ii) and (iii) will be discussed in the next subsection based on the perspective of flat surfaces.

\subsection{Spin stratum classes of signatures containing simple poles}\label{sec:simpole}

In this subsection, we will show how to handle the stratum of even type up to pairs of simple poles with residue condition relating them. According to \cite{boissy2015connected}, we can view a meromorphic differential as a non-compact flat surface, i.e. a surface obtained by glueing infinite half-cylinders and infinite half-planes on $\mathbb{R}^2$ such that there is no boundary. 

A pair of simple poles with the residue condition $r_1+r_2=0$ relating them can be visualized as two infinite half-cylinders of the same width but in opposite direction. Any flat surface $(X,\omega)$ with simple poles (which are paired up by residue conditions) can canonically form a flat surface $(X',\omega')$ of larger genus by cutting and glueing the half infinite cylinders. If the signature $\mu$ contains only even numbers and $-1$s (paired up by residue conditions), then the \emph{spin parity} of $(X,\omega)$ is defined to be the spin parity of $(X',\omega')$. The following two propositions are some easy results on the spin stratum classes of stratum that has genus $0$.

\begin{prop}\label{prop:base_case1}
The flat surface of signature $\mu=(0,-1,-1)$ has odd parity.
\end{prop}
\begin{proof}
    After cutting and glueing the infinite half-cylinders, we get a flat torus. A flat torus always has odd parity.
\end{proof}

\begin{prop}\label{prop:base_case2}
\label{prop:genus0}
Let $k\in \mathbb{Z}_{\geq 0}$. The spin stratum class of $\overline{\mathcal{H}}^\mathfrak{R}(2k,-2k,-1,-1)$, where $\mathfrak{R}$ is induced by the residue condition $r_3+r_4=0$, is equal to $\psi_1$.
\end{prop}

\begin{proof}
First, note that $\mathbb{P}\Omega\mathcal{M}_{0,4}^\mathfrak{R}(2k,-2k,-1,-1)$ has dimension $0$. Hence, we just need to count the number of such flat surfaces of even parity and of odd parity. Consider a flat surface corresponding to the signature $(2k,-2k,-1,-1)$ with residue space $\mathfrak{R}$. Then due to the realization of non-compact flat surfaces by half-planes and cylinders that were constructed in \cite{boissy2015connected}, we have $2(2k-1)$ half planes and $2$ half- infinite cylinders of opposite directions and the same width. Then by computing the spin parities of such flat surfaces where the half infinite cylinders are glued to different half-planes, we will know that the number of flat surfaces of even spin will be larger than that of odd spin by $1$.
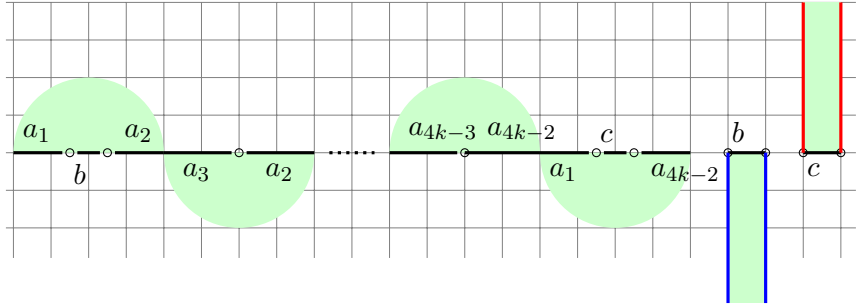
\begin{center}
    \begin{tikzpicture}
\draw[step=.5cm,gray,very thin] (2.9,-1.4) grid (14.2,2);
\fill[green!20!white] (5,0) arc (0:180:10mm) -- (5,0);
\fill[green!20!white] (7,0) arc (0:-180:10mm) -- (7,0);
\fill[green!20!white] (10,0) arc (0:180:10mm) -- (10,0);
\fill[green!20!white] (12,0) arc (0:-180:10mm) -- (12,0);
\node at (3.75,0) [circle,draw,inner sep=0pt,minimum size=3pt]{};
\node at (4.25,0) [circle,draw,inner sep=0pt,minimum size=3pt]{};
\node at (6,0) [circle,draw,inner sep=0pt,minimum size=3pt]{};
\node at (9,0) [circle,draw,inner sep=0pt,minimum size=3pt]{};
\node at (10.75,0) [circle,draw,inner sep=0pt,minimum size=3pt]{};
\node at (11.25,0) [circle,draw,inner sep=0pt,minimum size=3pt]{};
\draw[very thick,dotted] (7.2,0) -- (7.8,0);
\draw[very thick] (3,0) -- (3.65,0) node [above,text width=0.4cm,midway]{$a_1$};
\draw[very thick] (5,0) -- (5.9,0) node [below,text width=0.4cm,midway]{$a_3$};
\draw[very thick] (8,0) -- (8.9,0) node [above,text width=0.4cm,midway]{$a_{4k-3}$};
\draw[very thick] (10,0) -- (10.65,0) node [below,text width=0.4cm,midway]{$a_{1}$};
\draw[very thick] (4.35,0) -- (5,0) node [above,text width=0.4cm,midway]{$a_2$};
\draw[very thick] (6.1,0) -- (7,0) node [below,text width=0.4cm,midway]{$a_2$};
\draw[very thick] (3.85,0) -- (4.15,0) node [below,text width=0.4cm,midway]{$b$};
\draw[very thick] (10.85,0) -- (11.15,0) node [above,text width=0.4cm,midway]{$c$};
\draw[very thick] (9,0) -- (10,0) node [above,text width=0.4cm,midway]{$a_{4k-2}$};
\draw[very thick] (11.35,0) -- (12,0) node [below,text width=0.4cm,midway]{$a_{4k-2}$};

\fill[green!20!white] (12.5,0) -- (13,0) -- (13,-2) -- (12.5,-2) -- (12.5,0);
\node at (12.5,0) [circle,draw,inner sep=0pt,minimum size=3pt]{};
\node at (13,0) [circle,draw,inner sep=0pt,minimum size=3pt]{};
\draw[very thick] (12.5,0) -- (13,0) node [above,text width=0.4cm,midway]{$b$};
\draw[very thick,blue] (12.5,0) -- (12.5,-2);
\draw[very thick,blue] (13,0) -- (13,-2);
\fill[green!20!white] (13.5,0) -- (14,0) -- (14,2) -- (13.5,2) -- (13.5,0);
\node at (13.5,0) [circle,draw,inner sep=0pt,minimum size=3pt]{};
\node at (14,0) [circle,draw,inner sep=0pt,minimum size=3pt]{};
\draw[very thick] (13.5,0) -- (14,0) node [below,text width=0.4cm,midway]{$c$};
\draw[very thick,red] (13.5,0) -- (13.5,2);
\draw[very thick,red] (14,0) -- (14,2);

    \end{tikzpicture}
  
\end{center}
\end{proof}

In general for the signature $\mu=(a,-b,-1^{2s})$ and a set of residue conditions $\mathfrak{R}$ which pair up the simple poles, in order to compute the spin stratum class, we need to resolve a residue condition pairing two simple poles instead of computing the clutching pullbacks. This is because the codimension of the class will be outside the range of injectivity~\eqref{eq:range_inj} of clutching pullbacks. More generally, if we are dealing with some stratum of differentials
\begin{itemize}
    \item whose signature is of the form $\mu=(\mu',-1^{2s})$, where $\mu'$ is a partition of $2g-2+2s$ by even integers containing some negative entry;
    \item with a set of residue conditions $\mathfrak{R}$ pairing the simple poles up,
\end{itemize} then we can directly express the spin stratum classes in the Chow ring of ambient stratum, which is of the same signature but with a set of residue conditions $\mathfrak{R}_0$ such that the residue condition on the first pair of simple poles is removed. Before we state the following proposition, we want to define some special set of two-level graphs of the ambient stratum $\overline{B}=\mathbb{P}\Xi\overline{\mathcal{M}} ^{\mathfrak{R}_0}_{g,n}(\mu)$. We denote 
\begin{itemize}
    \item by $\LG^{\mathfrak{R}}_1(\overline{B})$ the set of level graphs such that the the first and second simple poles are on the top level and their residue sum is forced to be zero;
    \item by $\LG^{\top}_1(\overline{B})$ (resp. $\LG^{\bot}_1(\overline{B})$) the set of level graphs such that the the first simple pole is on the top (resp. bottom) level while the second simple pole is on the bottom (resp. top) level.
\end{itemize}

\begin{prop}\label{prop:sim_res_resolve}
Let $\mu=(\mu',-1^{2s})$ be a signature and $\mathfrak{R}$ (resp. $\mathfrak{R}_0$) be a set of residue conditions as we mentioned above. Then the cycle class of $\mathbb{P}\Xi\overline{\mathcal{M}} ^{\mathfrak{R}}_{g,n}(\mu)$ in the Chow ring of $\overline{B}= \mathbb{P}\Xi\overline{\mathcal{M}} ^{\mathfrak{R}_0}_{g,n}(\mu)$ can be expressed as follow: 
\begin{align}\label{eq:sim_res}
    [\mathbb{P}\Xi\overline{\mathcal{M}} ^{\mathfrak{R}}_{g,n}(\mu)]=-\sum_{\Delta\in\LG^\mathfrak{R}_1(\overline{B})}\ell_\Delta [D_\Delta]+\sum_{\Delta\in\LG^{\top}_1(\overline{B})}\ell_\Delta [D_\Delta].
\end{align}

Similarly, the spin stratum class can be expressed as follow:

\begin{align}\label{eq:spin_sim_res}
    [\mathbb{P}\Xi\overline{\mathcal{M}} ^{\mathfrak{R}}_{g,n}(\mu)]^{\spin}=-\sum_{\Delta\in\LG^\mathfrak{R}_1(\overline{B})}\ell_\Delta [D_\Delta]^{\spin},
\end{align}
\end{prop}
\begin{proof}
We first prove Equation~\eqref{eq:sim_res}. Let $f:\mathbb{P}\Xi\overline{\mathcal{M}} ^{\mathfrak{R}_0}_{g,n}(\mu)\longrightarrow\mathbb{P}^1$ be the extension of the map  
\begin{align*}
\Tilde{f}:&\mathbb{P}\Omega\mathcal{M} ^{\mathfrak{R}_0}_{g,n}(\mu)\longrightarrow \mathbb{P}^1\\
    &(X,\eta)\mapsto [-\res_{p_1}(\eta):\res_{p_2}(\eta)],
\end{align*}
where $p_1$ and $p_2$ are the first two simple poles. Then it is easy to see that the inverse image of $\{1\}$ is just the union of $\mathbb{P}\Xi\overline{\mathcal{M}} ^{\mathfrak{R}}_{g,n}(\mu)$ and $D_\Delta$ for $D_\Delta\in\LG^\mathfrak{R}_1(\overline{B})$ while the inverse image of $\{\infty\}$ is the union of $D_\Delta$ for $\Delta\in\LG^{\top}_1(\overline{B})$.  Then Equation~\eqref{eq:sim_res} is just obtained by considering the multiplicities of the irreducible components of the inverse image. The multiplicity of $D_\Delta$ is just $\ell_\Delta$ because the residues on the bottom level decays like $t^{\ell_\Delta}$. 

Second, our strategy to prove Equation~\eqref{eq:spin_sim_res} is to construct a map from the ambient stratum to $\mathbb{P}^1$ such that the spin components will be lying in the inverse images of two different points on $\mathbb{P}^1$. Without loss of generality, we will assume $\overline{B}$ is connected (otherwise we can restrict to a connected component of $\overline{B}$). We claim that the generator of the monodromy $\pi_1(\mathbb{P}^1\setminus\{0,\infty\},1)\simeq \mathbb{Z}$ will be lifted to a path in $\overline{B}\setminus f^{-1}(\{0,\infty\})$ connecting two elements of $f^{-1}(1)$ of opposite spin. (We will prove our claim later.) This will lead to the reducibility of the fiber product
\begin{center}
\begin{tikzcd}
\mathcal{F}\arrow[r,"p"]\arrow[d,"\bar{f}"] &\overline{B}\arrow[d,"f"]\\
 \mathbb{P}^1\arrow[r,"q"] & \mathbb{P}^1,
\end{tikzcd}
\end{center}
where $q:z\mapsto z^2$. Indeed, the generator of the monodromy $\pi_1(\mathbb{P}^1\setminus\{0,\infty\},1)\simeq \mathbb{Z}$ via $\bar{f}$ ($\mathbb{P}^1$ on the left hand side) will be lifted to a path in in $\mathcal{F}\setminus \bar{f}^{-1}(\{0,\infty\})$ connecting two elements in $ \bar{f}^{-1}(1)$ of the same spin. Thus, the spin components of the inverse image $\bar{f}^{-1}(1)$ will lie on different connected components of $\mathcal{F}\setminus \bar{f}^{-1}(\{0,\infty\})$. Let $\overline{X^+}$ resp. $\overline{X^-}$ be the irreducible components of $\mathcal{F}$ containing the spin components of $\bar{f}^{-1}(1)$. It is obvious that both $\overline{X^+}$ and $\overline{X^-}$ are isomorphic to $\overline{B}$. Hence, by considering the map
\begin{align*}
    \rho:\overline{B}\overset{\sim}{\longrightarrow} \overline{X}^+\overset{\bar{f}}{\longrightarrow}\mathbb{P}^1,
\end{align*}
we see that $\rho^{-1}(1)$ (resp. $\rho^{-1}(-1)$) is isomorphic to the even (resp. odd) spin component of $f^{-1}(1)$. This implies that 
\begin{align*}
    [\mathbb{P}\Xi\overline{\mathcal{M}} ^{\mathfrak{R}}_{g,n}(\mu)^+]+\sum_{\Delta\in\LG^\mathfrak{R}_1(\overline{B})}\ell_\Delta [D_\Delta^+]=[\mathbb{P}\Xi\overline{\mathcal{M}} ^{\mathfrak{R}}_{g,n}(\mu)^-]+\sum_{\Delta\in\LG^\mathfrak{R}_1(\overline{B})}\ell_\Delta [D_\Delta^-]
\end{align*}
and it yields Equation~\eqref{eq:spin_sim_res}.

Now it remains to show the claim on monodromy. We fix a collection of curves $\{\alpha_1,...,\alpha_{g+s},\beta_1,...,\beta_{g+s}\}$ on the flat surface $(X,\eta)$ (where we assume $X$ is smooth, otherwise we consider the welded surface associated to it) corresponding to an element in $f^{-1}(1)$ such that 
\begin{itemize}
    \item $\alpha_i$ only intersects $\beta_i$ (transversally at one point) and vice versa;
    \item $[\alpha_1],...,[\alpha_g],[\beta_1],...,[\beta_g]$ form a symplectic basis of the underlying closed surface
    \item $\beta_{g+1},...,\beta_{g+s}$ are geodesic closed curves around the simple poles (one simple pole for every pair of simple poles, we denote by $\beta_{g+1}',...,\beta_{g+s}'$ the geodesic closed curves of the opposite simple poles);
    \item $\alpha_{g+1},...,\alpha_{g+s} $ are curves connecting the geodesic closed curves $\beta_j,\beta_j'$, such that they are perpendicular to the $\beta_j,\beta_j'$.
\end{itemize}

Let $w=\sum_{i=1}^{g+s} (\ind(\alpha_i)+1)(\ind(\beta_i)+1)$. Note that $w \pmod 2$ will be the spin parity of $(X,\eta)$. Let $\gamma(t)=e^{2\pi i t}$ (where $t\in[0,1]$) be a loop on $\mathbb{P}^1\setminus\{0,\infty\}$ based at $1$. If we deform $(X,\eta)$ along some lift of $\gamma$ in $\overline{B}$, then we will end up with some $(X', \eta')$ of opposite spin. Indeed, the turning numbers of $\alpha_1,...,\alpha_g,\alpha_{g+2},...,\beta_1,...,\beta_{g+s}$ are locally constant along the deformation path as they take integral values. On the other hand, the turning number of $\alpha_{g+1}$ will change by $1$ as the turning angle of $\alpha_{g+1}$ changes continuously along the the deformation path (see Figure~\ref{fig:turning_angle}). This means that the Arf invariants of the original resp. resulting symplectic basis will differ by $1$. This completes our proof. 

\begin{figure}
    \centering
    \begin{tikzpicture}
      \draw[] (0,0) -- (0,2) node [left,text width=0.4cm,midway]{$a$};
       \draw[] (1,0) -- (1,2) node [right,text width=0.4cm,midway]{$a$};
       \draw[dotted] (0,2) -- (1,2);
       \draw[] (0,-1) -- (0,-3) node [left,text width=0.4cm,midway]{$b$};
       \draw[] (1,-1) -- (1,-3) node [right,text width=0.4cm,midway]{$b$};
       \draw[dotted] (0,-3) -- (1,-3);
       \draw[blue] (0,-2) -- (1,-2) node [below,text width=0.4cm,midway]{$\beta_{g+1}'$};
       \draw[blue] (0,1) -- (1,1) node [above,text width=0.4cm,midway]{$\beta_{g+1}$};
       \draw[red] (0.5,1) -- (0.5,0);
       \draw[red] (0.5,-1) -- (0.5,-2);
       \draw[red,dotted] (0.5,0)-- (0.5,-1) node [right,text width=0.4cm,midway]{$\alpha_{g+1}$};
       \begin{scope}[xshift=2.7cm]
         \draw[] (0,0) -- (0,2) node [left,text width=0.4cm,midway]{$a$};
       \draw[] (1,0) -- (1,2) node [right,text width=0.4cm,midway]{$a$};
       \draw[dotted] (0,2) -- (1,2);
       \draw[] (0.3,-1.6) -- (1.6,-3.4) node [left,text width=0.4cm,midway]{$b$};
       \draw[] (1.1,-1) -- (2.4,-2.8) node [right,text width=0.4cm,midway]{$b$};
       \draw[dotted] (1.6,-3.4) -- (2.4,-2.8);
       \draw[blue] (1.1,-2.7) -- (1.9,-2) node [below,text width=0.4cm,midway]{$\beta_{g+1}'$};
       \draw[blue] (0,1) -- (1,1) node [above,text width=0.4cm,midway]{$\beta_{g+1}$};
       \draw[red] (0.5,1) -- (0.5,0);
       \draw[red] (0.8,-1.4) -- (1.5,-2.3);
       \draw[red,dotted] (0.5,0).. controls (0.5,-0.7) and (0.5,-1).. (0.8,-1.4) node [right,text width=0.4cm,midway]{$\alpha_{g+1}$};
       \end{scope}
    \end{tikzpicture}
    \caption{Deformation of pair $(\alpha_{g+1},\beta_{g+1})$ along the loop $\gamma$}
    \label{fig:turning_angle}
\end{figure}
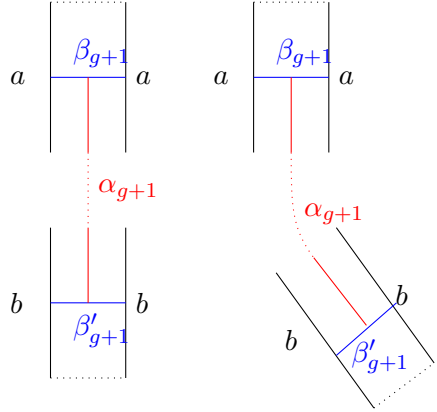
\end{proof}

\begin{rem}
 Proposition~\ref{prop:base_case2} can be also deduced from Proposition~\ref{prop:sim_res_resolve}.
\end{rem}

\subsection{Computation of the spin stratum class of certain horizontal level graphs}\label{sec:horG}

In a moduli space of multi-scale differentials of type $\mu$, where $\mu$ is of even type, there will be no boundary stratum associated to some one-edge horizontal enhanced level graph of compact type. However, in the computation of clutching pullback with respect to $\Gamma_0$ (according to Assumption~\ref{prop:assum}, this will only be applied to the case that $\mu=(m_1,m_2)$), we will encounter signatures of the form $\mu'=(m_1,m_2,-1,-1)$. Then we have to take care of some horizontal level graphs of compact type. This will happen if and only if $m_1,m_2\geq 0$. Hence, in this section, we will explain how to tackle this problem.

To compute the spin stratum class of a holomorphic stratum of signature $\mu=(m_1,...,m_n)$, we will have to compute first the spin stratum class of the meromorphic stratum of signature $\mu'=(m_1,...,m_n,-1,-1)$. If $n\geq 2$, then the multi-scale differentials of type $\mu'$ can be compatible with some one-edge horizontal level graphs of compact type. As all the integers $m_i\geq 0$, the one-edge horizontal graphs of compact type (associated to signature $\mu'$) will always be composed of two vertices such that there are two simple poles (and no other poles) on each vertex. The following example illustrates what kind of horizontal level graphs of compact type we will have to take into account. 

\begin{exa}
\label{exa:hor}
Let $\mu=(2,2)$ and $\mu'=(2,2,-1,-1)$. We want to compute \\$\xi_\Gamma^*p_*\big([\mathbb{P}\Xi\overline{\mathcal{M}}_{g,n}(2,2,-1,-1)]^{\spin}\big)$, where $\Gamma$ is the following stable graph:
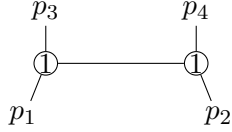
\begin{center}
    \begin{tikzpicture}
    
         \node (a) at (-1,0) [circle,draw,inner sep=0pt,minimum size=5pt]{$1$};
   
    \node (c) at (1,0) [circle,draw,inner sep=0pt,minimum size=5pt]{$1$};

    \path[-] (a) edge (c) ;
    \node at (-1.3,-0.7){$p_1$} ;
    \node at (-1,0.7){$p_3$} ;
    \node at (1,0.7){$p_4$} ;
    \node at (1.3,-0.7){$p_2$} ;
    \path [-] (a) edge (-1.2,-0.5);
    \path [-] (a) edge (-1,0.5);
    \path [-] (c) edge (1.2,-0.5);
    \path [-] (c) edge (1,0.5);
    \end{tikzpicture}
\end{center}
By the Clutching Pullback Formula, we will need to compute the projection pushforward of the divisor class corresponding to the following level graph $\Gamma_H$:

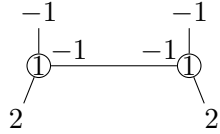
\begin{center}
    \begin{tikzpicture}
    
         \node (a) at (-1,0) [circle,draw,inner sep=0pt,minimum size=5pt]{$1$};
   
    \node (c) at (1,0) [circle,draw,inner sep=0pt,minimum size=5pt]{$1$};

    \path[-] (a) edge (c) ;
    \node at (-1.3,-0.7){$2$} ;
    \node at (-1,0.7){$-1$} ;
    \node at (-0.6,0.2){$-1$} ;
    \node at (1,0.7){$-1$} ;
    \node at (0.6,0.2){$-1$} ;
    \node at (1.3,-0.7){$2$} ;
    \path [-] (a) edge (-1.2,-0.5);
    \path [-] (a) edge (-1,0.5);
    \path [-] (c) edge (1.2,-0.5);
    \path [-] (c) edge (1,0.5);
    \end{tikzpicture}
\end{center}
As in Example~\ref{exa:recur_clutch_0}, the spin parity of a multi-scale differential associated with such kind of level graph seems to be hard to determine. Nonetheless, the good thing here is that the spin parity on each vertex is still well defined. The reason is, if we look at a single vertex, the simple poles are still paired by residue condition.
\end{exa}

However, one needs to be careful that the spin parity of a multi-scale differential associated with the level graph above is not just the sum of the spin parities at the vertices. Indeed, if we cut the horizontal edge and consider the flat surface corresponding to each vertex, we will doubly count the turning number contribution of the loop formed by the "$-1$"-legs. Hence, we should eliminate this effect. The correct way to count is to take the sum of spin parities of the vertices and then minus one. Let $\mu_0=(m_{0,1},...,m_{0,n_0},-1,-1)$ and  $\mu_1=(m_{1,1},...,m_{1,n_1},-1,-1) $ be the signatures on the two vertices respectively. Then we have

\begin{equation}
    p_{\Gamma_H*}([\mathbb{P}\Xi\overline{\mathcal{M}}_{\Gamma_H}]^{\spin})=- \pi_0^*p_{0*}([\mathbb{P}\Xi\overline{\mathcal{M}}_{g_0,n_0}(\mu_0)]^{\spin})\cdot \pi_1^*p_{1*}([\mathbb{P}\Xi\overline{\mathcal{M}}_{g_1,n_1}(\mu_1)]^{\spin}),
\end{equation}
 where $\pi_0,\pi_1$ are the projection map of $\overline{\mathcal{M}}_\Gamma$ to its components. 
 
 In Example~\ref{exa:hor}, each vertex is of the signature $(2,-1,-1)$, thus the corresponding stratum of differentials has only the odd spin component. As a result, $p_{\Gamma_H*}([\mathbb{P}\Xi\overline{\mathcal{M}}_{\Gamma_H}]^{\spin})=-p_{\Gamma_H*}([\mathbb{P}\Xi\overline{\mathcal{M}}_{\Gamma_H}])$.

\begin{rem}
For a stratum whose signature $\mu$ is of even type, there will be no one-edge horizontal level graph which is of compact type. Otherwise, the degree of the differentials on each vertex would be odd.

Assume that $\mu$ is minimal (i.e. there will be only one entry and the entry is non-negative) and $\mu'$ be the signature obtained by adding $-1$s to $\mu$. Then there will be no one-edge horizontal level graph of compact type for $\mu'$. The reason is that there will be a vertex of signature with only negative entries. This can only be possible if the vertex has genus $0$. This will force that the vertex has at least $3$ legs. Since three poles will sum up to strictly less than $-2$, this kind of vertex will not be possible.
\end{rem}

Now, we really have all the tools we need to construct our recursive algorithm. We will explain the algorithm in the next section in detail. 

\section{The algorithm to compute the spin stratum classes}

In this section, we will show the architecture of our algorithm. Then we will enumerate some cases one can easily cross-check our results with. Before we start, let us clarify the background of the sage package we used. Our codes are based on the \texttt{Sage} package \texttt{admcycles} and its subpackage \texttt{diffstrata}. The package \texttt{admcycles} has been designed for doing intersection theory in the tautological rings of moduli spaces of stable curves (cf. \cite{delecroix2020admcycles}). Moreover, it can compute a basis for each component $\RH^{2k}(\overline{\mathcal{M}}_{g,n})$ of the tautological ring by assuming the generalized Faber-Zagier relations. On the other hand, the subpackage \texttt{diffstrata} is programmed for doing intersection theory in the tautological rings of the compactified generalised strata (cf. \cite{costantini2020diffstrata}).

\subsection{Description of the algorithm}
Our algorithm is mainly based on the following tools we introduced in Section~\ref{sec:prepare}:
\begin{itemize}\label{list:tools}
    \item[(i)] Proposition~\ref{prop:arbarello} asserts the existence of a unique solution to the system of linear equations obtained by clutching pullbacks with respect to one-edge stable graphs. The basis of the tautological ring can be computed by \texttt{admcycles}.
    \item[(ii)] Assumption~\ref{prop:assum} allows us to avoid doing clutching pullback with respect to $\Gamma_0$ if the number of markings is larger than $2$. This will reduce the computation complexity.
    \item[(iii)] To compute the clutching pullback of the spin stratum class with respect to a one-edge graph $\Gamma$, we apply the Clutching Pullback Formula~\eqref{eq:excess_p}.
    \item[(iv)] In the computation of a single term of the excess intersection formula, Proposition~\ref{prop:res_resl}, Proposition~\ref{prop:xi} and Proposition~\ref{prop:sim_res_resolve} allow us to pass the spin variant of the fundamental class of a compactified generalised stratum to a tautological class of the ambient stratum. Recursively, one can instead compute the projection pushforwards of a tautological class of a compactified generalised stratum with no residue condition.
    \item[(v)] If we encounter a horizontal level graph, we can just use the method introduced in Section~\ref{sec:horG} to determine the spin parity.
    \item[(vi)] By Corollary~\ref{cor:1.3}, we can ignore the computation of the projection pushforward $p_*([D_\Delta]^{\spin})=0$ if $\Delta$ has any vertical edge of even enhancement. 
    \item[(vii)] Proposition~\ref{prop:base_case1} and Proposition~\ref{prop:base_case2} gives an explicit formula to compute the base cases (the signature is of the form $(0,-1,-1)$ and $(2k,-2k,-1,-1)$) in the recursion. In addition, a differential of even type and $g=0$, which has no simple poles, will have even spin by default.
\end{itemize}

 Our algorithm is designed not only to compute the spin stratum class (which can be regarded as the projection pushforward of the spin variant of the fundamental class of a compactified stratum), but also the projection pushforward of the spin variant of a tautological class of a compactified generalised stratum. 
 
 Before we dig into the details of the algorithm, we introduce some terminology in our algorithm. An \emph{additive generator} is analogous to a decorated stratum in the tautological rings of moduli spaces of stable curves. More precisely, an additive generator of the tautological ring of a compactified generalised stratum $\pmoxrmoduli$ is just a product of a $\psi$-polynomial with the cycle class of a boundary stratum $[B_\Delta]^{\spin}\in \mathrm{R}^*(\pmoxrmoduli)$. Here, $B_\Delta$ refers to the boundary stratum, which is associated to the level graph $\Delta$. Our algorithm will only support additive generators associated with level graphs \emph{without horizontal edges}. A \emph{tautological class} is just a $\mathbb{Q}$-linear combination of additive generators.

\begin{alg}
This algorithm will convert the spin variant of a tautological class of a compactified generalised stratum $\pmoxrmoduli$ (which admits spin structure) to a product tautological class of the moduli space of disconnected stable curves $\prod_{g\in\boldsymbol{g},n\in\boldsymbol{n}}\overline{\mathcal{M}}_{g,n}$. It consists of the following steps:
\begin{description}
\item[Step 1] First, we break down the input tautological class into additive generators and reduce the computation on the single additive generator. As an additive generator $A$ is just $f(\boldsymbol{\psi})\cdot [B_\Delta]$ (where $f(\boldsymbol{\psi})$ is a $\psi$-polynomial), the projection pushforward $p_*(A)$ will be just $f(\boldsymbol{\psi})\cdot p_*([B_\Delta])$. 

\item[Step 2] It reduces to compute $p_*([B_\Delta])$. If $\Delta$ is the trivial level graph, we will jump to Step 3. Otherwise, we extract the level strata $\mathbb{P}\Xi\overline{\mathcal{M}}_{\Delta,i}$ (where $i=0,-1,...,-L+1$) from the levels of $\Delta$. Then we compute the projection pushforwards of the spin variant of the fundamental classes of each level stratum and clutch the classes together to yield $p_*([B_\Delta])$. (Here, if $\Delta$ has any edge of even enhancement, $p_*([B_\Delta])$ will be zero.)

\item[Step 3] At this stage, we have to compute the projection pushforward of the spin variant of the fundamental class of a compactified generalised stratum $\pmoxrmoduli$ on $\prod_{g\in\boldsymbol{g},n\in\boldsymbol{n}}\overline{\mathcal{M}}_{g,n}$. If there is no residue condition (or the poles are only paired up simple poles), then we will go to Step 4 in case the tuple $\boldsymbol{g}$ has only one entry; while it will just return $0$ in case the generalised stratum has more than one product components due to dimensional reason. Otherwise, we will resolve the residue conditions and express the spin variant of the fundamental classes as a tautological class of the ambient stratum. After that, we will return to Step 1.

\item[Step 4] The algorithm now needs to compute the spin stratum class of a signature $\mu$ of even type. If $g=0$ or the signature is of the form $(0,-1,-1)$ and $(2k,-2k,-1,-1)$, we will apply our results of the base cases. Otherwise, we will compute the clutching pullbacks of the spin stratum class with respect to one-edge stable graphs (self-loop graph $\Gamma_0$ will be ignored if the length of $\mu$ is greater than $2$). To apply the Clutching Pullback Formula to a specific one-edge graph $\Gamma$, we will sort out the two-level graphs or horizontal graphs with a $\Gamma$-structure.

\item[Step 5] At this stage, we need to compute a single term in the Clutching Pullback Formula associated to some codimension-$1$ level graph $\Delta$ with a $\Gamma$-structure. We extract the top resp. bottom level stratum $\mathbb{P}\Xi\overline{\mathcal{M}}_{\Delta,\top}$ resp. $\mathbb{P}\Xi\overline{\mathcal{M}}_{\Delta,\bot}$. Then we compute the projection pushforward of the spin variant of the fundamental classes of level these strata (return to Step 3) and clutch these classes to yield a tautological class in $\overline{\mathcal{M}}_\Gamma$.   
\item[Step 6]
After we have computed the clutching pullbacks of the spin stratum class and expressed the pullback classes by the basis computed by \texttt{admcycles}, we will solve the system of linear equations and yield the spin stratum class which is expressed in a basis computed by \texttt{admcycles}.
\end{description}
\end{alg}

If we just want to theoretically show that the spin stratum classes are recursively computable, we can drop Assumption~\ref{prop:assum} and argue just by Proposition~\ref{prop:arbarello} and the Clutching Pullback Formula.
\begin{proof}[Proof of Theorem~\ref{prop:main_result1}]
We can proceed with the algorithm above, but pretend we have a basis for the cohomology ring of every moduli space of stable curves $\overline{\mathcal{M}}_{g,n}$. Let $\mu=(m_1,...,m_r,(-1,-1)^{t})$ be a signature of even type (the expression $(-1,-1)^{t}$ means that we have $2t$ simple poles and these simple poles are paired up by residue condition) where $-2t+\sum_{i=1}^r m_i=2g-2$. The codimension of the spin stratum class of such signature will be 
\begin{align*}
    k=\begin{cases}
    2g+2t-2 & \mbox{ if all } m_i \mbox{ are non-negative}\\
    2g+2t & \mbox{ otherwise }
    \end{cases}
\end{align*}
 Recall the range of injectivity in Proposition~\ref{prop:arbarello}:
\begin{align*}
     d(g,n)=
    \begin{cases}
    2n-7 &\mbox{ if } g=0\\
    2g-2 & \mbox{ if } n=0,1\\
    2g & \mbox{ if } n=2\\
    2g-2+n & \mbox{ otherwise } 
    \end{cases}
\end{align*}
Note that the number of markings $n=r+2t$. In addition,
\begin{align*}
\begin{cases}
r\geq 1  &\mbox{ if all $m_i$ are positive}\\ 
    r\geq 2 &\mbox{ otherwise.}
\end{cases}
\end{align*}

 For $g=0$, the range $d(g,n)=2(r+2t)-7=2r+4t-7$. If all $m_i$ are positive, then the codimension is $2t-2$ and $d(g,n)-(2t-2)=2r+2t-5\geq 0$ unless $r=1,t=1$, which is the case of $\mu=(0,-1,-1)$. Otherwise, the codimension will be $2t$ and $d(g,n)-2t=2r+2t-7\geq 0$ unless $r=2,t=1$, which is the case of $\mu=(2k,-2k,-1,-1)$. Since the initial cases $\mu=(0,-1,-1)$ and $\mu=(2k,-2k,-1,-1)$ are known, we can recursively compute the spin stratum classes for $g=0$ by the algorithm above theoretically via the clutching pullback formula and solving a system of linear equations.
 
 For $g>0$, we have
\begin{align*}
k=\begin{cases}
 2g-3+(1+2t)\leq 2g-3+(r+2t)=d(g,n) &\mbox{ if all $m_i$ are non-negative}\\
    2g-3+(3+2t)\leq 2g-3+(r+2t)=d(g,n) &\mbox{ if } r\geq 3
\end{cases}
\end{align*}
 The remaining signatures for $g>0$ which we cannot apply the clutching pullback method will be of the form $(a,-b,(-1,-1)^t)$. By Proposition~\ref{prop:sim_res_resolve}, the spin stratum classes of these signatures can be reduced to the computation for spin stratum classes on which the clutching pullback method can be applied.
\end{proof}

After we have presented our algorithm, we want to explain why the recursion based on Assumption~\ref{prop:assum} will work. The main problem is whether the emergence of simple poles will finally reduce to the cases we manage to handle. 

\begin{itemize}
    \item[(i)] If $\mu=(m_1,...,m_n)$ is of even type, then the signature $\mu'=(m_1,...,m_n,-1,-1)$ (obtained by degenerating to irreducible nodal curve) will have length of at least $3$. By Assumption~\ref{prop:assum}, to compute the spin stratum class, it suffices to compute the clutching pullbacks with respect to the one-edge stable graphs of compact type. Thus, in the recursion, we only need to consider spin stratum classes of signatures containing at most one pair of simple poles.
    \item[(ii)] Moreover, the signatures $\mu'=(m_1,...,m_n,-1,-1)$ appearing in the computation of clutching pullback with respect to the $\Gamma_0$ are of length less equal to $4$. This is because if $n\geq 3$, then we will not need to compute the clutching pullbacks with respect to $\Gamma_0$ due to Assumption~\ref{prop:assum}. 
    \item[(iii)] Our algorithm only needs to handle horizontal level graphs of compact type for signatures of the form $(m_1,m_2,-1,-1)$ where $m_1,m_2\geq 0$. Then according to Section~\ref{sec:horG}, we will reduce to compute the spin stratum classes associated to signatures $\mu_0=(m_0,-1,-1)$ and $\mu_1=(m_1,-1,-1)$ respectively. For both signatures, Assumption~\ref{prop:assum} still applies, so we will not need to do the clutching pullback with respect to $\Gamma_0$.
    \item[(iv)] The recursion of computing spin stratum class associated to a signature of the form $\mu=(m_0,m_1)$ will end up either with the computation for signatures of the form $(2k,-2k,-1,-1)$ or with the computation for the signature $(0,-1,-1)$.
\end{itemize}

\begin{proof}[Proof of Theorem~\ref{prop:cond_result}]
By the discussion above, if Assumption~\ref{prop:assum} is true for $g\leq g_0$ and $n\leq n_0$, then our algorithm will control the computation of clutching pullbacks for signatures that we manage to handle within \texttt{admcycles}. Then as in the proof of Theorem~\ref{prop:main_result1}, we can conclude that for $g\leq g_0$ and $n\leq n_0$, our algorithm will yield the correct spin stratum class.
\end{proof}

We will give some explicit examples of computing the spin stratum classes in Section~\ref{sec:exa_spin_pull}. 
\subsection{\texttt{Sage} commands for handling (spin) stratum}

We now show how to input a (spin) additive generator or tautological class in \texttt{diffstrata} and call the methods we have coded to compute the spin stratum classes. We first need to define a stratum (with/without spin structure): either using the class \texttt{GeneralisedStratum} or its subclass \texttt{SpinStratum}. For example, we can input the compactified generalised stratum of differentials of even type $\mathbb{P}\Xi\overline{\mathcal{M}}_{1,3}^\mathfrak{R}(4,-2,-2)$, where the residue condition is $r_2=0$, by the following code:
\begin{lstlisting}
sage: from admcycles.diffstrata import *
sage: from admcycles.diffstrata.spinstratum import SpinStratum, addspin
sage: X = SpinStratum([Signature((4,-2,-2))],res_cond=[[(0,1)]])

# alternatively
sage: X_nospin = GeneralisedStratum([Signature((4,-2,-2))],res_cond=[[(0,1)]])
sage: X = addspin(X_nospin)
\end{lstlisting}

 An object in the class \texttt{GeneralisedStratum} has a list \texttt{bics} of vertical two-level graphs. For example, if we want to enter a (spin) additive generator of the stratum above :
\begin{lstlisting}
sage: from admcycles.diffstrata.spinstratum import AG_with_spin, AG_addspin
sage: A = AG_with_spin(X,((2,),0),leg_dict={1:2,2:1})

# alternatively
sage: A_nospin = AdditiveGenerator(X_nospin,((2,),0),leg_dict={1:2,2:1})
sage: A = AG_addspin(A_nospin)
\end{lstlisting}
The tuple $((2,),0)$ refers to the $3$rd two-level graph on the list \texttt{bics}. If we enter $((),0)$, then it refers to the underlying level graph of the generalised stratum itself. The condition \texttt{leg\_dict=\{1:2,2:1\}} means that the level graph is decorated by the monomial $\psi_1^2\psi_2$. The tautological class corresponding to two times of this additive generator will be:
\begin{lstlisting}
from admcycles.diffstrata.spinstratum import ELGT_with_spin, ELGT_addspin
sage: E = ELGT_with_spin(X,[(2,A)])

# alternatively
sage: E_nospin = ELGTautClass(X_nospin,[(2,A_nospin)])
sage: E = ELGT_addspin(E_nospin)
\end{lstlisting}
Hence, if we enter the following codes, it will represent the (spin) stratum class in the tautological ring of $\mathbb{P}\Xi\overline{\mathcal{M}}_{g,n}(\mu)$.
\begin{lstlisting}
sage: A_nospin = AdditiveGenerator(X,((),0))  # no spin
sage: E_nospin = ELGTautClass(X,[(1,A)])  # no spin
sage: A = AG_with_spin(X,((),0))
sage: E = ELGT_with_spin(X,[(1,A)])
\end{lstlisting}

To compute the (spin) stratum class in the tautological ring of $\overline{\mathcal{M}}_{g,n}$, one can use the following:
\begin{lstlisting}
sage: stratum_cl = E_nospin.to_prodtautclass().pushforward()  #nospin
sage: spin_stratum_cl = E.to_prodtautclass_spin().pushforward()
\end{lstlisting}

\subsection{Cross-checking our results with existing results}

The spin stratum class for a stratum of differentials on genus $1$ curves can be easily determined. The odd spin stratum $\mathcal{H}_1(\mu)^-$ coincides with the stratum $\mathcal{H}_1(\mu')$, where $2\mu'=\mu$. Indeed, one can see that from the perspective of theta characteristics. Let $(C;p_1,...,p_n)$ be an marked elliptic curve and $\mu=(2m_1,...,2m_n)$ be a signature such that $\sum_i m_i=0$. Then the sheaf $\mathcal{O}_C(\sum_i m_ip_i)$ has a global section if and only if it is isomorphic to $\mathcal{O}_C\simeq \omega_C$. In other words, $h^0(\mathcal{O}_C(\sum m_ip_i))=1$ if $\sum m_ip_i$ is a canonical divisor, otherwise $h^0(\mathcal{O}_C(\sum m_ip_i))=0$. This implies that the $\mathcal{H}_1(2m_1,...,2m_n)^-\simeq \mathcal{H}_1(m_1,...,m_n)$. Hence, 
\[[\mathcal{H}_1(2m_1,...,2m_n)]^{\spin}=[\mathcal{H}_1(2m_1,...,2m_n)]-2[ \mathcal{H}_1(m_1,...,m_n)]\]

Due to the fact that all the explicit computations have to be done by \texttt{admcycles}, we will show in Section~\ref{sec:cross_check_sage} how one can cross-check our results of spin stratum classes in \texttt{Sage} with the existing results. We will demonstrate the \texttt{Sage} commands for doing the cross-check for the following signatures:

\begin{itemize}
    \item $g=1$: $\mu=(4,4,-8)$ (by the discussion above);
    \item $g=2$: $\mu=(2)$ (cf. \cite{kontsevich2003connected}) and $\mu=(4,-2)$ (cf. \cite{schmitt2020intersections}, \cite{boissy2015connected});
    \item $g=3$: $\mu=(4)$ and $\mu=(2,2)$ (cf. \cite{kontsevich2003connected}.
\end{itemize}

\subsection{Cross-checking the results in \texttt{Sage}}\label{sec:cross_check_sage}

We apply the following methods in \texttt{admcycles} to check the results for $\mu=(4,4,-8)$:

\begin{lstlisting}
sage: from admcycles import *
sage: from admcycles.diffstrata import *
sage: from admcycles.diffstrata.spinstratum import Spin_strataclass
sage: Spinclass = Strataclass(1,1,(4,4,-8))- 2*Strataclass(1,1,(2,2,-4)) 
sage: Spinclass.basis_vector()
(0, -17, 9, 33, -26)
sage: X = SpinStratum([Signature((4,4,-8))])
sage: A=AG_with_spin(X,((),0))
sage: E=ELGT_with_spin(X,[(1,A)])
sage: our_result = E.to_prodtautclass_spin().pushforward()
sage: our_result.basis_vector()
(0, -17, 9, 33, -26)

# alternatively, we have a simpler method Spin_strataclass
sage: Spin_strataclass((4,4,-8)).basis_vector()
(0, -17, 9, 33, -26)
\end{lstlisting}

If $\mu=(2)$, according to \cite{kontsevich2003connected}, the stratum of type $(2)$ is connected and hyperelliptic of odd spin, thus we have $[\overline{\mathcal{H}}(2)]^{\spin}=-[\overline{\mathcal{H}}(2)]$.

\begin{lstlisting}
sage: Spinclass = -Strataclass(2,1,(2,))
sage: Spinclass.basis_vector()
(1/2, -7/2, 1/2)
sage: X = SpinStratum([Signature((2,))])
sage: A=AG_with_spin(X,((),0))
sage: E=ELGT_with_spin(X,[(1,A)])
sage: our_result = E.to_prodtautclass_spin().pushforward()
sage: our_result.basis_vector()
(1/2, -7/2, 1/2)
\end{lstlisting}

If $\mu=(4,-2)$, the stratum of differentials has a hyperelliptic component of odd spin and a component of even spin, according to \cite{boissy2015connected}. In \cite{schmitt2020intersections}, Schmitt and van Zelm have used the method of clutching pullbacks to compute the fundamental class of loci of admissible covers in $\overline{\mathcal{M}}_{g,n}$. The hyperelliptic locus is a special case. The method is built-in in \texttt{admcycles}.

\begin{lstlisting}
sage: G = CyclicPermutationGroup(2)
sage: H = HurData(G,[G[1],G[1],G[1],G[1],G[1],G[1]])
sage: Hyp = (1/factorial(4))*Hidentify(2,H,markings=[1,2])
sage: Spinclass = Strataclass(2,1,(4,-2)) - 2*Hyp
sage: Spinclass.basis_vector()
(-63/2, 21/2, -22, -27, 44, 41, 54, 63/2, -147/2, 21/2, -1, -6, -21/2, 0)
sage: X = SpinStratum([Signature((4,-2))])
sage: A=AG_with_spin(X,((),0))
sage: E=ELGT_with_spin(X,[(1,A)])
sage: our_result = E.to_prodtautclass_spin().pushforward()
sage: our_result.basis_vector()
(-63/2, 21/2, -22, -27, 44, 41, 54, 63/2, -147/2, 21/2, -1, -6, -21/2, 0)
\end{lstlisting}

If $\mu=(4)$, according to \cite{kontsevich2003connected}, the differential stratum has a hyperelliptic component of even spin and a component of odd spin. Hence we can use the same trick as before. 

\begin{lstlisting}
sage: G = CyclicPermutationGroup(2)
sage: H = HurData(G,[G[1],G[1],G[1],G[1],G[1],G[1],G[1],G[1]])
sage: Hyp = (1/factorial(7))*Hidentify(3,H,markings=[1])
sage: Spinclass = 2*Hyp - Strataclass(3,1,(4,))  
sage: Spinclass.basis_vector()
(-729, 1099/12, -787/12, 213, -1151/12, 33/2, 533/2, 985/12, -779/12, -78, 219/2, 19/2, 50, -2213/12, -197/24, 35)
sage: X = SpinStratum([Signature((4,))])
sage: A=AG_with_spin(X,((),0))
sage: E=ELGT_with_spin(X,[(1,A)])
sage: our_result = E.to_prodtautclass_spin().pushforward()
sage: our_result.basis_vector()
(-729, 1099/12, -787/12, 213, -1151/12, 33/2, 533/2, 985/12, -779/12, -78, 219/2, 19/2, 50, -2213/12, -197/24, 35)
\end{lstlisting}

\begin{rem}
The fundamental class of the hyperelliptic component was actually also computed by Chen in \cite{chen2016loci}:
\begin{align*}
  [\overline{\mathcal{H}}_{3}(4)^+]=& \psi(18\lambda-2\delta_0-9\delta^{\{1\}}_1-6\delta^{\{1\}}_1)-\lambda(45\lambda-\frac{19}{2}\delta_0-24\delta^{\{1\}}_2)\\&-\frac{1}{2}\delta_0^2-\frac{5}{2}\delta_0\delta^{\{1\}}_2-3(\delta^{\{1\}}_2)^2  
\end{align*}
One can check that it coincides with the class computed in \texttt{admcycles}.
\begin{lstlisting}
sage: Hyp.basis_vector()
(577/2, -36, 49/2, -81, 77/2, -8, -108, -65/2, 25, 31, -83/2, -7/2, -41/2, 151/2, 13/4, -14)
sage: chen_Hyp = psiclass(1,3,1)*(18*lambdaclass(1,3,1)- tautgens(3,1,1)[4]-9*tautgens(3,1,1)[3]-6*tautgens(3,1,1)[2])-lambdaclass(1,3,1)*(45*lambdaclass(1,3,1)-(19/4)*tautgens(3,1,1)[4]-24*tautgens(3,1,1)[2])-(1/8)*tautgens(3,1,1)[4]^2-(5/4)*tautgens(3,1,1)[4]*tautgens(3,1,1)[2]-3*tautgens(3,1,1)[2]^2
sage:  chen_Hyp.basis_vector()
(577/2, -36, 49/2, -81, 77/2, -8, -108, -65/2, 25, 31, -83/2, -7/2, -41/2, 151/2, 13/4, -14)
\end{lstlisting}

\end{rem}

If $\mu=(2,2)$, the stratum also has exactly two components, namely the hyperelliptic component and the odd spin component. We can use the same method as in the case of $\mu=(4)$ to compute the spin stratum class. The result agrees with the spin stratum class yielded by our algorithm.

In the next section, we will introduce the conjecture of the twisted spin double ramification cycles, and show how our algorithm can be improved by making use of it. 
 
\section{Conjectural formula of the spin twisted DR cycle} \label{sec:conj}

In this section, we will first introduce the theorem of twisted double ramification cycles (twisted DR cycle) and then state the conjectural formula of the spin twisted double ramification cycle from \cite{costantini2021integrals}. We now recall the definition of a twisted double ramification cycle. A \emph{$k$-twisted double ramification cycle} $\DR_g(a)$, where the $k$-ramification vector $a=(a_1,...,a_n)$ is a tuple of integers summing up to $k(2g-2+n)$, is a cycle on $\overline{\mathcal{M}}_{g,n}$ compactifying the condition $\omega_{\log}^{\otimes k}\simeq \mathcal{O}_C(a_1p_1+...+a_np_n)$. 
Notice that given a signature $\mu=(m_1,...,m_n)$ of a stratum of $k$-differentials (i.e. $\sum_im_i=k(2g-2)$), one can define a $k$-twisted double ramification cycle $\DR_g(a_\mu)$ by letting $a_\mu=(m_1+k,...,m_n+k)$.

\subsection{A formula of the $k$-twisted double ramification cycles}
Now we will introduce various geometric approaches to express the $k$-twisted double ramification cycles. For the (non-twisted) double ramification cycles, Pixton has constructed a formula expressed in terms of generators of the tautological ring of $\overline{\mathcal{M}}_{g,n}$, which has been proved in \cite{janda2017double}. The formula of double ramification cycles with targets was later proved in \cite{janda2020double}. 

The formula for $k$-twisted double ramification cycles is very similar to Pixton's formula. We need to give some definitions to state the formula. Let $a=(a_1,...,a_n)$ be a $k$-ramification vector. An \emph{admissible $k$-weighting} $\pmod r$ $w$ of a dual graph $\Gamma$ is a function $w:H(\Gamma)\longrightarrow \{0,1,...,r-1\}$, where $H(\Gamma)$ is the set of half edges and legs of $\Gamma$, such that 
\begin{itemize}
    \item for $h_i$, which is corresponding to marking $i\in \{1,...,n\}$, one has $w(h_i)=a_i \pmod r$,
    \item for any edge $e$ consisting of two half edge $h,h'$, we have $w(h)+w(h')=0 \pmod r$,
    \item for any $v\in V(\Gamma)$, \[\sum_{h\in H_\Gamma(v)} w(h)=k(2g(v)-2+n(v)) \pmod r , \]
    where $H_\Gamma(v)$ is the set of half-edges and legs incident to $v$.
    
\end{itemize}
 We define a mixed degree tautological class:

\begin{align*}
 P^{r,\bullet}_g(a)=\sum_{\Gamma\in G_{g,n}}\sum_{w\in W_{\Gamma,r}(a)} \frac{1}{|\Aut(\Gamma)|}\frac{1}{r^{h^1(\Gamma)}}\Cont_{a,\Gamma,w},   
\end{align*}
where $G_{g,n}$ is the set of all dual graphs with genus $g$ and $n$; $W_{\Gamma,r}(a)$ is the set of admissible $k$-weightings $\pmod r$ on $\Gamma$ with respect to the ramification vector $a$; and
\begin{align*}
    &\Cont_{a,\Gamma,w}\\
    =& \xi_{\Gamma *}\Bigg[\prod_{v\in V(\Gamma)}\exp(-k\kappa_1[v])\prod_{i=1}^n\exp(a_i^2\psi_{h_i})\prod_{e=(h,h')\in E(\Gamma)}\frac{1-\exp(-w(h)w(h')(\psi_h+\psi_{h'}))}{\psi_h+\psi_{h'}} \Bigg].
\end{align*}
 For sufficiently large values of $r$, the class $P^{r,\bullet}_{g}(a)$ is a mixed degree tautological class on $\overline{\mathcal{M}}_{g,n}$ whose coefficients are polynomials in $r$. We denote the class we obtained by substituting $r=0$ to the polynomials by $P^{\bullet}_g(a)$.
 
Finally, we can give the formula of the $k$-twisted double ramification cycle:
\begin{align}
 \DR_g(a)=2^{-g}P^g_g(a),
\end{align}
where $P^g_g(a)$ is the degree $g$ terms of $P^{\bullet}_g(a)$.

\subsection{The theorem of $k$-twisted DR cycle} In \cite{farkas2018moduli}, Farkas and Pandharipande have given a conjecture that the $1$-twisted double ramification cycle $\DR_g(a_\mu)$ agrees with the weighted fundamental class of the moduli space of twisted canonical divisors $[\widetilde{H}_g(\mu)]$. One can recall the definition of twisted $k$-canonical divisors by Remark~\ref{rem:twist_can}. Later in \cite{schmitt2016dimension}, the conjecture has been generalised for $k>1$. The generalised conjecture was later proved in \cite{holmes2021infinitesimal} and \cite{bae2020pixton}.
A \emph{simple star graph} is a vertical two-level enhanced level graph of type $\mu=(m_1,...,m_n)$ (where $\sum_im_i=k(2g-2)$) such that:
\begin{itemize}
    \item there is only one vertex $v_0$ (which is called the center) on the bottom level;
    \item there is no leg on the top level representing a pole;
    \item all the orders of poles and zeros of any vertex on the top level are divisible by $k$.
\end{itemize}
We denote the set of simple star graphs compatible to a given signature $\mu$ by $SG_1(\mu)$, now we can state the theorem of the $k$-twisted double ramification cycles.

\begin{thm}[\cite{holmes2021infinitesimal},\cite{bae2020pixton}]\label{thm:bigthm}
Let $g,n\geq 0$ and $k>0$, and $\mu=(m_1,...,m_n) \in \mathbb{Z}^n\setminus k\mathbb{Z}^n_{\geq 0}$ such that $\sum_ia_i=k(2g-2)$. Then we have
\begin{align}
    \DR_g(a_\mu)= \sum_{\Delta\in SG_1(\mu)}\frac{\prod_{e\in E(\Delta)}\kappa_e}{k^{N_0}|\Aut(\Delta)|}\xi_{\Delta*}\Bigg[& \big[\overline{\mathcal{H}}_{g(v)}^k(\mu[v],\kappa[v]-k)\big]\cdot\\
    &\prod_{v\in V(\Delta^\bot)}\big[\overline{\mathcal{H}}_{g(v)}(\frac{\mu[v]}{k},\frac{\kappa[v]-k}{k})\big]\Bigg],
\end{align}
where $\kappa_e$ is the enhancement of the edge $e$ and $N_0$ is the number of vertices on the top level. In addition, $\overline{\mathcal{H}}_{g(v)}(\mu[v], \kappa[v]-k)$ is the stratum $k$-differentials of signature prescribed by $\mu[v]$ and the enhancements of edges adhered to $v$. 
\end{thm}
 
 \subsection{The conjecture on spin $k$-twisted DR cycle}\label{sec:spin_DR_conj} 
 We will first define the spin double ramification cycle and then state the conjecture of the spin version of Theorem~\ref{thm:bigthm}. Let $k$ be a positive odd integer and $a$ be a ramification vector whose entries are all odd. We define the spin variant of the mixed degree tautological class $P^{r,\bullet}_g(a)$ (where $r$ is an even number) as:
 \begin{align*}
     P^{r,\spin,\bullet}_g(a)=\sum_{\Gamma\in G_{g,n}}\sum_{w\in W_{\Gamma,r,odd}(a)} \frac{1}{2^{g-h^1(\Gamma)}|\Aut(\Gamma)|}\frac{1}{r^{h^1(\Gamma)}}\Cont_{a,\Gamma,w},
 \end{align*}
 where $W_{\Gamma,r,odd}(a)$ is the set of admissible $k$-weightings such that all the edges are weighted by odd numbers. For sufficiently large values of $r$, $P^{r,\spin,\bullet}_g(a)$ is a mixed degree tautological class whose coefficients are polynomials in $r$. We denote the class we obtained by substituting $r=0$ to the polynomials by $P^{\spin,\bullet}_g(a)$.
 
 The formula of a \emph{spin $k$-twisted double ramification cycle ramification cycle} is:
 \begin{align*}
     \DR_g^{\spin}(a)=2^{-g}P^{\spin,g}_g(a).
 \end{align*}
 
 Finally, we can state the conjecture on the spin $k$-twisted double ramification cycles:
 \begin{conj}[\cite{costantini2021integrals}]\label{conj:spin_pixton}
 Let $g,n\geq 0$ and $k$ be a positive odd integer. In addition, let $\mu=(m_1,...,m_n) \in \mathbb{Z}^n\setminus k\mathbb{Z}^n_{\geq 0}$ such that $\sum_ia_i=k(2g-2)$ and all the entries of $\mu$ are even. Then we have
\begin{align}\label{eq:spin_pixton}
    \begin{split}
    \DR^{\spin}_g(a_\mu)= \sum_{\Delta\in SG_1^{odd}(\mu)}\frac{\prod_{e\in E(\Delta)}\kappa_e}{k^{N_0}|\Aut(\Delta)|}\xi_{\Delta*}\Bigg[& \big[\overline{\mathcal{H}}_{g(v)}^k(\mu[v],\kappa[v]-k)\big]^{\spin}\cdot\\ &\prod_{v\in V(\Delta^\bot)}\big[\overline{\mathcal{H}}_{g(v)}(\frac{\mu[v]}{k},\frac{\kappa[v]-k}{k})\big]^{\spin}\Bigg],
    \end{split}
\end{align}
where $SG_1^{odd}(\mu)$ is the subset of $SG_1(\mu)$ consisting of the level graphs such that the enhancement of every edge is odd.
 \end{conj}

The left hand side of \eqref{eq:spin_pixton} is already computable in \texttt{admcycles} but the right hand side requires us to compute the spin stratum classes of strata of $k$-differentials. For $k=1$, the algorithms we developed can help compute the right hand side so that we can verify the conjecture for plenty of signatures $\mu=(m_1,...,m_n)$ ($\sum_im_i=2g-2$) of meromorphic strata where $(g,n)$ equal:
\begin{align*}
(2,2),(2,3),(2,4),(2,5),(3,2),(3,3),(4,2).    
\end{align*}

Note that on the right hand side of \eqref{eq:spin_pixton}, the contribution of a vertex on the top level of a simple star graph is just the spin stratum class of some holomorphic stratum of $1$-differential. Thus, if we assume the conjecture to be true, then we can first use our algorithm to compute the spin stratum classes of strata of holomorphic $1$-differentials and compute the spin stratum classes of strata of meromorphic $k$-differentials by recursively solving the equation \eqref{eq:spin_pixton}. 

In \texttt{admcycles}, one can directly call the function \texttt{Strataclass} to compute the spin stratum class for a signature of even type. Moreover, one can also use the keyword argument \texttt{spin\_conj=True} to apply Conjecture~\ref{conj:spin_pixton} to the computation of the spin stratum class.
\begin{lstlisting}
sage: from admcycles import *
sage: cl1=Strataclass(1,1,(6,-4),spin=True,spin_conj=True)
sage: cl2=Strataclass(1,1,(6,-4),spin=True)
sage: cl1 == cl2
True
\end{lstlisting}

\appendix
\section{Examples and \texttt{Sage} commands}

\subsection{Example: Applying the clutching pullback formula to $[\overline{\mathcal{H}}_1(4,-2,-2)]$}\label{sec:exa_clutch}
Before we start to consider the clutching pullback explicitly, let us introduce the notation of boundary strata of $\overline{\mathcal{M}}_{g,n}$:
\begin{itemize}
    \item Consider $P\subseteq I=\{1,...,n\}$ and $g\geq q\geq 0$, we use $S_q^P$ to denote the boundary stratum of curves with one node and the markings in $P$ lie in the component of genus $q$. Notice that $S_q^P=S_{g-q}^{P^c}$, where $P^c$ is the complement of $P$. We use $\delta_{q}^{P}$ to denote the fundamental class of the corresponding boundary stratum and $\Gamma_q^{P}$ the corresponding dual graph.
    \item $S_0$ refers to the boundary stratum parametrizing curves with exactly one self-node and $\delta_0$ is referred to as its fundamental class. We denote the dual graph of this boundary stratum by $\Gamma_0$.
\end{itemize}
Furthermore, recall that on $\mathbb{P}\Xi\overline{\mathcal{M}}_{g,n}(\mu)$, a boundary stratum is parametrised by some enhanced level graph $\Delta$. A codimension one boundary stratum of $\mathbb{P}\Xi\overline{\mathcal{M}}_{g,n}(\mu)$ is parametrized by some enhanced level graph $\Delta$ that has exactly one horizontal edge or is vertical with two levels. We use the following notations to denote the divisor corresponding to a one-edge horizontal graph:
\begin{itemize}
    \item[(i)] $D_h$ is referred to as the divisor corresponding to the boundary stratum corresponding to the self-loop graph;
    \item[(ii)] $D_{q}^P$ is referred to as the divisor with respect to the boundary stratum associated to the one-edge horizontal level graph of compact type such that the markings in $P$ lie in the component of genus $q$;
    \item[(iii)] We will also use $\Gamma_{q}^{P}$ and $\Gamma_0$ to denote the level graph of the divisors in (i) and (ii).
\end{itemize}

We now consider the stratum class $[\overline{\mathcal{H}}_1(4,-2,-2)]\in H^{2}(\overline{\mathcal{M}}_{1,3})$. The boundary divisors of the moduli space $\mathbb{P}\Xi \overline{\mathcal{M}}_{1,3}(4,-2,-2)$ are parametrized by vertical two-level graphs or one-edge horizontal graphs. We now consider the clutching pullback of $[\overline{\mathcal{H}}_1(4,-2,-2)]$ with respect to the one-edge stable graphs. The vertical two-level graphs and one-edge horizontal graphs which have some $\Gamma_0$-structures are:
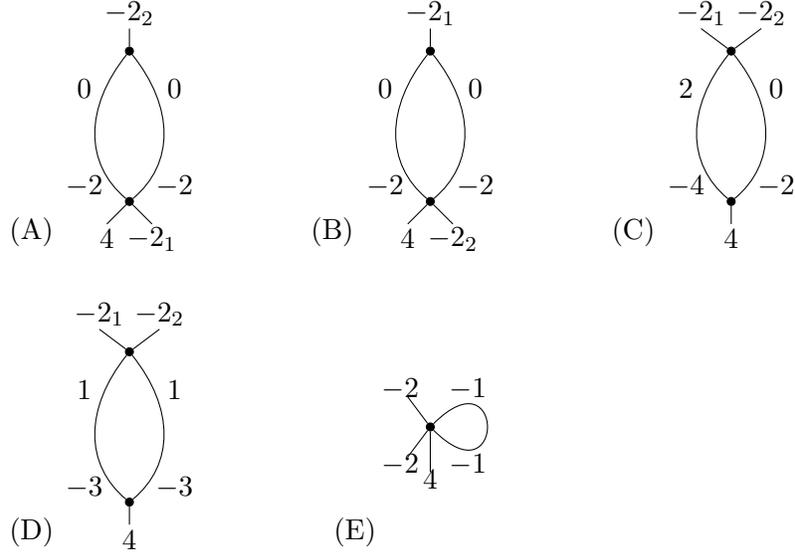
\begin{center}
    \begin{tikzpicture}
    \begin{scope}[xshift=0cm,yshift=0cm]
        \node (a) at (0,0) [circle,draw,inner sep=0pt,fill,minimum size=3pt]{};
        \node (b) at (0,-2) [circle,draw,fill,inner sep=0pt,minimum size=3pt]{};
        \draw (a).. controls (-0.6,-0.75) and (-0.6,-1.5).. (b);
         \draw (a).. controls (0.6,-0.75) and (0.6,-1.5).. (b);
        \path [-] (b) edge (-0.3,-2.3);
        \path [-] (b) edge (0.3,-2.3);
        \path[-] (a) edge (0,0.3);
        \node at (-1.3,-2.4){(A)};  \node at (-0.3,-2.5){$4$};
        \node at (0.3,-2.5){$-2_1$}; \node at (0,0.5){$-2_2$};
        \node at (-0.6,-0.5){$0$};
        \node at (0.6,-0.5){$0$};
        \node at (-0.6,-1.8){$-2$};
        \node at (0.6,-1.8){$-2$};
        \node(0) at (2,1){};
    \end{scope}
    \begin{scope}[xshift=4cm,yshift=0cm]
        \node (a) at (0,0) [circle,draw,inner sep=0pt,fill,minimum size=3pt]{};
        \node (b) at (0,-2) [circle,draw,fill,inner sep=0pt,minimum size=3pt]{};
        \draw (a).. controls (-0.6,-0.75) and (-0.6,-1.5).. (b);
         \draw (a).. controls (0.6,-0.75) and (0.6,-1.5).. (b);
        \path [-] (b) edge (-0.3,-2.3);
        \path [-] (b) edge (0.3,-2.3);
        \path[-] (a) edge (0,0.3);
        \node at (-1.3,-2.4){(B)};  \node at (-0.3,-2.5){$4$};
        \node at (0.3,-2.5){$-2_2$}; \node at (0,0.5){$-2_1$};
        \node at (-0.6,-0.5){$0$};
        \node at (0.6,-0.5){$0$};
        \node at (-0.6,-1.8){$-2$};
        \node at (0.6,-1.8){$-2$};
       
    \end{scope}
    
     \begin{scope}[xshift=8cm,yshift=0cm]
        \node (a) at (0,0) [circle,draw,inner sep=0pt,fill,minimum size=3pt]{};
        \node (b) at (0,-2) [circle,draw,fill,inner sep=0pt,minimum size=3pt]{};
        \draw (a).. controls (-0.6,-0.75) and (-0.6,-1.5).. (b);
         \draw (a).. controls (0.6,-0.75) and (0.6,-1.5).. (b);
        \path [-] (a) edge (-0.4,0.3);
        \path [-] (a) edge (0.4,0.3);
        \path[-] (b) edge (0,-2.3);
        \node at (-1.3,-2.4){(C)};  \node at (0,-2.5){$4$};
        \node at (0.4,0.5){$-2_2$}; \node at (-0.4,0.5){$-2_1$};
        \node at (-0.6,-0.5){$2$};
        \node at (0.6,-0.5){$0$};
        \node at (-0.6,-1.8){$-4$};
        \node at (0.6,-1.8){$-2$};
       
    \end{scope}
    \begin{scope}[xshift=0cm,yshift=-4cm]
        \node (a) at (0,0) [circle,draw,inner sep=0pt,fill,minimum size=3pt]{};
        \node (b) at (0,-2) [circle,draw,fill,inner sep=0pt,minimum size=3pt]{};
        \draw (a).. controls (-0.6,-0.75) and (-0.6,-1.5).. (b);
         \draw (a).. controls (0.6,-0.75) and (0.6,-1.5).. (b);
        \path [-] (a) edge (-0.4,0.3);
        \path [-] (a) edge (0.4,0.3);
        \path[-] (b) edge (0,-2.3);
        \node at (-1.3,-2.4){(D)};  \node at (0,-2.5){$4$};
        \node at (0.4,0.5){$-2_2$}; \node at (-0.4,0.5){$-2_1$};
        \node at (-0.6,-0.5){$1$};
        \node at (0.6,-0.5){$1$};
        \node at (-0.6,-1.8){$-3$};
        \node at (0.6,-1.8){$-3$};
       
    \end{scope}
    
       \begin{scope}[xshift=4cm,yshift=-5cm]
         \node (a) at (0,0) [circle,draw,fill,inner sep=0pt,minimum size=3pt]{};
         \draw (a).. controls (1,1) and (1,-1).. (a);
         \node at (0,-0.7){$4$}; 
         \node at (-0.4,0.5){$-2$}; 
         \node at (-0.4,-0.5){$-2$}; 
         \node at (0.5,0.5){$-1$};
         \node at (0.5,-0.5){$-1$}; 
         \path [-] (a) edge (0,-0.6);
         \path [-] (a) edge (-0.3,-0.4);
         \path [-] (a) edge (-0.3,0.4);
         \node at (-1,-1.4){(E)};
    \end{scope}
    \end{tikzpicture}
\end{center}
In this example, $\overline{\mathcal{M}}_{\Gamma_0}=\overline{\mathcal{M}}_{0,5}$. The contribution of each of the above vertical two-level graphs or one-edge horizontal graphs to $\xi^*_{\Gamma_0}[\overline{\mathcal{H}}_{1}(4,-2,-2)]$ are:
\begin{itemize}
    \item[(A)] There are four $\Gamma_0$-structures on Graph $A$. They map the $4$-th and $5$-th markings of $\overline{\mathcal{M}}_{0,5}$ to any one of the vertical edges and each edge gives us two possibilities to allocate the markings. Hence, its contribution of is
    \begin{center}
  \begin{tabular}{c@{\hskip 0.2 cm}c@{\hskip 0.2 cm}c@{\hskip 0.2 cm}c}
    $\frac{1}{2}\cdot\frac{1}{1}\cdot\frac{1}{1} \Bigg[\begin{tikzpicture}[->,baseline=-3pt,node distance=1.3cm,thick,main node/.style={circle,draw,font=\Large,scale=0.5}]
\node at (0,0) (C) {};
\node [scale=.3,draw,circle,fill] [above of =C] (A) {};
\node [scale=.3,draw,circle,fill] [below of =C] (B) {};
\node at (-0.5,-0.7) (n5) {5};
\node at (0.5,-0.7) (n2) {2};
\node at (0.0,-1.0) (n1) {1};
\node at (-0.5,0.7) (n4) {4};
\node at (0.5,0.7) (n3) {3};
\draw [-] (A) to (B);
\draw [-] (B) to (n1);
\draw [-] (B) to (n2);
\draw [-] (B) to (n5);
\draw [-] (A) to (n3);
\draw [-] (A) to (n4);
\end{tikzpicture}\Bigg]$
& 
    +$\frac{1}{2}\cdot\frac{1}{1}\cdot\frac{1}{1} \Bigg[\begin{tikzpicture}[->,baseline=-3pt,node distance=1.3cm,thick,main node/.style={circle,draw,font=\Large,scale=0.5}]
\node at (0,0) (C) {};
\node [scale=.3,draw,circle,fill] [above of =C] (A) {};
\node [scale=.3,draw,circle,fill] [below of =C] (B) {};
\node at (-0.5,-0.7) (n5) {4};
\node at (0.5,-0.7) (n2) {2};
\node at (0.0,-1.0) (n1) {1};
\node at (-0.5,0.7) (n4) {5};
\node at (0.5,0.7) (n3) {3};
\draw [-] (A) to (B);
\draw [-] (B) to (n1);
\draw [-] (B) to (n2);
\draw [-] (B) to (n5);
\draw [-] (A) to (n3);
\draw [-] (A) to (n4);
\end{tikzpicture}\Bigg]$
& 
     $ + \frac{1}{2}\cdot\frac{1}{1}\cdot\frac{1}{1} \Bigg[\begin{tikzpicture}[->,baseline=-3pt,node distance=1.3cm,thick,main node/.style={circle,draw,font=\Large,scale=0.5}]
\node at (0,0) (C) {};
\node [scale=.3,draw,circle,fill] [above of =C] (A) {};
\node [scale=.3,draw,circle,fill] [below of =C] (B) {};
\node at (-0.5,-0.7) (n5) {1};
\node at (0.5,-0.7) (n2) {5};
\node at (0.0,-1.0) (n1) {2};
\node at (-0.5,0.7) (n4) {3};
\node at (0.5,0.7) (n3) {4};
\draw [-] (A) to (B);
\draw [-] (B) to (n1);
\draw [-] (B) to (n2);
\draw [-] (B) to (n5);
\draw [-] (A) to (n3);
\draw [-] (A) to (n4);
\end{tikzpicture}\Bigg]$
&
 $ + \frac{1}{2}\cdot\frac{1}{1}\cdot\frac{1}{1} \Bigg[\begin{tikzpicture}[->,baseline=-3pt,node distance=1.3cm,thick,main node/.style={circle,draw,font=\Large,scale=0.5}]
\node at (0,0) (C) {};
\node [scale=.3,draw,circle,fill] [above of =C] (A) {};
\node [scale=.3,draw,circle,fill] [below of =C] (B) {};
\node at (-0.5,-0.7) (n5) {1};
\node at (0.5,-0.7) (n2) {4};
\node at (0.0,-1.0) (n1) {2};
\node at (-0.5,0.7) (n4) {3};
\node at (0.5,0.7) (n3) {5};
\draw [-] (A) to (B);
\draw [-] (B) to (n1);
\draw [-] (B) to (n2);
\draw [-] (B) to (n5);
\draw [-] (A) to (n3);
\draw [-] (A) to (n4);
\end{tikzpicture}\Bigg]$\\

$= \Bigg[\begin{tikzpicture}[->,baseline=-3pt,node distance=1.3cm,thick,main node/.style={circle,draw,font=\Large,scale=0.5}]
\node at (0,0) (C) {};
\node [scale=.3,draw,circle,fill] [above of =C] (A) {};
\node [scale=.3,draw,circle,fill] [below of =C] (B) {};
\node at (-0.5,-0.7) (n5) {5};
\node at (0.5,-0.7) (n2) {2};
\node at (0.0,-1.0) (n1) {1};
\node at (-0.5,0.7) (n4) {4};
\node at (0.5,0.7) (n3) {3};
\draw [-] (A) to (B);
\draw [-] (B) to (n1);
\draw [-] (B) to (n2);
\draw [-] (B) to (n5);
\draw [-] (A) to (n3);
\draw [-] (A) to (n4);
\end{tikzpicture}\Bigg]$
& 
    +$ \Bigg[\begin{tikzpicture}[->,baseline=-3pt,node distance=1.3cm,thick,main node/.style={circle,draw,font=\Large,scale=0.5}]
\node at (0,0) (C) {};
\node [scale=.3,draw,circle,fill] [above of =C] (A) {};
\node [scale=.3,draw,circle,fill] [below of =C] (B) {};
\node at (-0.5,-0.7) (n5) {4};
\node at (0.5,-0.7) (n2) {2};
\node at (0.0,-1.0) (n1) {1};
\node at (-0.5,0.7) (n4) {5};
\node at (0.5,0.7) (n3) {3};
\draw [-] (A) to (B);
\draw [-] (B) to (n1);
\draw [-] (B) to (n2);
\draw [-] (B) to (n5);
\draw [-] (A) to (n3);
\draw [-] (A) to (n4);
\end{tikzpicture}\Bigg]$
  \end{tabular}
\end{center}

\item[(B)] Similarly, the contribution of Graph (B) is
    \begin{center}
  \begin{tabular}{c@{\hskip 0.2 cm}c@{\hskip 0.2 cm}c@{\hskip 0.2 cm}c}
    $\frac{1}{2}\cdot\frac{1}{1}\cdot\frac{1}{1} \Bigg[\begin{tikzpicture}[->,baseline=-3pt,node distance=1.3cm,thick,main node/.style={circle,draw,font=\Large,scale=0.5}]
\node at (0,0) (C) {};
\node [scale=.3,draw,circle,fill] [above of =C] (A) {};
\node [scale=.3,draw,circle,fill] [below of =C] (B) {};
\node at (-0.5,-0.7) (n5) {5};
\node at (0.5,-0.7) (n2) {3};
\node at (0.0,-1.0) (n1) {1};
\node at (-0.5,0.7) (n4) {4};
\node at (0.5,0.7) (n3) {2};
\draw [-] (A) to (B);
\draw [-] (B) to (n1);
\draw [-] (B) to (n2);
\draw [-] (B) to (n5);
\draw [-] (A) to (n3);
\draw [-] (A) to (n4);
\end{tikzpicture}\Bigg]$
& 
    $+\frac{1}{2}\cdot\frac{1}{1}\cdot\frac{1}{1} \Bigg[\begin{tikzpicture}[->,baseline=-3pt,node distance=1.3cm,thick,main node/.style={circle,draw,font=\Large,scale=0.5}]
\node at (0,0) (C) {};
\node [scale=.3,draw,circle,fill] [above of =C] (A) {};
\node [scale=.3,draw,circle,fill] [below of =C] (B) {};
\node at (-0.5,-0.7) (n5) {4};
\node at (0.5,-0.7) (n2) {3};
\node at (0.0,-1.0) (n1) {1};
\node at (-0.5,0.7) (n4) {5};
\node at (0.5,0.7) (n3) {2};
\draw [-] (A) to (B);
\draw [-] (B) to (n1);
\draw [-] (B) to (n2);
\draw [-] (B) to (n5);
\draw [-] (A) to (n3);
\draw [-] (A) to (n4);
\end{tikzpicture}\Bigg]$
& 
     $ + \frac{1}{2}\cdot\frac{1}{1}\cdot\frac{1}{1} \Bigg[\begin{tikzpicture}[->,baseline=-3pt,node distance=1.3cm,thick,main node/.style={circle,draw,font=\Large,scale=0.5}]
\node at (0,0) (C) {};
\node [scale=.3,draw,circle,fill] [above of =C] (A) {};
\node [scale=.3,draw,circle,fill] [below of =C] (B) {};
\node at (-0.5,-0.7) (n5) {1};
\node at (0.5,-0.7) (n2) {5};
\node at (0.0,-1.0) (n1) {3};
\node at (-0.5,0.7) (n4) {2};
\node at (0.5,0.7) (n3) {4};
\draw [-] (A) to (B);
\draw [-] (B) to (n1);
\draw [-] (B) to (n2);
\draw [-] (B) to (n5);
\draw [-] (A) to (n3);
\draw [-] (A) to (n4);
\end{tikzpicture}\Bigg]$
&
 $ + \frac{1}{2}\cdot\frac{1}{1}\cdot\frac{1}{1} \Bigg[\begin{tikzpicture}[->,baseline=-3pt,node distance=1.3cm,thick,main node/.style={circle,draw,font=\Large,scale=0.5}]
\node at (0,0) (C) {};
\node [scale=.3,draw,circle,fill] [above of =C] (A) {};
\node [scale=.3,draw,circle,fill] [below of =C] (B) {};
\node at (-0.5,-0.7) (n5) {1};
\node at (0.5,-0.7) (n2) {4};
\node at (0.0,-1.0) (n1) {3};
\node at (-0.5,0.7) (n4) {2};
\node at (0.5,0.7) (n3) {5};
\draw [-] (A) to (B);
\draw [-] (B) to (n1);
\draw [-] (B) to (n2);
\draw [-] (B) to (n5);
\draw [-] (A) to (n3);
\draw [-] (A) to (n4);
\end{tikzpicture}\Bigg]$\\

$= \Bigg[\begin{tikzpicture}[->,baseline=-3pt,node distance=1.3cm,thick,main node/.style={circle,draw,font=\Large,scale=0.5}]
\node at (0,0) (C) {};
\node [scale=.3,draw,circle,fill] [above of =C] (A) {};
\node [scale=.3,draw,circle,fill] [below of =C] (B) {};
\node at (-0.5,-0.7) (n5) {5};
\node at (0.5,-0.7) (n2) {3};
\node at (0.0,-1.0) (n1) {1};
\node at (-0.5,0.7) (n4) {4};
\node at (0.5,0.7) (n3) {2};
\draw [-] (A) to (B);
\draw [-] (B) to (n1);
\draw [-] (B) to (n2);
\draw [-] (B) to (n5);
\draw [-] (A) to (n3);
\draw [-] (A) to (n4);
\end{tikzpicture}\Bigg]$
& 
    +$ \Bigg[\begin{tikzpicture}[->,baseline=-3pt,node distance=1.3cm,thick,main node/.style={circle,draw,font=\Large,scale=0.5}]
\node at (0,0) (C) {};
\node [scale=.3,draw,circle,fill] [above of =C] (A) {};
\node [scale=.3,draw,circle,fill] [below of =C] (B) {};
\node at (-0.5,-0.7) (n5) {4};
\node at (0.5,-0.7) (n2) {3};
\node at (0.0,-1.0) (n1) {1};
\node at (-0.5,0.7) (n4) {5};
\node at (0.5,0.7) (n3) {2};
\draw [-] (A) to (B);
\draw [-] (B) to (n1);
\draw [-] (B) to (n2);
\draw [-] (B) to (n5);
\draw [-] (A) to (n3);
\draw [-] (A) to (n4);
\end{tikzpicture}\Bigg]$
  \end{tabular}
\end{center}

\item[(C)] Graph (C) has also four $\Gamma_0$-structures but they are not symmetric because the two vertical edges are not interchangeable. For the $\Gamma_0$-structures by which the edge of $\Gamma_0$ is mapped to the left edge of the level graph, $\kappa_{f}=3$. If the edge of $\Gamma_0$ is mapped to another edge, then $\kappa_{f}=1$. Thus, the contribution of Graph (C) is
    \begin{center}
  \begin{tabular}{c@{\hskip 0.2 cm}c@{\hskip 0.2 cm}c@{\hskip 0.2 cm}c}
    $\frac{1}{1}\cdot\frac{3}{3}\cdot\frac{3}{3} \Bigg[\begin{tikzpicture}[->,baseline=-3pt,node distance=1.3cm,thick,main node/.style={circle,draw,font=\Large,scale=0.5}]
\node at (0,0) (C) {};
\node [scale=.3,draw,circle,fill] [above of =C] (A) {};
\node [scale=.3,draw,circle,fill] [below of =C] (B) {};
\node at (0.5,-0.7) (n1) {1};
\node at (-0.5,-0.7) (n5) {5};
\node at (0,1.0) (n2) {2};
\node at (-0.5,0.7) (n4) {4};
\node at (0.5,0.7) (n3) {3};
\draw [-] (A) to (B);
\draw [-] (B) to (n1);
\draw [-] (A) to (n2);
\draw [-] (B) to (n5);
\draw [-] (A) to (n3);
\draw [-] (A) to (n4);
\end{tikzpicture}\Bigg]$
& 
    $+\frac{1}{1}\cdot\frac{3}{3}\cdot\frac{3}{3} \Bigg[\begin{tikzpicture}[->,baseline=-3pt,node distance=1.3cm,thick,main node/.style={circle,draw,font=\Large,scale=0.5}]
\node at (0,0) (C) {};
\node [scale=.3,draw,circle,fill] [above of =C] (A) {};
\node [scale=.3,draw,circle,fill] [below of =C] (B) {};
\node at (0.5,-0.7) (n1) {1};
\node at (-0.5,-0.7) (n5) {4};
\node at (0,1.0) (n2) {2};
\node at (-0.5,0.7) (n4) {5};
\node at (0.5,0.7) (n3) {3};
\draw [-] (A) to (B);
\draw [-] (B) to (n1);
\draw [-] (A) to (n2);
\draw [-] (B) to (n5);
\draw [-] (A) to (n3);
\draw [-] (A) to (n4);
\end{tikzpicture}\Bigg]$

&  $+\frac{1}{1}\cdot\frac{3}{1}\cdot\frac{3}{3} \Bigg[\begin{tikzpicture}[->,baseline=-3pt,node distance=1.3cm,thick,main node/.style={circle,draw,font=\Large,scale=0.5}]
\node at (0,0) (C) {};
\node [scale=.3,draw,circle,fill] [above of =C] (A) {};
\node [scale=.3,draw,circle,fill] [below of =C] (B) {};
\node at (0.5,-0.7) (n1) {5};
\node at (-0.5,-0.7) (n5) {1};
\node at (0,1.0) (n2) {3};
\node at (-0.5,0.7) (n4) {2};
\node at (0.5,0.7) (n3) {4};
\draw [-] (A) to (B);
\draw [-] (B) to (n1);
\draw [-] (A) to (n2);
\draw [-] (B) to (n5);
\draw [-] (A) to (n3);
\draw [-] (A) to (n4);
\end{tikzpicture}\Bigg]$

&  $+\frac{1}{1}\cdot\frac{3}{1}\cdot\frac{3}{3} \Bigg[\begin{tikzpicture}[->,baseline=-3pt,node distance=1.3cm,thick,main node/.style={circle,draw,font=\Large,scale=0.5}]
\node at (0,0) (C) {};
\node [scale=.3,draw,circle,fill] [above of =C] (A) {};
\node [scale=.3,draw,circle,fill] [below of =C] (B) {};
\node at (0.5,-0.7) (n1) {4};
\node at (-0.5,-0.7) (n5) {1};
\node at (0,1.0) (n2) {3};
\node at (-0.5,0.7) (n4) {2};
\node at (0.5,0.7) (n3) {5};
\draw [-] (A) to (B);
\draw [-] (B) to (n1);
\draw [-] (A) to (n2);
\draw [-] (B) to (n5);
\draw [-] (A) to (n3);
\draw [-] (A) to (n4);
\end{tikzpicture}\Bigg]$ \\

 $=4 \Bigg[\begin{tikzpicture}[->,baseline=-3pt,node distance=1.3cm,thick,main node/.style={circle,draw,font=\Large,scale=0.5}]
\node at (0,0) (C) {};
\node [scale=.3,draw,circle,fill] [above of =C] (A) {};
\node [scale=.3,draw,circle,fill] [below of =C] (B) {};
\node at (0.5,-0.7) (n1) {1};
\node at (-0.5,-0.7) (n5) {5};
\node at (0,1.0) (n2) {2};
\node at (-0.5,0.7) (n4) {4};
\node at (0.5,0.7) (n3) {3};
\draw [-] (A) to (B);
\draw [-] (B) to (n1);
\draw [-] (A) to (n2);
\draw [-] (B) to (n5);
\draw [-] (A) to (n3);
\draw [-] (A) to (n4);
\end{tikzpicture}\Bigg]$
& 
$+4 \Bigg[\begin{tikzpicture}[->,baseline=-3pt,node distance=1.3cm,thick,main node/.style={circle,draw,font=\Large,scale=0.5}]
\node at (0,0) (C) {};
\node [scale=.3,draw,circle,fill] [above of =C] (A) {};
\node [scale=.3,draw,circle,fill] [below of =C] (B) {};
\node at (0.5,-0.7) (n1) {1};
\node at (-0.5,-0.7) (n5) {4};
\node at (0,1.0) (n2) {2};
\node at (-0.5,0.7) (n4) {5};
\node at (0.5,0.7) (n3) {3};
\draw [-] (A) to (B);
\draw [-] (B) to (n1);
\draw [-] (A) to (n2);
\draw [-] (B) to (n5);
\draw [-] (A) to (n3);
\draw [-] (A) to (n4);
\end{tikzpicture}\Bigg]$

  \end{tabular}
\end{center}

\item[(D)] The contribution of Graph (D) is
    \begin{center}
  \begin{tabular}{c@{\hskip 0.2 cm}c@{\hskip 0.2 cm}c@{\hskip 0.2 cm}c}
    $\frac{1}{2}\cdot\frac{2}{2}\cdot\frac{4}{2} \Bigg[\begin{tikzpicture}[->,baseline=-3pt,node distance=1.3cm,thick,main node/.style={circle,draw,font=\Large,scale=0.5}]
\node at (0,0) (C) {};
\node [scale=.3,draw,circle,fill] [above of =C] (A) {};
\node [scale=.3,draw,circle,fill] [below of =C] (B) {};
\node at (0.5,-0.7) (n1) {1};
\node at (-0.5,-0.7) (n5) {5};
\node at (0,1.0) (n2) {2};
\node at (-0.5,0.7) (n4) {4};
\node at (0.5,0.7) (n3) {3};
\draw [-] (A) to (B);
\draw [-] (B) to (n1);
\draw [-] (A) to (n2);
\draw [-] (B) to (n5);
\draw [-] (A) to (n3);
\draw [-] (A) to (n4);
\end{tikzpicture}\Bigg]$
& 
    $+\frac{1}{2}\cdot\frac{2}{2}\cdot\frac{4}{2} \Bigg[\begin{tikzpicture}[->,baseline=-3pt,node distance=1.3cm,thick,main node/.style={circle,draw,font=\Large,scale=0.5}]
\node at (0,0) (C) {};
\node [scale=.3,draw,circle,fill] [above of =C] (A) {};
\node [scale=.3,draw,circle,fill] [below of =C] (B) {};
\node at (0.5,-0.7) (n1) {1};
\node at (-0.5,-0.7) (n5) {4};
\node at (0,1.0) (n2) {2};
\node at (-0.5,0.7) (n4) {5};
\node at (0.5,0.7) (n3) {3};
\draw [-] (A) to (B);
\draw [-] (B) to (n1);
\draw [-] (A) to (n2);
\draw [-] (B) to (n5);
\draw [-] (A) to (n3);
\draw [-] (A) to (n4);
\end{tikzpicture}\Bigg]$

&  $+\frac{1}{2}\cdot\frac{2}{2}\cdot\frac{4}{2} \Bigg[\begin{tikzpicture}[->,baseline=-3pt,node distance=1.3cm,thick,main node/.style={circle,draw,font=\Large,scale=0.5}]
\node at (0,0) (C) {};
\node [scale=.3,draw,circle,fill] [above of =C] (A) {};
\node [scale=.3,draw,circle,fill] [below of =C] (B) {};
\node at (0.5,-0.7) (n1) {5};
\node at (-0.5,-0.7) (n5) {1};
\node at (0,1.0) (n2) {3};
\node at (-0.5,0.7) (n4) {2};
\node at (0.5,0.7) (n3) {4};
\draw [-] (A) to (B);
\draw [-] (B) to (n1);
\draw [-] (A) to (n2);
\draw [-] (B) to (n5);
\draw [-] (A) to (n3);
\draw [-] (A) to (n4);
\end{tikzpicture}\Bigg]$

&  $+\frac{1}{2}\cdot\frac{2}{2}\cdot\frac{4}{2} \Bigg[\begin{tikzpicture}[->,baseline=-3pt,node distance=1.3cm,thick,main node/.style={circle,draw,font=\Large,scale=0.5}]
\node at (0,0) (C) {};
\node [scale=.3,draw,circle,fill] [above of =C] (A) {};
\node [scale=.3,draw,circle,fill] [below of =C] (B) {};
\node at (0.5,-0.7) (n1) {4};
\node at (-0.5,-0.7) (n5) {1};
\node at (0,1.0) (n2) {3};
\node at (-0.5,0.7) (n4) {2};
\node at (0.5,0.7) (n3) {5};
\draw [-] (A) to (B);
\draw [-] (B) to (n1);
\draw [-] (A) to (n2);
\draw [-] (B) to (n5);
\draw [-] (A) to (n3);
\draw [-] (A) to (n4);
\end{tikzpicture}\Bigg]$ \\

 $=2 \Bigg[\begin{tikzpicture}[->,baseline=-3pt,node distance=1.3cm,thick,main node/.style={circle,draw,font=\Large,scale=0.5}]
\node at (0,0) (C) {};
\node [scale=.3,draw,circle,fill] [above of =C] (A) {};
\node [scale=.3,draw,circle,fill] [below of =C] (B) {};
\node at (0.5,-0.7) (n1) {1};
\node at (-0.5,-0.7) (n5) {5};
\node at (0,1.0) (n2) {2};
\node at (-0.5,0.7) (n4) {4};
\node at (0.5,0.7) (n3) {3};
\draw [-] (A) to (B);
\draw [-] (B) to (n1);
\draw [-] (A) to (n2);
\draw [-] (B) to (n5);
\draw [-] (A) to (n3);
\draw [-] (A) to (n4);
\end{tikzpicture}\Bigg]$
& 
$+2 \Bigg[\begin{tikzpicture}[->,baseline=-3pt,node distance=1.3cm,thick,main node/.style={circle,draw,font=\Large,scale=0.5}]
\node at (0,0) (C) {};
\node [scale=.3,draw,circle,fill] [above of =C] (A) {};
\node [scale=.3,draw,circle,fill] [below of =C] (B) {};
\node at (0.5,-0.7) (n1) {1};
\node at (-0.5,-0.7) (n5) {4};
\node at (0,1.0) (n2) {2};
\node at (-0.5,0.7) (n4) {5};
\node at (0.5,0.7) (n3) {3};
\draw [-] (A) to (B);
\draw [-] (B) to (n1);
\draw [-] (A) to (n2);
\draw [-] (B) to (n5);
\draw [-] (A) to (n3);
\draw [-] (A) to (n4);
\end{tikzpicture}\Bigg]$

  \end{tabular}
\end{center}

\item[(E)] The contribution of the graph (E) will be just the class $[\overline{\mathcal{H}}_0^\mathfrak{R}(4,-2,-2,-1,-1) ]$ where $\mathfrak{R}$ is induced by the residue condition $r_4+r_5=0$.
\end{itemize}
To check whether we are correct, we can use the method \texttt{totensorTautbasis} in the \texttt{Sage} package \texttt{admcycles} (the commands cf. \cite{delecroix2020admcycles}). We first compute the clutching pullback:

\begin{lstlisting}
sage: from admcycles import *
sage: stratum_class = Strataclass(1,1,(4,-2,-2))
sage: stgraph = StableGraph([0],[[1,2,3,4,5]],[(4,5)])
sage: pull = stgraph.boundary_pullback(stratum_class)
sage: pull.totensorTautbasis(1,vecout=True)
(10, 5, -7, -7, 0)
\end{lstlisting}
On the other hand, we compute the result we get from the clutching pullback formula:

\begin{lstlisting}
sage: A_contr=tautgens(0,5,1)[8]+tautgens(0,5,1)[7]
sage: B_contr=tautgens(0,5,1)[11]+tautgens(0,5,1)[10]
sage: C_contr=4*tautgens(0,5,1)[15]+4*tautgens(0,5,1)[14]
sage: D_contr=2*tautgens(0,5,1)[15]+2*tautgens(0,5,1)[14]
sage: E_contr= Strataclass(0,1,(4,-2,-2,-1,-1), res_cond=((0,0,0,1,1),))
sage: (A_contr+B_contr+C_contr+D_contr+E_contr).basis_vector()
(10, 5, -7, -7, 0)
\end{lstlisting}

Now we will do the clutching pullback with respect to other one-edge stable graphs. There is only one vertical two-level graph (or one-edge horizontal graph) which is degenerations to $\Gamma_1^{\{1\}}$, namely

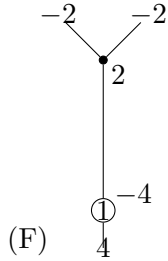
\begin{center}
    \begin{tikzpicture}
    
        \node (a) at (0,0) [circle,draw,fill,inner sep=0pt,minimum size=3pt]{};
    \node (b) at (0,-2) [circle,draw,inner sep=0pt,minimum size=5pt]{$1$};
    
    \path [-] (a) edge (b);
   \node at (0.2,-0.2){$2$} ;
    \node at (-0.6,0.6){$-2$} ;
    \node at (0.6,0.6){$-2$} ;
    \node at (0.4,-1.8){$-4$} ;
    \path [-] (a) edge (-0.5,0.5);
    \path [-] (a) edge (0.5,0.5);
    \path [-] (b) edge (0,-2.5);
    \node at (0,-2.5){$4$};
    \node at (-1,-2.4){(F)};
    \end{tikzpicture}
\end{center}
Hence, $\xi_{\Gamma_1^{\{1\}}}^*([\overline{\mathcal{H}}_1(4,-2,-2)])\in CH^1(\overline{\mathcal{M}}_{\Gamma_1^{\{1\}}})$ will be equal to the following
\begin{center}
 
    $\frac{1}{1}\cdot\frac{3}{3}\cdot\frac{3}{3}\bigg(\pi_\top^*([\overline{\mathcal{M}}_{0,3}])\cdot \pi_\bot^*([\overline{\mathcal{H}}_1(4,-4)])\bigg)= 15\pi_\bot^*\psi_1
$

\end{center}
We have also the following vertical two-level graphs:
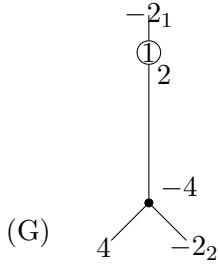
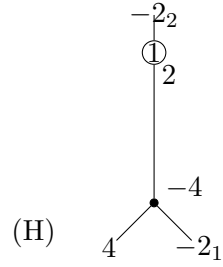
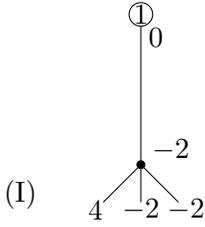
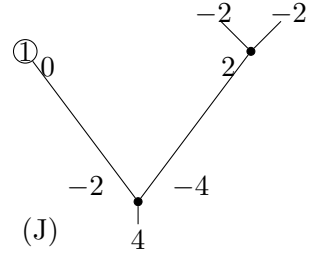
\begin{center}
\begin{tikzpicture}

    \node (a) at (0,0) [circle,draw,inner sep=0pt,minimum size=5pt]{$1$};
    \node (b) at (0,-2) [circle,draw,fill,inner sep=0pt,minimum size=3pt]{};
    
    \path [-] (a) edge (b);
   \node at (0.2,-0.3){$2$} ;
    \node at (-0.6,-2.6){$4$} ;
    \node at (0.6,-2.6){$-2_2$} ;
    \node at (0.4,-1.8){$-4$} ;
    \path [-] (b) edge (-0.5,-2.5);
    \path [-] (b) edge (0.5,-2.5);
    \path [-] (a) edge (0,0.5);
    \node at (0,0.5){$-2_1$};
    \node at (-1.6,-2.4){(G)};

    \begin{scope}[xshift=10cm]
        \node (a) at (0,0) [circle,draw,inner sep=0pt,minimum size=5pt]{$1$};
    \node (b) at (0,-2) [circle,draw,fill,inner sep=0pt,minimum size=3pt]{};
    
    \path [-] (a) edge (b);
   \node at (0.2,-0.3){$2$} ;
    \node at (-0.6,-2.6){$4$} ;
    \node at (0.6,-2.6){$-2_1$} ;
    \node at (0.4,-1.8){$-4$} ;
    \path [-] (b) edge (-0.5,-2.5);
    \path [-] (b) edge (0.5,-2.5);
    \path [-] (a) edge (0,0.5);
    \node at (0,0.5){$-2_2$};
    \node at (-1.6,-2.4){(H)}; 
    \end{scope}
    \begin{scope}[xshift=0cm,yshift=-5cm]
       \node (a) at (0,0) [circle,draw,inner sep=0pt,minimum size=5pt]{$1$};
    \node (b) at (0,-2) [circle,draw,fill,inner sep=0pt,minimum size=3pt]{};
    
    \path [-] (a) edge (b);
   \node at (0.2,-0.3){$0$} ;
    \node at (-0.6,-2.6){$4$} ;
    \node at (0.6,-2.6){$-2$} ;
    \node at (0.4,-1.8){$-2$} ;
    \path [-] (b) edge (-0.5,-2.5);
    \path [-] (b) edge (0.5,-2.5);
    \path [-] (b) edge (0,-2.5);
    \node at (0,-2.6){$-2$};
    \node at (-1.6,-2.4){(I)}; 
    \end{scope}
    
      \begin{scope}[xshift=10cm, yshift=-5cm]
        \node (a) at (-1.5,0) [circle,draw,inner sep=0pt,minimum size=5pt]{$1$};
        \node (b) at (1.5,0) [circle,draw,fill, inner sep=0pt,minimum size=3pt]{};
        \node (c) at (0,-2) [circle,draw,fill, inner sep=0pt,minimum size=3pt]{};
        \node at (0,-2.5){$4$};
        \path [-] (b) edge (1.1,0.4);
        \path [-] (b) edge (1.9,0.4);
        \node at (1,0.5){$-2$} ;
        \node at (2,0.5){$-2$} ;
        \path [-] (a) edge (c);
        \path [-] (b) edge (c);
        \node at (0.7,-1.8){$-4$} ;
        \node at (1.2,-0.2){$2$} ;
        \node at (-1.2,-0.2){$0$} ;
        \node at (-0.7,-1.8){$-2$} ;
        \node at (-1.3,-2.4){(J)};  
        \path [-] (c) edge (0,-2.3);
    \end{scope}
    \end{tikzpicture}
\end{center}
Graph (G) is a degeneration of $\Gamma_1^{\{2\}}$, while graph (H) is a degeneration of $\Gamma_1^{\{3\}}$. Similarly, we have
\begin{align*}
    \xi_{\Gamma_1^{\{2\}}}^*([\overline{\mathcal{H}}_1(4,-2,-2)])= \frac{1}{1}\cdot\frac{3}{3}\cdot\frac{3}{3}\bigg(\pi_\top^*([\overline{\mathcal{H}}_1(2,-2)])\cdot\pi_\bot^*([\overline{\mathcal{M}}_{0,3}]) \bigg)= 3\pi_\top^*\psi_1\\
    \xi_{\Gamma_1^{\{3\}}}^*([\overline{\mathcal{H}}_1(4,-2,-2)])= \frac{1}{1}\cdot\frac{3}{3}\cdot\frac{3}{3}\bigg(\pi_\top^*([\overline{\mathcal{H}}_1(2,-2)])\cdot\pi_\bot^*([\overline{\mathcal{M}}_{0,3}]) \bigg)= 3\pi_\top^*\psi_1\
\end{align*}
Graph (I) and (J) are degenerations of $\Gamma_1^\emptyset$. However, note that the top level of graph (J) has more than one vertices without any residue condition relating the legs of them. Thus, the projection pushforward of graph (J) is zero. Then we have
\[\xi_{\Gamma_1^\emptyset}^*([\overline{\mathcal{H}}_1(4,-2,-2)])= \frac{1}{1}\cdot\frac{1}{1}\cdot\frac{1}{1}\bigg(\pi_\top^*([\overline{\mathcal{M}}_{1,1})\cdot\pi_\bot^*([\overline{\mathcal{H}}_0^\mathfrak{R}(4,-2,-2,-2)]) \bigg)= \pi_\bot^*\psi_4, \]

where $\mathfrak{R}$ represent the residue condition of $r_4=0$.

\subsection{Examples of computing the spin stratum class by clutching pullbacks}\label{sec:exa_spin_pull}

In this subsection, we will present examples of computing the spin stratum class (with and without residue conditions). In the following example, we will consider the stratum $\mathcal{H}_{1}(4,-2,-2)]$ which has been investigated in Section~\ref{sec:exa_clutch}.

\begin{exa}\label{exa:cal_spin}
The two-level graphs of $\mathbb{P}\Xi \overline{\mathcal{M}}_{1,3}(4,-2,-2)$, which are of compact type, are

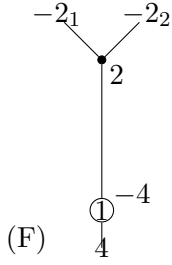
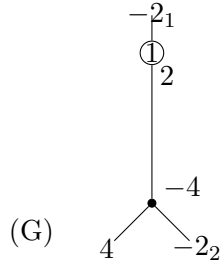
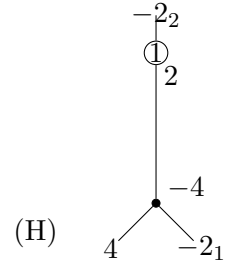
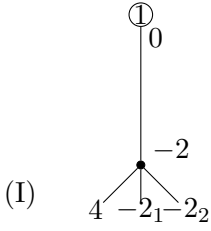
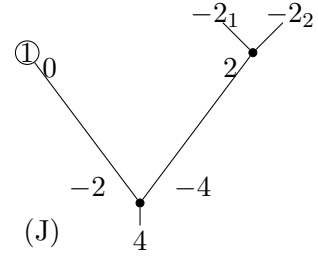
\begin{center}
    \begin{tikzpicture}
    \node (a) at (0,0) [circle,draw,fill,inner sep=0pt,minimum size=3pt]{};
    \node (b) at (0,-2) [circle,draw,inner sep=0pt,minimum size=5pt]{$1$};
    \path [-] (a) edge (b);
   \node at (0.2,-0.2){$2$} ;
    \node at (-0.6,0.6){$-2_1$} ;
    \node at (0.6,0.6){$-2_2$} ;
    \node at (0.4,-1.8){$-4$} ;
    \path [-] (a) edge (-0.5,0.5);
    \path [-] (a) edge (0.5,0.5);
    \path [-] (b) edge (0,-2.5);
    \node at (0,-2.5){$4$};
    \node at (-1,-2.4){(F)};
    
  \begin{scope}[xshift=5cm]
    \node (a) at (0,0) [circle,draw,inner sep=0pt,minimum size=5pt]{$1$};
    \node (b) at (0,-2) [circle,draw,fill,inner sep=0pt,minimum size=3pt]{};
    \path [-] (a) edge (b);
    \node at (0.2,-0.3){$2$} ;
    \node at (-0.6,-2.6){$4$} ;
    \node at (0.6,-2.6){$-2_2$} ;
    \node at (0.4,-1.8){$-4$} ;
    \path [-] (b) edge (-0.5,-2.5);
    \path [-] (b) edge (0.5,-2.5);
    \path [-] (a) edge (0,0.5);
    \node at (0,0.5){$-2_1$};
    \node at (-1.6,-2.4){(G)};   
\end{scope}  
    
    \begin{scope}[xshift=10cm]
    \node (a) at (0,0) [circle,draw,inner sep=0pt,minimum size=5pt]{$1$};
    \node (b) at (0,-2) [circle,draw,fill,inner sep=0pt,minimum size=3pt]{};
    \path [-] (a) edge (b);
    \node at (0.2,-0.3){$2$} ;
    \node at (-0.6,-2.6){$4$} ;
    \node at (0.6,-2.6){$-2_1$} ;
    \node at (0.4,-1.8){$-4$} ;
    \path [-] (b) edge (-0.5,-2.5);
    \path [-] (b) edge (0.5,-2.5);
    \path [-] (a) edge (0,0.5);
    \node at (0,0.5){$-2_2$};
    \node at (-1.6,-2.4){(H)}; 
    \end{scope}
    
    \begin{scope}[xshift=0cm,yshift=-5cm]
    \node (a) at (0,0) [circle,draw,inner sep=0pt,minimum size=5pt]{$1$};
    \node (b) at (0,-2) [circle,draw,fill,inner sep=0pt,minimum size=3pt]{};
    \path [-] (a) edge (b);
    \node at (0.2,-0.3){$0$} ;
    \node at (-0.6,-2.6){$4$} ;
    \node at (0.6,-2.6){$-2_2$} ;
    \node at (0.4,-1.8){$-2$} ;
    \path [-] (b) edge (-0.5,-2.5);
    \path [-] (b) edge (0.5,-2.5);
    \path [-] (b) edge (0,-2.5);
    \node at (0,-2.6){$-2_1$};
    \node at (-1.6,-2.4){(I)}; 
    \end{scope}
    
    \begin{scope}[xshift=10cm, yshift=-5cm]
    \node (a) at (-1.5,0) [circle,draw,inner sep=0pt,minimum size=5pt]{$1$};
    \node (b) at (1.5,0) [circle,draw,fill, inner sep=0pt,minimum size=3pt]{};
    \node (c) at (0,-2) [circle,draw,fill, inner sep=0pt,minimum size=3pt]{};
    \node at (0,-2.5){$4$};
    \path [-] (b) edge (1.1,0.4);
    \path [-] (b) edge (1.9,0.4);
    \node at (1,0.5){$-2_1$} ;
    \node at (2,0.5){$-2_2$} ;
    \path [-] (a) edge (c);
    \path [-] (b) edge (c);
    \node at (0.7,-1.8){$-4$} ;
    \node at (1.2,-0.2){$2$} ;
    \node at (-1.2,-0.2){$0$} ;
    \node at (-0.7,-1.8){$-2$} ;
    \node at (-1.3,-2.4){(J)};  
    \path [-] (c) edge (0,-2.3);
    \end{scope}
    \end{tikzpicture}
\end{center}

 Notice that the top level of the graph $(J)$ has two components which can be scaled independently. Thus $p_{J*}[\mathbb{P}\Xi\overline{\mathcal{M}}_J]^{\spin}=0$ as the image will be of lower dimension. Hence it will not affect the spin stratum class. 
 
 The pushforward of the spin variant of the divisor class of the boundary stratum, which is associated to the level graph (F),  on the moduli space of disconnected stable curves $\overline{\mathcal{M}}_F$ :
 \begin{align*}
      p_{F*}[\mathbb{P}\Xi\overline{\mathcal{M}}_F]^{\spin}=\pi_{\top}^*[\overline{\mathcal{H}}(2,-2,-2)]^{\spin}\cdot \pi_{\bot}^*[\overline{\mathcal{H}}_{1,2}(4,-4)]^{\spin}.
 \end{align*}
    
As $\mu=(2,-2,-2)$ is a signature of a stratum of differentials on genus $0$ curves. By default,
 \begin{align*}
    [\overline{\mathcal{H}}(2,-2,-2)]^{\spin}=&[\overline{\mathcal{H}}(2,-2,-2)]\\
    =&[\overline{\mathcal{M}}_{0,3}]
\end{align*}

On the other hand, one can show that
\begin{align*}
    [\overline{\mathcal{H}}_{1,2}(4,-4)]^{\spin}=9\psi_1
\end{align*}
Hence, 
\[p_{F*}[\mathbb{P}\Xi\overline{\mathcal{M}}_A]^{\spin}=9\pi_\bot^*\psi_1\]

Now we can do the computations of the clutching pullbacks:
\begin{itemize}
    \item Only the enhanced level graphs (F) and (J) have $\Gamma_1^{\{1\}}$-structures. Thus, by the excess intersection formula, we conclude that \begin{align*}
   \xi_{\Gamma_1^{\{1\}}}^*[\overline{\mathcal{H}}_1(4,-2,-2)]^{\spin}=9\pi_\bot^*\psi_1. 
    \end{align*}
    Only the enhanced level graph (G) has $\Gamma_1^{\{2\}}$-structure. One can compute that $[\overline{\mathcal{H}}_1(2,-2)]^{\spin}= 3\psi_1 $. Thus, we have: 
    \[\xi_{\Gamma_1^{\{2\}}}^*[\overline{\mathcal{H}}_1(4,-2,-2)]^{\spin}=9\pi_\top^*\psi_1\] 
    \item Only the enhanced level graph (H) has $\Gamma_1^{\{3\}}$-structure, thus we have:
    \[\xi_{\Gamma_1^{\{3\}}}^*[\overline{\mathcal{H}}_1(4,-2,-2)]^{\spin}=3\pi_\top^*\psi_1\] 
    \item Only the enhanced level graph (I) has $\Gamma_1^\emptyset$-structure, hence:
    \[\xi_{\Gamma_1^\emptyset}^*[\overline{\mathcal{H}}_1(4,-2,-2)]^{\spin}=-\pi_\bot^*\psi_1\] 
\end{itemize}

By using the function \texttt{totensorTautbasis()} in \texttt{admcycles}, we can convert the product tautological classes above into vectors. Now we can make use of the basis of $RH^2(\overline{\mathcal{M}}_{1,3})$ given by \texttt{admcycles}:
\begin{align*}
    \RH^2(\overline{\mathcal{M}}_{1,3})=\langle \kappa_1,\psi_1,\psi_2,\psi_3, \delta^{\{1\}}_1 \rangle
\end{align*} 
By solving the collection of linear equations, we get 
\begin{align*}
   [\overline{\mathcal{H}}_1(4,-2,-2)]^{\spin}=&(0,1,3,3,-8)&= \psi_1+3(\psi_2+\psi_3)- 8\delta^{\{1\}}_1 
\end{align*}

\end{exa}

In the next example, we will illustrate how we resolve a residue condition.

\begin{exa}
Now we consider the spin stratum classes $[\overline{\mathcal{H}}_{1}^\mathfrak{R}(4,-2,-2)]^{\spin}$, where $\mathfrak{R}$ represents the residue condition $r_3=0$. We will first apply Proposition~\ref{prop:res_resl} to resolve the residue condition. The level graphs in Section~\ref{sec:exa_clutch} such that the residue condition $r_3=0$ induces no extra condition on the top level are the following:
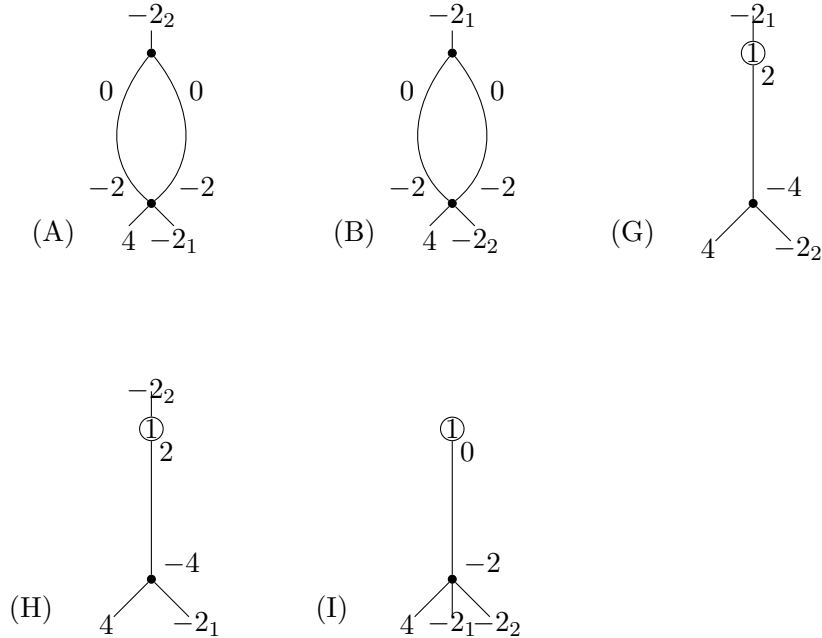
\begin{center}
    \begin{tikzpicture}
      \begin{scope}[xshift=0cm,yshift=0cm]
        \node (a) at (0,0) [circle,draw,inner sep=0pt,fill,minimum size=3pt]{};
        \node (b) at (0,-2) [circle,draw,fill,inner sep=0pt,minimum size=3pt]{};
        \draw (a).. controls (-0.6,-0.75) and (-0.6,-1.5).. (b);
         \draw (a).. controls (0.6,-0.75) and (0.6,-1.5).. (b);
        \path [-] (b) edge (-0.3,-2.3);
        \path [-] (b) edge (0.3,-2.3);
        \path[-] (a) edge (0,0.3);
        \node at (-1.3,-2.4){(A)};  \node at (-0.3,-2.5){$4$};
        \node at (0.3,-2.5){$-2_1$}; \node at (0,0.5){$-2_2$};
        \node at (-0.6,-0.5){$0$};
        \node at (0.6,-0.5){$0$};
        \node at (-0.6,-1.8){$-2$};
        \node at (0.6,-1.8){$-2$};
    \end{scope}
    \begin{scope}[xshift=4cm,yshift=0cm]
        \node (a) at (0,0) [circle,draw,inner sep=0pt,fill,minimum size=3pt]{};
        \node (b) at (0,-2) [circle,draw,fill,inner sep=0pt,minimum size=3pt]{};
        \draw (a).. controls (-0.6,-0.75) and (-0.6,-1.5).. (b);
         \draw (a).. controls (0.6,-0.75) and (0.6,-1.5).. (b);
        \path [-] (b) edge (-0.3,-2.3);
        \path [-] (b) edge (0.3,-2.3);
        \path[-] (a) edge (0,0.3);
        \node at (-1.3,-2.4){(B)};  \node at (-0.3,-2.5){$4$};
        \node at (0.3,-2.5){$-2_2$}; \node at (0,0.5){$-2_1$};
        \node at (-0.6,-0.5){$0$};
        \node at (0.6,-0.5){$0$};
        \node at (-0.6,-1.8){$-2$};
        \node at (0.6,-1.8){$-2$};
    \end{scope}
    
    \begin{scope}[xshift=8cm]
    \node (a) at (0,0) [circle,draw,inner sep=0pt,minimum size=5pt]{$1$};
    \node (b) at (0,-2) [circle,draw,fill,inner sep=0pt,minimum size=3pt]{};
    \path [-] (a) edge (b);
    \node at (0.2,-0.3){$2$} ;
    \node at (-0.6,-2.6){$4$} ;
    \node at (0.6,-2.6){$-2_2$} ;
    \node at (0.4,-1.8){$-4$} ;
    \path [-] (b) edge (-0.5,-2.5);
    \path [-] (b) edge (0.5,-2.5);
    \path [-] (a) edge (0,0.5);
    \node at (0,0.5){$-2_1$};
    \node at (-1.6,-2.4){(G)};   
    \end{scope}  
    
    \begin{scope}[xshift=0cm, yshift=-5cm ]
        \node (a) at (0,0) [circle,draw,inner sep=0pt,minimum size=5pt]{$1$};
    \node (b) at (0,-2) [circle,draw,fill,inner sep=0pt,minimum size=3pt]{};
    
    \path [-] (a) edge (b);
   \node at (0.2,-0.3){$2$} ;
    \node at (-0.6,-2.6){$4$} ;
    \node at (0.6,-2.6){$-2_1$} ;
    \node at (0.4,-1.8){$-4$} ;
    \path [-] (b) edge (-0.5,-2.5);
    \path [-] (b) edge (0.5,-2.5);
    \path [-] (a) edge (0,0.5);
    \node at (0,0.5){$-2_2$};
    \node at (-1.6,-2.4){(H)}; 
    \end{scope}
    \begin{scope}[xshift=4cm,yshift=-5cm]
       \node (a) at (0,0) [circle,draw,inner sep=0pt,minimum size=5pt]{$1$};
    \node (b) at (0,-2) [circle,draw,fill,inner sep=0pt,minimum size=3pt]{};
    
    \path [-] (a) edge (b);
   \node at (0.2,-0.3){$0$} ;
    \node at (-0.6,-2.6){$4$} ;
    \node at (0.6,-2.6){$-2_2$} ;
    \node at (0.4,-1.8){$-2$} ;
    \path [-] (b) edge (-0.5,-2.5);
    \path [-] (b) edge (0.5,-2.5);
    \path [-] (b) edge (0,-2.5);
    \node at (0,-2.6){$-2_1$};
    \node at (-1.6,-2.4){(I)}; 
    \end{scope}

    \end{tikzpicture}
\end{center}
Notice that only the graph (I) has odd spin (due to the top level). The other graphs are all of even spin.  Thus, we have

\begin{align*}
   [\mathbb{P}\Xi\overline{\mathcal{M}}_{1,3}^\mathfrak{R}(4,-2,-2)]^{\spin}=-\xi^{\spin}  -[D_A]-[D_B]-3[D_G]-3[D_H] +[D_I]. 
\end{align*}

By Proposition~\ref{prop:xi}, if we take the reference leg to be the third marked point, then we have
\begin{align*}
    \xi^{\spin}&=(-2+1)\psi_3^{\spin} -[D_B]- 3[D_G] + [D_I]
\end{align*}
Then, 
\begin{align*}
  [\mathbb{P}\Xi\overline{\mathcal{M}}_{1,3}^\mathfrak{R}(4,-2,-2)]^{\spin}& = \psi_3^{\spin}-[D_A]-3[D_H]
\end{align*}
By the result in Example~\ref{exa:cal_spin}, we then have 

\begin{align*}
    [\overline{\mathcal{H}}^\mathfrak{R}_{1}(4,-2,-2)]^{\spin}=&p_*[\mathbb{P}\Xi\overline{\mathcal{M}}_{1,3}^\mathfrak{R}(4,-2,-2)]^{\spin}\\=&p_*(\psi_3^{\spin}-[D_A]-3[D_H])\\=&\psi_3\cdot [\overline{\mathcal{H}}_{1}(4,-2,-2)]^{\spin}-p_*[D_A]-3p_*[D_H]\\=& \bigg(\psi_3(\psi_1+3\psi_2+3\psi_3)-8\psi_3\delta^{\{1\}}_1\bigg) - \Bigg[\begin{tikzpicture}[->,baseline=-3pt,node distance=1.3cm,thick,main node/.style={circle,draw,font=\Large,scale=0.5}]
\node at (0,0) (C) {};
\node [scale=.3,draw,circle,fill] [above of =C] (A) {};
\node [scale=.3,draw,circle,fill] [below of =C] (B) {};
\node at (0.5,-0.7) (n2) {2};
\node at (-0.5,-0.7) (n1) {1};
\node at (0,1.0) (n3) {3};
\draw [-] (A) to [out=-120, in=120] (B);
\draw [-] (A) to [out=-60, in=60] (B);
\draw [-] (B) to (n2);
\draw [-] (B) to (n1);
\draw [-] (A) to (n3);
\end{tikzpicture}\Bigg]
- 3(3\psi_3)\cdot\Bigg[\begin{tikzpicture}[->,baseline=-3pt,node distance=1.3cm,thick,main node/.style={circle,draw,font=\Large,scale=0.5}]
\node at (0,0) (C) {};
\node [scale=.3,draw,circle] [above of =C] (A) {1};
\node [scale=.3,draw,circle,fill] [below of =C] (B) {};
\node at (0.5,-0.7) (n2) {2};
\node at (-0.5,-0.7) (n1) {1};
\node at (0,1.0) (n3) {3};
\draw [-] (A) to (B);
\draw [-] (B) to (n2);
\draw [-] (B) to (n1);
\draw [-] (A) to (n3);
\end{tikzpicture}\Bigg].
\end{align*} 

\end{exa}

\section{Computation Results of $[\overline{\mathcal{H}}_g(\mu)]^{\spin}\in \RH^*(\overline{\mathcal{M}}_{g,n})$}\label{appendix:results}

In the following, we will list some of our results of the spin classes for $(g,n)$ equal to
\begin{align*}
 (2,1),(2,2), (2,3), (2,4), (3,1), (3,2), (3,3), (4,1), (4,2), (4,3)   
\end{align*}
This also implies that Assumption~\ref{prop:assum} holds during the computation for the spin stratum classes of strata compatible to the values of $(g,n)$ above. Our results will be written in vector form, whereas the basis is computed by \texttt{admcycles}. For $(g,n)= (3,3), (4,2)$, the expression of spin stratum classes of meromorphic strata will be super lengthy, so we will not list them here. 

For $(g,n)= (2,1)$, by our recursion, we get
\begin{lstlisting}
Spin_Class(2,) = (1/2, -7/2, 1/2)
\end{lstlisting}

For $(g,n)= (2,2)$,

\begin{lstlisting}
Spin_Class(-6, 8) = (-1563/2, 521/2, -637, -602, 1274, 1021, 1204, 1563/2, -3647/2, 521/2, -116, -81, -521/2, 0)

Spin_Class(-4, 6) = (-423/2, 141/2, -174, -159, 348, 273, 318, 423/2, -987/2, 141/2, -33, -18, -141/2, 0)

Spin_Class(-2, 4) = (-63/2, 21/2, -27, -22, 54, 41, 44, 63/2, -147/2, 21/2, -6, -1, -21/2, 0)
\end{lstlisting}

The spin stratum classes we listed above are expressed in the basis of $\RH^4(\overline{\mathcal{M}}_{2,2},\mathbb{Q})$, which consists of:
\begin{align*}
    \kappa_2, \kappa_1^2, \kappa_1\psi_1,\kappa_1\psi_2,\psi_1\psi_2, \psi_1^2,\psi_2^2,...
\end{align*}
\begin{rem}
The first entry of the basis of $\RH^{2j}(\overline{\mathcal{M}}_{g,n},\mathbb{Q})$ constructed by \texttt{admcycles} is by default just $\kappa_j$.
\end{rem}
From now on, the vector of the spin stratum class will be quite lengthy. Thus we will only give the first few entries. For $(g,n)= (2,3)$,

\begin{lstlisting}
Spin_Class(-1, -1, 4) = (0, 0, 4, 0, ...)

Spin_Class(-6, -6, 14) = (-10623/2, 3541/2, -4067, -5329, ...)

Spin_Class(-6, -2, 10) = (-3039/2, 1013/2, -995, -1517, ...)

Spin_Class(-6, 4, 4) = (-1023/2, 341/2, -167, -514, ...)

Spin_Class(-4, 2, 4) = (-162, 54, -58, -162, ...)
\end{lstlisting}

For $(g,n)=(2,4)$,

\begin{lstlisting}
Spin_Class(-6, -6, -2, 16) = (-6208, 3129, -7270, -8261, ...)

Spin_Class(-6, -4, 2, 10) = (-5777/2, 1053/2, -2027/2, -2441/2, ...)
\end{lstlisting}

For $(g,n)=(3,1)$,

\begin{lstlisting}
Spin_Class(4,) = (-729, 1099/12, -787/12, 213, ...)
\end{lstlisting}

For $(g,n)=(3,2)$,

\begin{lstlisting}
Spin_Class(-4, 8) = (-34167333/140, 14160639/280, -976627/280, -358429/40, ...)

Spin_Class(2, 2) = (159/4, -179/48, -7/24, -7/24, ...)
\end{lstlisting}

For $(g,n)=(4,1)$,

\begin{lstlisting}
Spin_Class(6,) = (-1125203/120, 85013/48, -325/3, -65089/24, ...)

\end{lstlisting}

For $(g,n)=(4,2)$,

\begin{lstlisting}
Spin_Class(2, 4) = (59623999/840, -156407/16, 302305/672, -229619/112, ...)
\end{lstlisting}

for $(g,n)=(4,3)$,
\begin{lstlisting}
Spin_Class(2,2,2) =  (-1572061/48, 42127/9, -64043/288, -116291/72, ...)
\end{lstlisting}

One can input our result in \texttt{admcycles} by the method \texttt{Tautvb\_to\_tautclass}. The following is the example for $\mu=(2,2)$,
\begin{lstlisting}
sage: from admcycles import *
sage: Spin_Class_2_2_vb = (159/4, -179/48, -7/24, -7/24, -131/48, -23/24, -131/48, 149/24, -83/16, 271/48, 77/48, 77/48, -193/8, -185/48, 127/48, 395/48, 221/48, -73/8, -185/48, 127/48, 395/48, 221/48, -73/8, -185/48, -41/48, 389/48, 389/48, -23/8, 51/8, -139/16, -323/16, 1/8, -25/4, -11/4, -11/4, -23/4, 973/48, 37/96, -5/2, 23/96, 23/96, -25/32)
sage: Spin_Class_2_2 = Tautvb_to_tautclass(Spin_Class_2_2_vb, g=3, n=2, r=2)
\end{lstlisting}
Here $r$ is the codimension of our class.

\printbibliography
\end{document}